\documentclass[10pt]{article}
\usepackage{amsbsy}
\usepackage{amsmath,amsfonts,amssymb,amsthm,mathrsfs,mathtools}
\usepackage[a4paper,vmargin={3.5cm,3.5cm},hmargin={2.5cm,2.5cm}]{geometry}
\usepackage[font=sf, labelfont={sf,bf}, margin=1cm]{caption}
\usepackage{graphicx,graphics}
\usepackage{epsfig}
\usepackage{latexsym}
\usepackage[applemac]{inputenc}
\usepackage{ae,aecompl}
\usepackage{soul}
\usepackage[english]{babel}
 \usepackage[pdfpagemode=UseNone,bookmarksopen=false,colorlinks=true]{hyperref}
\usepackage{pstricks}
\usepackage{enumerate}

\usepackage[normalem]{ulem}

\usepackage{bbm}

\newcommand{\one}{{\triangle^{1}}}
\newcommand{\tone}{\mathsf{t}_{\triangle^{1}}}
\newcommand{\aone}{\mathsf{a}_{\triangle^{1}}}

\newcommand{\Z}{\mathbb{Z}}
\renewcommand{\P}{\mathbb{P}}
\newcommand{\E}{\mathbb{E}}
\renewcommand{\t}{\mathbf{t}}
\def\U{\mathbb{U}}
\def\N{\mathbb{N}}
\def\e{{\rm e}}
\def\d{{\rm d}}
\newcommand{\R}{\mathbb{R}}
\newcommand\Es[1]{\mathbb{E}\left[#1\right]}
\renewcommand\Pr[1]{\mathbb{P}\left(#1\right)}
\newcommand \fl[1] {\left\lfloor #1 \right\rfloor}

\newcommand\Esp[1]{\mathbb{E}^{(p)}\left[#1\right]}
\newcommand\Prp[1]{\mathbb{P}^{(p)}\left(#1\right)}
\newcommand\Espb[1]{\mathbb{E}^{(p)}[#1]}

\newtheorem{theorem}{Theorem}[]

\newtheorem{proposition}[theorem]{Proposition}
\newtheorem{lemma}[theorem]{Lemma}
\newtheorem{corollary}[theorem]{Corollary}
\newtheorem{conjecture}[]{Conjecture}
\theoremstyle{definition}
\newtheorem*{remark}{Remark}

\def\llbracket{[\hspace{-.10em} [ }
\def\rrbracket{ ] \hspace{-.10em}]}

\def\build#1_#2^#3{\mathrel{
\mathop{\kern 0pt#1}\limits_{#2}^{#3}}}

\title{\vspace{-2cm}\textsc{Random planar maps  \&  growth-fragmentations} \vspace{0.5cm}}

\date{}

\DeclareSymbolFont{extraup}{U}{zavm}{m}{n}
\DeclareMathSymbol{\varheart}{\mathalpha}{extraup}{86}
\DeclareMathSymbol{\vardiamond}{\mathalpha}{extraup}{87}

\makeatletter
\renewcommand*{\@fnsymbol}[1]{\ensuremath{\ifcase#1\or  \vardiamond \or \clubsuit\or \spadesuit\or
   \mathsection\or \mathparagraph\or \|\or **\or \dagger\dagger
   \or \ddagger\ddagger \else\@ctrerr\fi}}
\makeatother

\author{Jean Bertoin \thanks{Universit\"at Z\"urich.\hfill  \texttt{jean.bertoin@math.uzh.ch}}\qquad  Nicolas Curien \thanks{Universit\'e Paris-Sud.\hfill  \texttt{nicolas.curien@gmail.com}}  \qquad Igor Kortchemski\thanks{CNRS \& CMAP, \'Ecole polytechnique. \hfill \texttt{igor.kortchemski@normalesup.org}}}
\begin{document}
\maketitle

\begin{abstract}
We are interested in the cycles obtained by slicing at all heights random Boltzmann triangulations with a simple boundary. We establish a functional invariance principle for the lengths of these cycles, appropriately rescaled, as the size of the boundary grows. The limiting process is described using a self-similar growth-fragmentation process with explicit parameters. To this end, we introduce a branching peeling exploration of Boltzmann triangulations, which allows us to identify a crucial martingale involving the perimeters of cycles at given heights. We also use a recent result concerning self-similar scaling limits of Markov chains on the nonnegative integers. A motivation for this work is to give a new construction of the Brownian map from a growth-fragmentation process.
\end{abstract}

\begin{figure}[!h]
 \begin{center}
 \includegraphics[width=14cm,height=9cm]{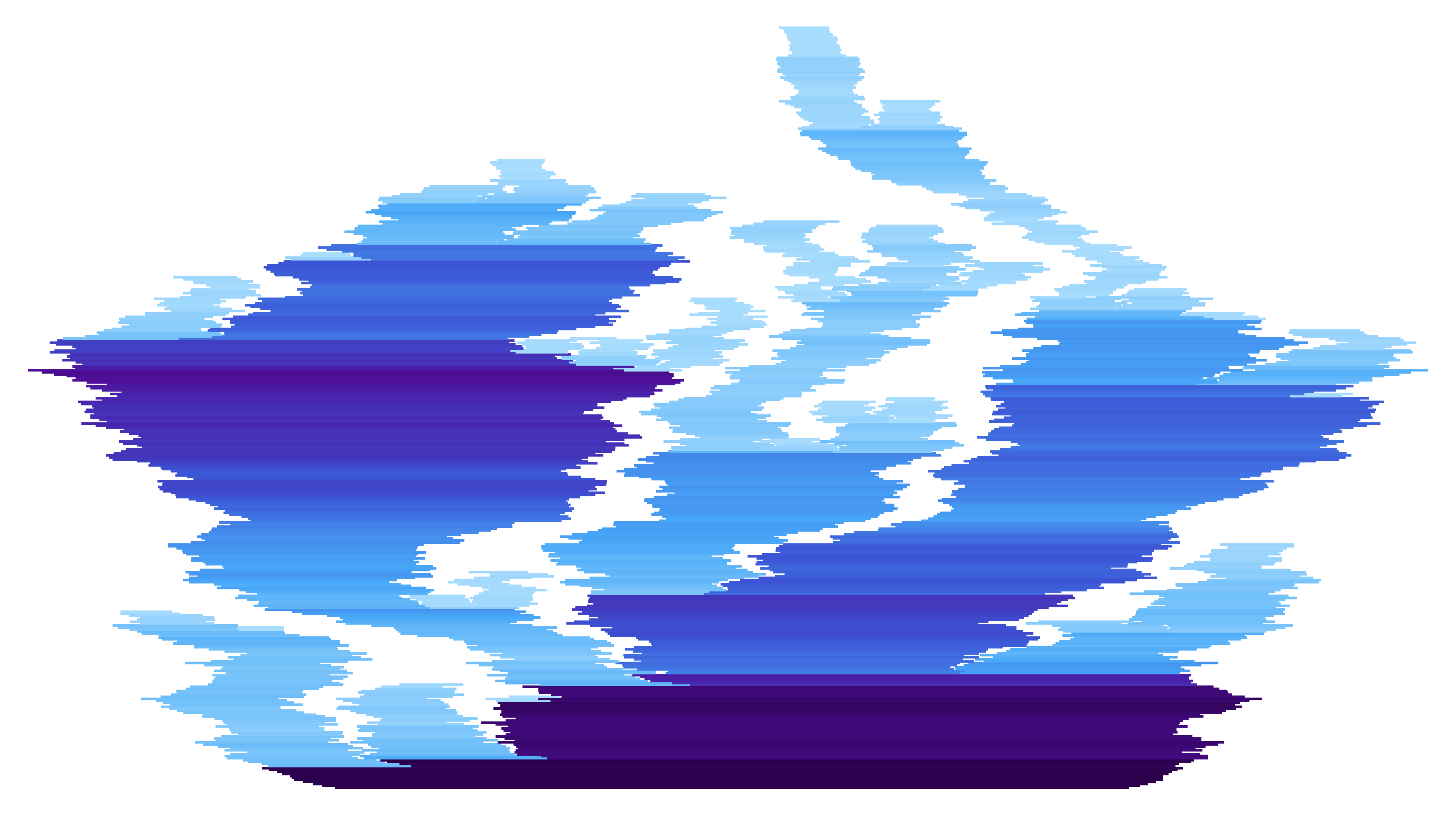}
 \caption{ \label{fig:sim1}A representation of the cycle lengths of a Boltzmann triangulation with a large boundary obtained by slicing it at all heights: horizontal line segments correspond to the lengths of the cycles of the ball of radius $r$ of the triangulation as $r$ increases. Here the longest cycles are the darkest ones. 
 }
 \end{center}
 \end{figure}
\clearpage 

\section{Introduction}
The study of the geometry of large random planar maps is a very active topic in probability theory, in part motivated by its connections with two dimensional Liouville quantum gravity, see \cite{LeGallICM,MieStFlour} for a detailed account and references. One of the main recent achievements in the area is the proof that the Brownian map, which is a random compact surface almost surely homeomorphic to the sphere, is the universal scaling limit  of various classes of random planar maps \cite{LG11,Mie11}. Apart from bijections between maps and decorated trees developed following the work of Schaeffer \cite{Sch98}, 
one of the main techniques to study random maps is the so-called \emph{peeling process}. The peeling process is an algorithmic procedure that explores a map step-by-step in a Markovian way.  It has been introduced in \cite{Ang03,Wat95} and since has  been a key ingredient  to establish many important results concerning the geometric structure of random planar maps \cite{Ang05,Ang03,ACpercopeel,BCsubdiffusive,CurKPZ,CLGpeeling,MN13}. The peeling process was also used to define ``hyperbolic''-type random maps \cite{AR13,CurPSHIT} and served as an inspiration for the introduction of QLE \cite{MS13}. See also the recent work of Budd introducing a variant of the peeling process called the ``lazy'' peeling process \cite{Bud15}.

In this work, we use a ``branching'' peeling 
process to study the lengths of the separating cycles at  fixed heights in large finite random triangulations.

   \begin{figure}[!h]
  \begin{center}
  \includegraphics[width=10cm]{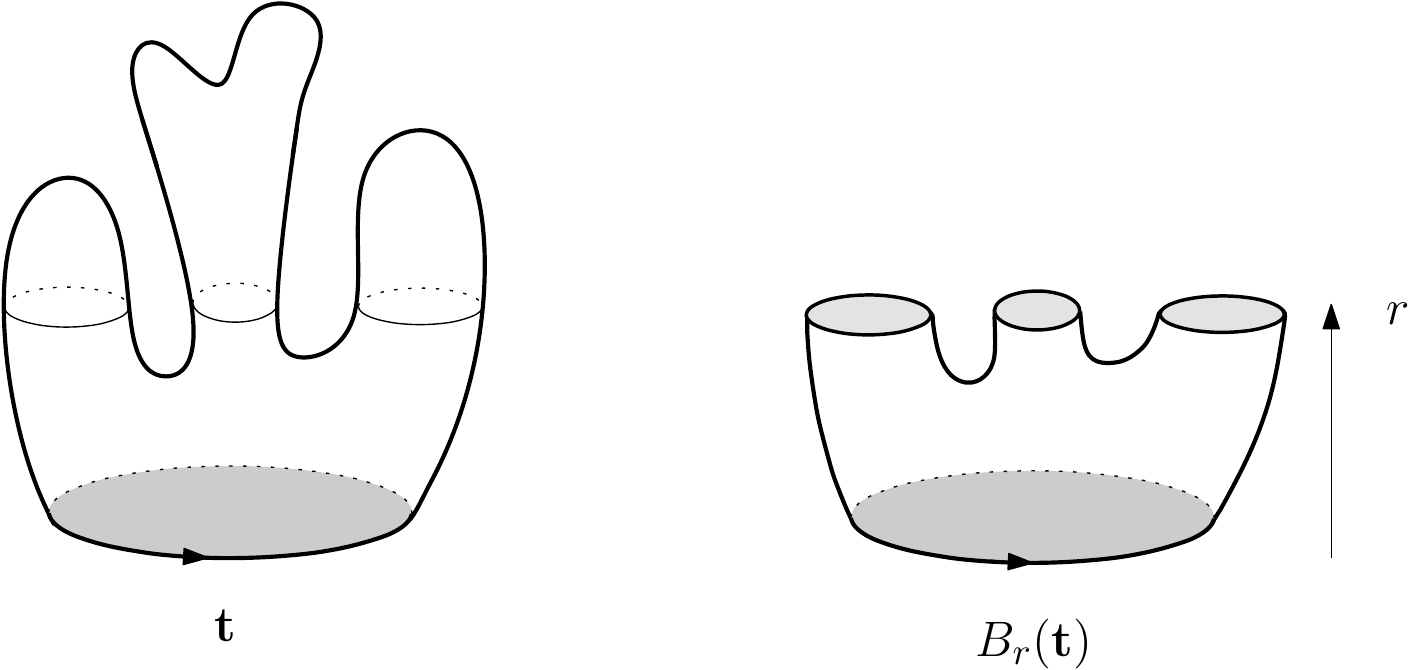}
  \caption{\label{fig:boule-trig} Illustration of the ball of radius $r$ in a triangulation with a boundary.}
  \end{center}
  \end{figure}
  
In order to state our main result, we start by introducing some notation. We restrict our attention to rooted triangulations of type I (i.e.~one-connected) where loops and multiple edges are allowed. As always, all our maps are rooted, meaning that a certain oriented edge, called the \emph{root edge}, is distinguished. In the sequel, without further notice, by \emph{triangulation} we always mean rooted one-connected triangulation.  For $p \geq 1$, a \emph{triangulation of the $p$-gon} is a (finite or infinite) planar map whose faces are all triangles except the face incident to the right of the root edge, called the \emph{external face}, which must be a simple face (pinch-points are not allowed) of arbitrary degree $p \geq1$. We say that $\mathbf{t}$ is a \emph{triangulation with a (simple) boundary} if $\mathbf{t}$ is a triangulation of a $p$-gon for a certain $p \geq 1$. For every $n\geq 0$ and $p \geq1$, we let $ \mathcal{T}_{n,p}$ denote the set of all triangulations of the $p$-gon with  $n$ internal vertices (that are vertices that do not belong to the external boundary).    For $p \geq 1$, the \emph{(critical) Boltzmann distribution} on triangulations of the $p$-gon is the probability
  measure on $\cup_{n\geq 0} \mathcal{T}_{n,p}$ that assigns a mass proportional to 
  $(12 \sqrt{3})^{-n}$ to each triangulation of $ \mathcal{T}_{n,p}$ (this measure can indeed be normalized to give rise to a probability measure, see Section~\ref{sec:enumeration}).  If $ \mathbf{t}$ is a triangulation of the $p$-gon and $x$ is a vertex of $ \mathbf{t}$, the \emph{height} of $x$ is the distance from $x$ to the boundary of $ \mathbf{t}$. For every $r \geq 0$, we denote by $B_{r}( \mathbf{t})$ the ball of radius $r$ of $ \mathbf{t}$ which consists of all the faces of $ \mathbf{t}$ that have at least one vertex at height less than or equal to $r-1$ in $ \mathbf{t}$ (with an additional operation concerning edges linking two vertices at height $r$ that will be discussed in Section~\ref{sec:def}), see Fig.~\ref{fig:boule-trig} and Fig.~\ref{fig:bouleplate}. Apart from its external face, the map $B_{r}( \mathbf{t})$ has a finite collection of ``holes'' surrounded by so-called \emph{cycles}.

 Let $T^{(p)}$ be a random Boltzmann triangulation of the $p$-gon. The main object of interest in the present work is the sequence of the lengths (or perimeters) of the cycles of $B_{r}(T^{(p)})$ ranked in decreasing order, which we  denote by 
 $$ {{\bf L}}^{(p)}(r) \coloneqq \left(L^{(p)}_1(r),   L^{(p)}_2(r), \ldots\right).$$

\paragraph{Main result.}  Our main result is a functional invariance principle that describes the scaling limit of the process ${{\bf L}}^{(p)}=( {{\bf L}}^{(p)}(r), r\geq 0)$ as $p\to\infty$ (see Figure \ref{fig:sim1} for an illustration). In order to state this result, we have  to introduce  first a certain random process with values in the space of non-increasing cube-summable series
$$\ell^{\downarrow}_3 \coloneqq  \left\{{\bf x}=(x_i)_{i\in\N}: x_1\geq x_2\geq \cdots \geq 0 \hbox{ and } \sum_{i=1}^{\infty} x_i^3<\infty\right\}.$$
In this direction, we consider  for every $q\geq 0$
\begin{equation}  \label{eq:defpsi}
 \Psi(q) \coloneqq  - \frac{8}{3} q + \int_{1/2}^1 \left( x^q -1 + q (1-x)\right) \big(x(1-x)\big)^{-5/2} \mathrm{d}x.  
 \end{equation}
 The change of variables $x=\exp(y)$ in the integral above enables us to view $\Psi$ as the Laplace exponent of a spectrally negative L\'evy process. That is, there is a process $\xi=(\xi(t); t\geq 0)$ with independent and stationary increments and no positive jumps, such that 
$\E[\exp( q \xi(t))]=\exp(t\Psi(q))$ for every $t\geq 0$ and $q\geq 0$. It is easily checked that $\Psi'(0)=\E[\xi(1)]<0$ {(see \eqref{eq:drift})}, so $\xi$ drifts to $-\infty$, in the sense that $\lim_{t\to\infty} \xi(t)=-\infty$ a.s.  
{For $x>0$,} we then consider the time-substitution
$$\tau(t)= \inf  \left\{ u \geq 0 ; \int_{0}^{u}  \e^{  \xi(s)/2}\d s>t \right\}, \qquad  t \geq 0 $$
with the convention that $\inf \emptyset =\infty$, i.e.~$\tau(t)=\infty$ whenever $t\geq \int_{0}^{\infty}  \e^{  \xi(s)/2} \mathrm{d}s$.
The process derived from $\xi$ by the Lamperti transformation \cite{Lam72} {started from $x>0$}
\begin{equation} \label{eq:defY}
{{X}}(t)   \coloneqq  {x} \exp\left(\xi(\tau(  {x^{-1/2}} t))\right) \,, \qquad t\geq 0,
\end{equation}
(with the convention $\exp\left(\xi(\infty)\right)=0$) is a self-similar Markov process {with index $-1/2$ started from $x$. This means that or every $x>0$, the rescaled process $(x{X}(x^{-1/2}t), t\geq 0)$, where $X$ is started from $1$, has the same law as ${ X}$ started from $x$. This index of self-similarity is the one of the limiting process in our functional invariance principle, and comes, roughly speaking, from the fact that the height of $T^{(p)}$ is of order $p^{2}$.}

We next use ${{X}}$  to define a self-similar growth-fragmentation process with binary dislocations,
following \cite{BeGF} {(see Section~\ref{sec:cell-systems} for details)}. Informally, we view ${{X}}(t)$ as the size of a typical particle or cell at age $t$, and consider the following system. We start at time $0$ from a single cell with size $1$, and suppose that its size evolves according to ${{X}}$. We interpret each (negative) jump of ${{X}}$ as a division event for the cell, in the sense that whenever 
$\Delta {{X}}(t)\coloneqq {{X}}(t)-{{X}}(t-)=-y<0$, the cell divides at time $t$ into a mother cell and a daughter. After the splitting event, the mother cell has size ${{X}}(t)$ and the daughter cell has size $y$  and the evolution of the daughter cell is then governed by the law of the same self-similar Markov process ${{X}}$ (starting of course from $y$), and is independent of the processes of all the other daughter particles. And so on for the granddaughters, then great-granddaughters, ... 

Specializing results in \cite{BeGF} to this case, one can check that for every $t\geq 0$,  the family of the sizes of cells which are present in the system at time $t$ is cube-summable, and can therefore be  ranked in  non-increasing order. This yields a random variable with values in $\ell^{\downarrow}_3$ which we denote by ${\bf X}(t)=(X_1(t), X_2(t), \ldots)$. The process ${\bf X}=({\bf X}(t); t\geq 0)$  is  called the self-similar growth-fragmentation process with index $-1/2$ associated to the spectrally negative L\'evy process $\xi$ with Laplace exponent $\Psi$.

Finally, set
\begin{equation}
\label{eq:ta}\mathsf{t}_{\one} = \frac{\sqrt{3}}{8 \sqrt{\pi}}, \qquad \aone= \frac{1}{2 \sqrt{3}}.
\end{equation}
    The reason for introducing these quantities stems from the universality of $ {\bf X}$: we believe that the next result holds for a wide class of random maps, and that the only difference will appear in the time change of the limiting process, see \cite[Section 6]{CLGpeeling}. 

Recall that $T^{(p)}$ is  a Boltzmann triangulation of the $p$-gon and that $ {{\bf L}}^{(p)}(r)$ denotes the sequence of the lengths of the cycles of $B_{r}(T^{(p)})$ ranked in decreasing order. 

  \begin{theorem}  \label{thm:main}We have 
    $$ \left( \frac{1}{p} \cdot   {{\bf L}}^{(p)}\big(r \sqrt{p}\big) ; r \geq 0\right) \quad  \xrightarrow[p\to\infty]{(d)}  \quad  \left(  {\bf X} \left(  \frac{2\tone}{\aone} \cdot r \right) ; r \geq 0 \right),$$ 
where the convergence holds in distribution in the space of càdlàg process taking values in $\ell^{\downarrow}_3$ equipped with the Skorokhod $J_{1}$ topology.    \end{theorem}
    
    {Note that only the quantity ${\tone}/{\aone}$ is relevant in this statement. However, we have introduced  $\mathsf{t}_{\one}$ and $\aone$ to underline that the latter quantity  is a mixture of two model-dependent constants of different nature, identified in \cite{CLGpeeling}.}

We believe that (versions of) the growth-fragmentation process $ {\bf X}$ should also naturally appear in the continuous scaling limits of large Boltzmann triangulation as in the Boltzmann  ``Brownian Disk'' \cite{Bet11,BM15} or even more directly in the Brownian map \cite{LG11,Mie11}, the Brownian plane \cite{CLGHull,CLGplane} or in the Quantum Loewner Evolution of parameter $ \sqrt{8/3}$  \cite{MS15,MS13}. 
 
 It may now be interesting to briefly, and somewhat informally, recall some of the main properties of ${\bf X}$ which follow from \cite{BeGF}. Roughly speaking, we can think of ${\bf X}$ 
 as a self-similar compensated fragmentation, in the sense that it describes the evolution of particles that grow and divide independently one of the other as time passes.
 In particular, ${\bf X}$ fulfills the branching property, and is self-similar with index $-1/2$, in the sense that for every $c>0$, the rescaled process $(c{\bf X}(c^{-1/2}t), t\geq 0)$ has the same law as ${\bf X}$ started from the sequence $(c, 0, 0, \ldots)$. The path with values in $\ell^{\downarrow}_3$ defined by $t\to {\bf X}(t)$ is a.s. càdlàg  (actually, it even takes values in 
 $\ell^{\downarrow}_q$ for every $q>3/2$ and is càdlàg in 
 $\ell^{\downarrow}_q$ for every $q>2$).  
 
 By construction, all the dislocations occurring in ${\bf X}$ are binary, i.e.~they correspond to replacing some mass $m$ in the system by two smaller masses $m_1$ and $m_2$ with $m_1+m_2=m$. 
 The rates at which such dislocations occur are described by the so-called dislocation measure, which we can view  as a measure $\nu$ on $[1/2,1)$ by focusing on the distribution of the largest fragment (because all dislocations are binary), and is simply obtained by taking the image of the L\'evy measure of the L\'evy process $\xi$ by the exponential function, that is
 $$\nu(\d x) = (x(1-x))^{-5/2} \d x\,,\qquad x\in [1/2,1).$$  Informally, in ${\bf X}$, each mass $m>0$ splits into a pair of smaller masses $(xm,(1-x)m)$ at rate $m^{-1/2} \nu(\d x)$. 
 Observe that $\int(1-x)^2\nu(\d x)<\infty$, as required in \cite{BeCF}, however $\int(1-x)\nu(\d x)=\infty$ which underlines the necessity of compensating the dislocations.
 
 In this direction, we recall that a different  self-similar fragmentation process also occurs when splitting at heights the Brownian Continuum Random Tree (CRT) \cite{Ald91}. Namely, consider a Brownian CRT and for every $r\geq 0$, we denote by $B^{\rm c}_{r}$ the complement of the ball with radius $r$, i.e.~the set of all points at distance greater than or equal to $r$ from the root. As the level $r$ increases, the sequence of the sizes of the connected components of $B^{\rm c}_{r}$ forms a self-similar (pure) fragmentation process with index $-1/2$, no erosion, having only binary dislocations, 
and whose dislocation measure has the form 
$$\nu'(\d x)   =   (2\pi)^{-1/2}(x(1-x))^{-3/2} \d x\,,\qquad x\in [1/2,1).$$
Note that $\int(1-x)\nu'(\d x)<\infty$, so dislocations need not to be compensated and indeed for the CRT, the sizes of the connected components of $B^{\rm c}_{r}$  decrease as $r$ increases. We refer to \cite{BeSSF} and \cite{UB} for details, and further to \cite{MiFST} for an extension to the fragmentation at heights for the stable L\'evy tree. The similarity between the dislocation measures $\nu$ and $\nu'$ is striking. Notice however, that apart from the dislocation measure and the self-similarity exponent,  the description of the process $ {\bf X}$ requires an additional parameter: the drift term $ -\frac{8}{3}$ present in the definition of $ \Psi$. In fact, we expect that we can give an alternative definition of the Brownian map (or more generally of the Brownian disk or the Brownian pla ne) whose primary ingredients are a growth-fragmentation process similar to $ {\bf X}$ describing the cycle structure at heights and a family of independent uniform random variables describing how to split a cycle when a dislocation event occurs. We plan to pursue this goal in future works. See also a related recent approach of Miller and Sheffield giving an axiomatic characterization of the Brownian map \cite{MS15} by using a related ``breadth-first exploration'' of the Brownian map, as well as \cite{MS15+}, which may suggest that the growth-fragmentation $ {\bf X}$ appears in QLE$(8/3,0)$. Let us now give some elements of the proof of Theorem \ref{thm:main}.

\paragraph{Branching peeling process of Boltzmann triangulations.}   A \emph{triangulation with holes} is a planar map whose faces are all triangles except for the external face (which is the one lying on the right of the root edge) and certain distinguished faces (possibly none), called holes, whose boundaries are simple cycles which share no edge in common (but may share edges with the external face). Note that by convention, the external face is never a hole.  The boundaries of the holes are called the \emph{cycles}. Cycles are further rooted, i.e. for each cycle, an oriented edge is distinguished. If $\mathbf{h}$ is a triangulation with holes and $e$ is an edge belonging to a cycle of $\mathbf{h}$, the triangulation with holes obtained by the \emph{peeling of $e$} is the triangulation $\mathbf{h}$ to which we ``add'' the face incident to $e$ that was not already in $\mathbf{h}$ (see Section~\ref{sec:bpp} for a formal definition).
 
In the cases of the UIPT (Uniform Infinite Planar Triangulation) and the UIPQ (Uniform Infinite Planar Quadrangulation), the peeling explorations appearing in the literature are sequences of triangulations with holes starting from the root and obtained iteratively by peeling edges along the boundary of the explored region, but by also adding at each step the finite regions that the added face may enclose (it is known that the UIPT and the UIPQ have one end). In particular, for the UIPT (resp.~UIPQ), only triangulations (resp.~quadrangulations) with a single hole appear in such peeling explorations. However, in our setting of finite maps, we will work with peeling explorations where one does not fill-in holes when adding faces, and one has potentially the choice to peel an edge belonging to different holes at each step. We also restrict ourselves to explorations where at each step, the peeled edge is chosen in a deterministic way (see Section~\ref{sec:bpp} below for a formal definition). For this reason, we call such an exploration a (deterministic) \emph{branching peeling process}.

We exhibit two martingales that appear in any deterministic branching peeling process of a Boltzmann triangulation with a boundary. Roughly speaking, the first one, called the \emph{volume martingale}, is related to  the sum of the squares of the lengths of the cycles of a triangulation with holes, while the second one, called the \emph{cycle martingale}, is related to the sum of their cubes. The latter probabilistic structure will play a key role in this work, and was already indentified in \cite[Theorem 4]{CLGpeeling} for a specific peeling algorithm. Specifically, this martingale can be seen as the Radon--Nikodym derivative of the peeling exploration in the UIPT of the $p$-gon with respect to the same exploration in a Boltzmann triangulation  of the $p$-gon (Proposition \ref{prop:martingale}). This enables us to reformulate questions concerning the peeling process on Boltzmann triangulations in terms of the peeling process of the UIPT, which is well understood \cite{CLGpeeling}. Intuitively speaking, this is very similar to the fact that a critical Galton--Watson tree conditioned to survive is the law of the Galton--Watson tree biased by the population size at each generation, which is a martingale.

Another important tool is the understanding of the evolution of the locally largest cycle: imagine a branching peeling process that starts from a triangulation with holes with one particular distinguished cycle $\mathscr{C}$. Then, at the first time an edge is peeled on $\mathscr{C}$, the cycle $\mathscr{C}$ may give rise to a new cycle or may split into two cycles. Choose to distinguish the longest one, and then repeat the procedure (see Section~\ref{sec:llc} for a formal definition). These distinguished cycles are called the \emph{locally largest} ones. It turns out that their length evolves as a Markov chain on the nonnegative integers, whose transition kernel is  described by using the explicit transition probabilities of the peeling process. An application of the results of \cite{BK14} then yields a functional invariance principle for the perimeter of the locally largest cycle in any (deterministic) branching peeling process  (Proposition \ref{prop:scalingllcycle}).

\paragraph{Branching peeling by layers.} The main tool to prove Theorem \ref{thm:main} is to use the \emph{peeling by layers algorithm}, which specifies how to choose the next edge to peel in a particular way so as to explore metric balls. Indeed, if $T^{(p)}$ is a Boltzmann triangulation of the $p$-gon, it turns out that the sequence of triangulations with holes $(B_{r}( T^{(p)}); r \geq 0)$ may be recovered by considering the branching peeling by layers along a certain increasing sequence of stopping times. By adapting the arguments of \cite{CLGpeeling} to our case, we get a functional invariance principle for the perimeter of the locally largest cycle appearing in  $(B_{r}( T^{(p)}); r \geq 0)$ (Proposition \ref{prop:hauteur}). A final ingredient in the proof of Theorem \ref{thm:main} involves a cutoff procedure.  Indeed, the previous results allow to control the evolution of the perimeters of cycles in a branching peeling process of $T^{(p)}$ until we find a cycle of perimeter less than $ \varepsilon p$. This enables us to control the cycle structure of the triangulation $ \mathsf{Cut}(T^{(p)}, \varepsilon p)$ obtained{, roughly speaking,} by keeping faces of $T^{(p)}$ adjacent to vertices that are not separated from the external boundary by a cycle of perimeter less than $ \varepsilon p$ (see {Section~\ref{sec:cutoff}} for a precise definition and Fig.~\ref{fig:cutoff} for an illustration).

\begin{figure}[!h]
 \begin{center}
 \includegraphics[width=5.2cm]{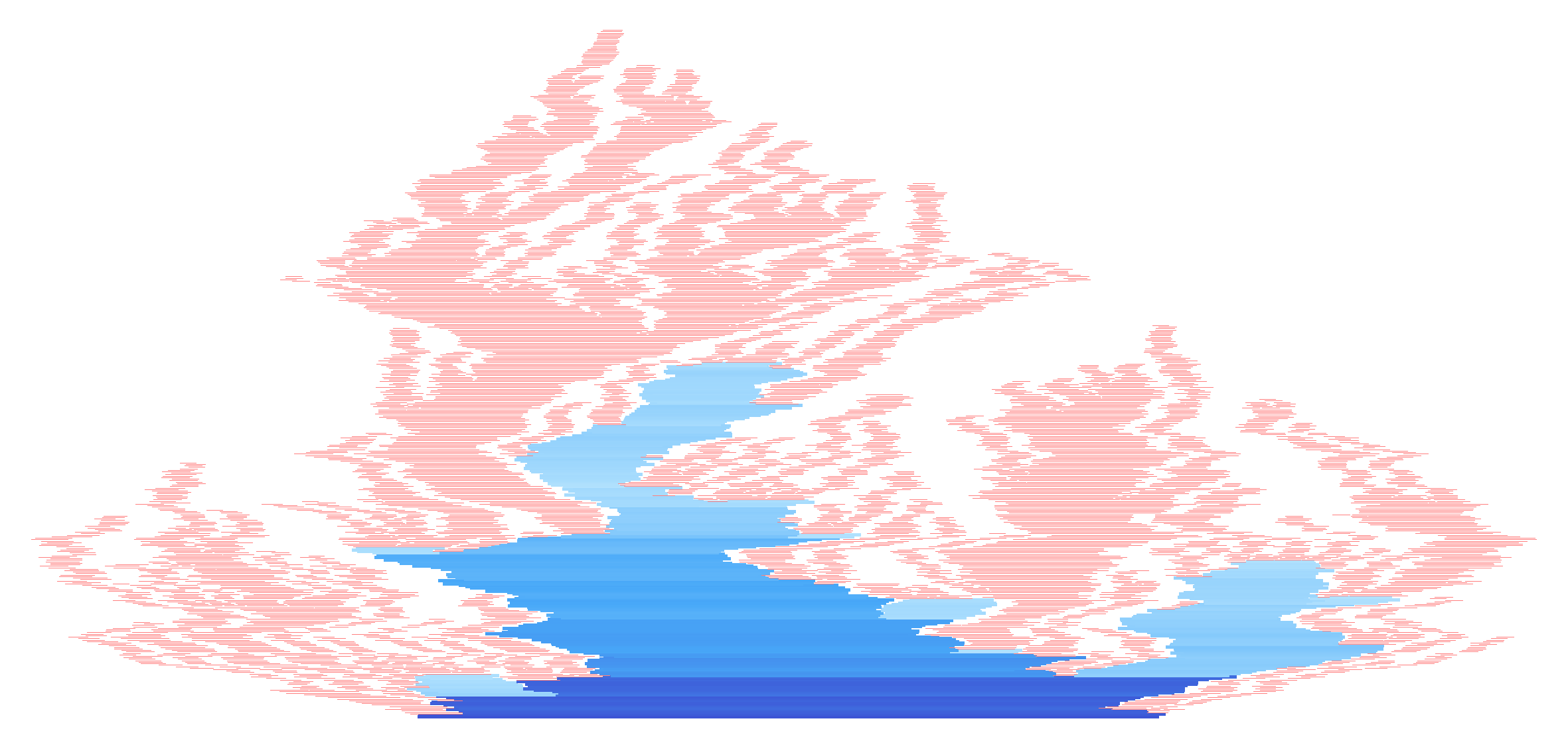}
  \includegraphics[width=5.2cm]{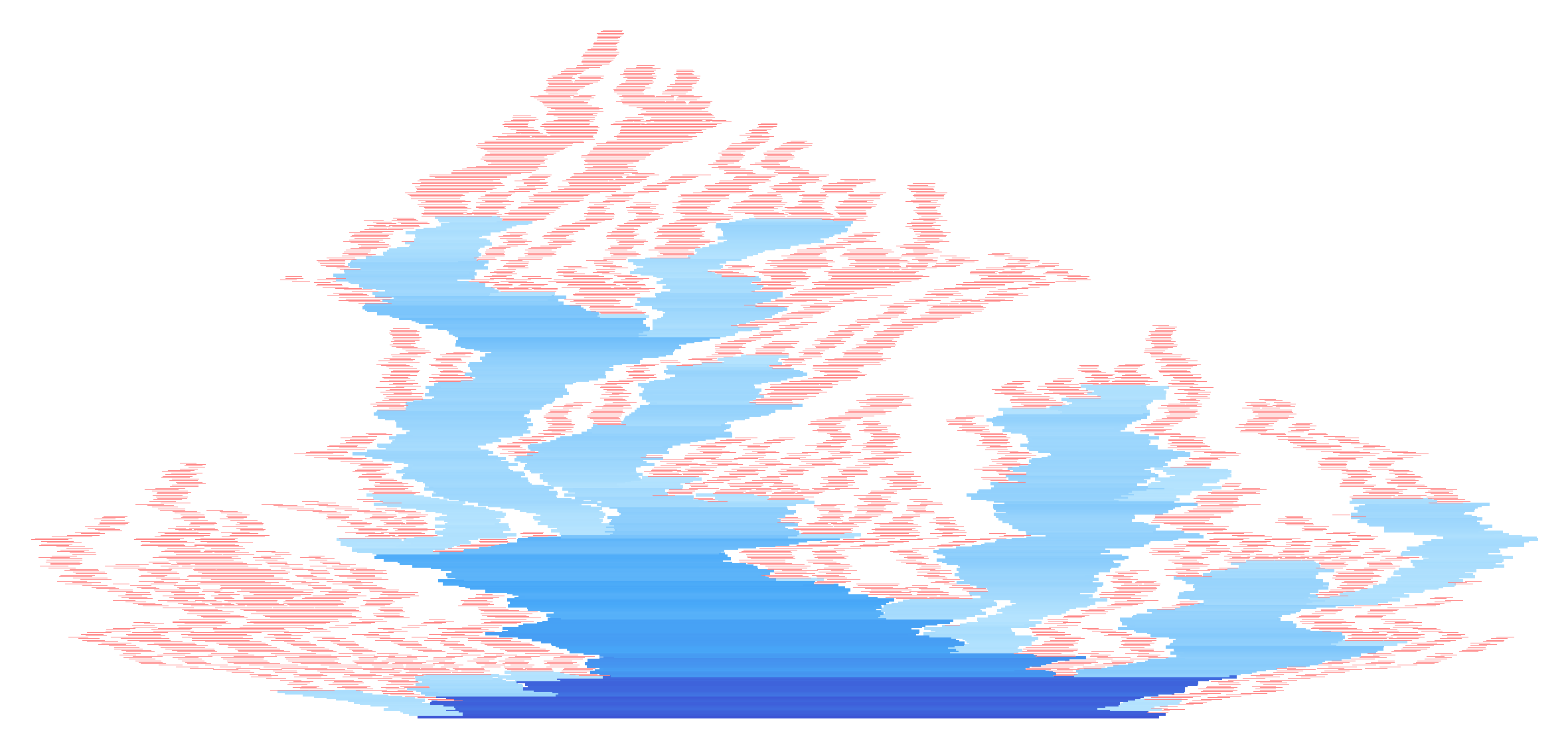}
   \includegraphics[width=5.2cm]{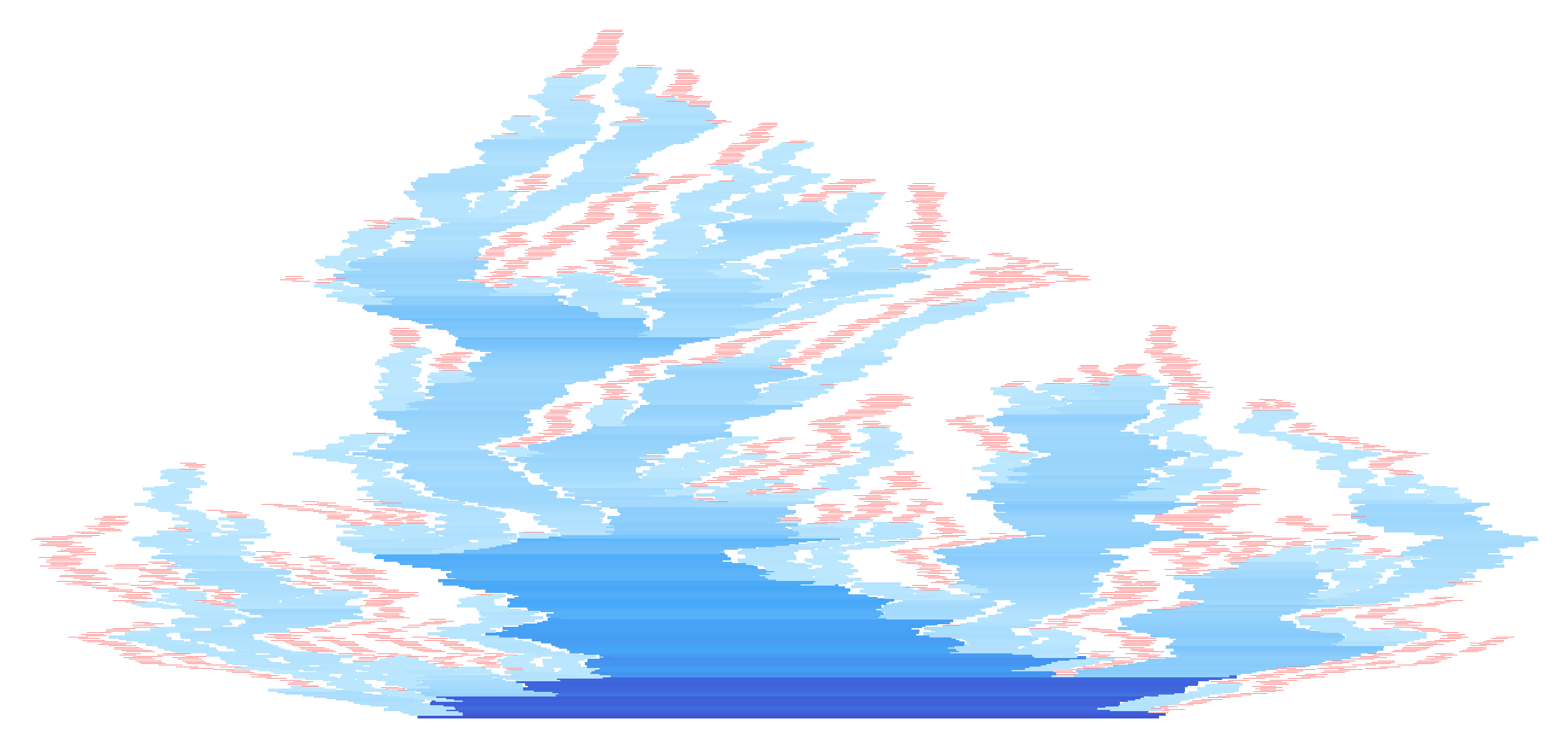}
 \caption{ \label{fig:cutoff}The cutoff procedure: we only keep the cycles that  are not separated from the external boundary of $T^{(p)}$ by a cycle of length less than $ \varepsilon p$. In the figure, the cycles in pink are discarded, and the value of $\varepsilon$ decreases from left to right.}
 \end{center}
 \end{figure}

Towards proving Theorem \ref{thm:main} we show that the cycles at heights of $T^{(p)}$ that have been discarded in $ \mathsf{Cut}(T^{(p)}, \varepsilon p)$ are negligible in the $\ell^3$-sense uniformly in $p \geq 1$ as $ \varepsilon \to 0$ (Proposition \ref{prop:l3lostcycles}). For this proof, the cycle martingale and its associated probabilistic structure play a crucial role. \medskip 

Finally, we establish a result which completes Theorem \ref{thm:main} by a more geometric point of view. We show that as $ \varepsilon \to 0$, uniformly in $p$, the metric structure of $\mathsf{Cut}(T^{(p)}, \varepsilon p)$ is close to that of $ T^{(p)}$. Indeed notice that even if Theorem \ref{thm:main} shows that the cycles of $ T^{(p)}\backslash  \mathsf{Cut}(T^{(p)}, \varepsilon p)$ are small in the $\ell^{3}$ sense, this does not rule out the possibility of having long and thin ``tentacles'' of length order $\sqrt{p}$ in $ T^{(p)}\backslash  \mathsf{Cut}(T^{(p)}, \varepsilon p)$. However, this is not the case: 

\begin{theorem}\label{thm:no tentacles} Let $T^{(p)}$ be a Boltzmann triangulation of the $p$-gon. Then, for every $\delta >0$, we have
$$ \sup_{p \geq 1} \mathbb{P}\Big( \mathrm{d_{H}}\big( T^{(p)} , \mathsf{Cut}(T^{(p)}, \varepsilon p) \big) \geq \delta \sqrt{p} \Big)  \quad \xrightarrow[\varepsilon \to 0]{} \quad 0,$$
{where $ \mathrm{d_{H}}$ denotes the Hausdorff distance (we consider the graph distance on $T^{(p)}$ and view $\mathsf{Cut}(T^{(p)}, \varepsilon p)$ as a subset of   $T^{(p)}$)}.
\end{theorem}

Note that since the height of $\mathsf{Cut}(T^{(p)}, \varepsilon p)$ is of order $ \sqrt{p}$, Theorem \ref{thm:no tentacles} tells us that $ \mathsf{Cut}(T^{(p)}, \varepsilon p)$ is indeed a good metric approximation of $T^{(p)}$.  An informal consequence of the above result together with Theorem \ref{thm:main} is that the ``cactus tree'' made by contracting all cycles at heights of $T^{(p)}$ into points converges, after scaling by ${p}^{-1/2}$,  towards the ``continuous tree'' associated with the growth-fragmentation process $ \mathbf{X}$, see \cite{CLGMcactus} for a related convergence using totally different tools.  This is a further indication that the scaling limit of random planar maps could indeed be described via (a version of) the process $ {\bf X}$ of Theorem \ref{thm:main}. The main idea underlying the proof of Theorem \ref{thm:no tentacles} is, roughly speaking,  to couple with positive probability a Boltzmann triangulation $T^{(p)}$ with a large uniform triangulation of the sphere of total size at least $p^2$  (Lemma \ref{lem:coupling}). This coupling enables us to transfer known results concerning the metric structure of uniform triangulations of the sphere (which have been established using bijective tools) to the case of Boltzmann triangulations with a boundary of fixed length.\bigskip

The rest of the paper is organized as follows. In Section \ref{bigsec:bpp}, we introduce the notion of general branching peeling explorations, and study the probabilistic structure that arises for Boltzmann triangulations. In particular, we explain their relation with the peeling explorations of the UIPT considered in \cite{CLGpeeling}. Using \cite{BK14}, we then obtain a scaling limit for the lengths of the locally largest cycle in large Boltzmann triangulations. In Section \ref{bigsec:pbl}, we then use the peeling by layers, which is specific peeling algorithm that is the key in this work, to get the scaling limit of the locally largest cycle at given heights. We also introduce the cut-off procedure.  Section \ref{bigsec:proofthm1} introduces the basics on cell systems and their scaling limits in order to prove Theorem \ref{thm:main}. Finally, Section \ref{bigsec:proofthm2} is devoted to the proof of Theorem \ref{thm:no tentacles}.\medskip

\noindent \textbf{Acknowledgments:} We thank the Isaac Newton Institute for hospitality, where part of  this work has been carried out during the Random Geometry 2015 program. We are grateful to Mireille Bousquet-Mélou for finding a simple form of $\kappa(q)$ in \eqref{eq:mireille}. NC and IK acknowledge the support of ANR GRAAL (ANR-14-CE25-0014) and ANR MAC 2 (ANR-10-BLAN-0123). {Finally, we would like to thank two anonymous referees for many useful comments.}

\tableofcontents

\section{Branching peeling exploration of Boltzmann triangulations}
\label{bigsec:bpp}

In this section we rigorously introduce  the notion of deterministic branching peeling process, identify the two martingales appearing in any (deterministic) branching peeling process, and establish a functional invariance principle for the perimeter of the locally largest cycle in such explorations.

\subsection{Definitions}
\label{sec:def}

Recall from the introduction that a triangulation with holes is  a planar map whose faces are all triangles except for the external face (the one lying on the right of the root edge) and certain distinguished faces (possibly none), called holes, whose boundaries are simple cycles which share no edge in common (but can share edges with the external face). Recall that  boundaries of the holes are called the cycles and that the external face is never a hole (a triangulation with holes having no holes is just a triangulation of the polygon). It will be implicit that a distinguished oriented edge is chosen on each cycle, which allows  to glue a triangulation with holes $\mathbf{h}$ inside a cycle $\mathscr{C}$ in a canonical way by gluing the external face of $\mathbf{h}$ on $\mathscr{C}$ and by matching the roots.

If $ \mathbf{t}$ is a triangulation of the $p$-gon and $x\in\mathbf{t}$, recall that the \emph{height} of $x$ is the distance of $x$ to the boundary of $ \mathbf{t}$. For $r \geq 1$, the ball of radius $r$ of $ \mathbf{t}$ is the map $B_{r}( \mathbf{t})$ that  consists of all the faces of $ \mathbf{t}$ which have a vertex  at height less than or equal to  $r-1$ in $ \mathbf{t}$, and by convention $B_{0}(\mathbf{t})$ is just the boundary of $\mathbf{t}$. We also make an additional operation:  in $ B_{r}( \mathbf{t})$, the edges between two vertices at distance $r$ which do not belong to a same cycle are split into two edges enclosing a $2$-gon. This may seem strange at first glance, but will be essential in the sequel: roughly speaking, seen from the external boundary of $B_{r}( \mathbf{t})$ one does not yet know whether or not there are vertices ``inside'' these $r-r$ edges. By construction, $B_{r}( \mathbf{t})$ is a triangulation with holes (see Fig.~\ref{fig:bouleplate} for an example).
  
 \begin{figure}[!h]
  \begin{center}
  \includegraphics[width=0.9 \linewidth]{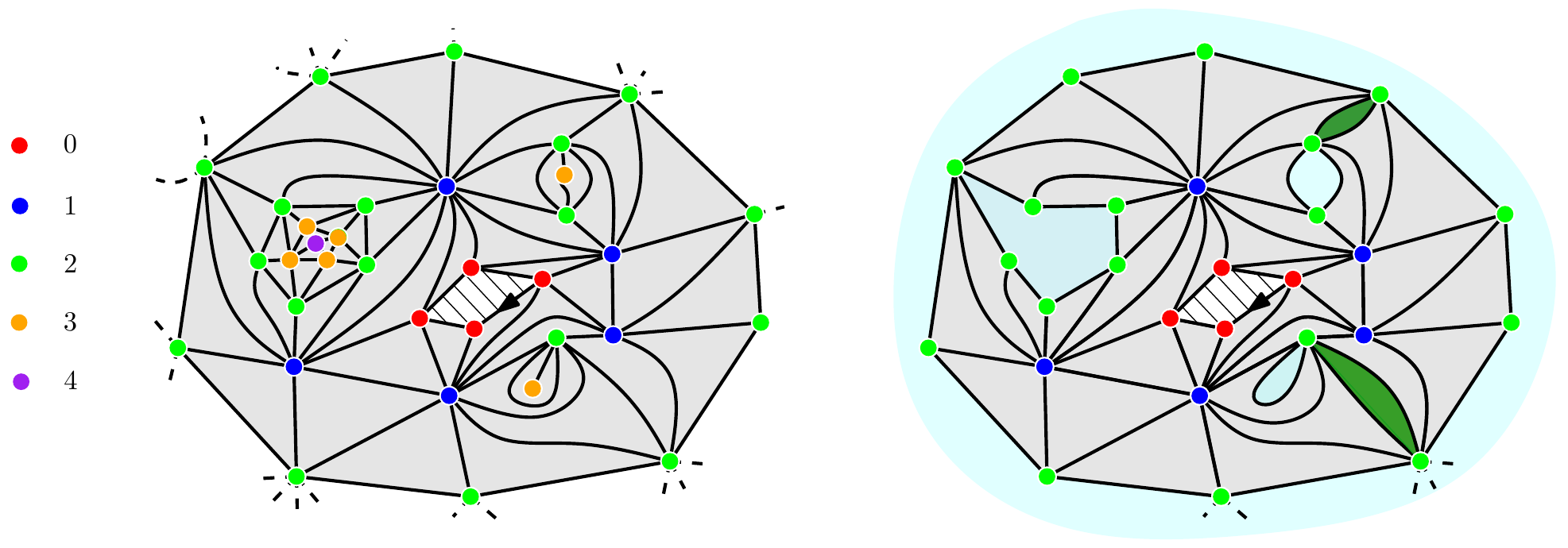}
  \caption{ \label{fig:bouleplate} Illustration of the ball of radius $2$ in a certain triangulation of the $4$-gon. Notice that the edges between two vertices at distance $2$ which do not belong to a same cycle are split into two edges enclosing a $2$-gon  (in dark green). The external face of the triangulation is hatched.}
  \end{center}
  \end{figure}

\subsection{Enumeration of triangulations with a simple boundary}

\label{sec:enumeration}
Recall that $ \mathcal{T}_{n,p}$ denotes the set of all triangulations of the $p$-gon with  $n$ internal vertices. We now state several known enumerative results that we will need in the sequel (we refer to \cite[Section~6.1]{CLGpeeling}  and to \cite{Kri07} for proofs). First, one can exactly enumerate the set $\mathcal{T}_{n,p}$ for $n \geq 0, p \geq 1$ with $(n,p) \neq (0,1)$:
\begin{equation}
\label{eq:tnpexact}
   \# \mathcal{T}_{n,p} =4^{n - 1}  \frac{p\, (2p)!\,(2 p + 3 n - 5)!!}{(p!)^2\,n! \,
(2 p + n - 1)!!}   \quad \mathop{\sim}_{n \rightarrow \infty} \quad    C(p) \,(12 \sqrt{3})^{n}\, n^{-5/2},
     \end{equation}
where
   \begin{equation} C(p) = \frac{3^{p-2} \, p \, (2 p)!}{4 \sqrt{2 \pi} \,  (p!)^2}  \quad \mathop{\sim}_{p \rightarrow \infty} \quad \frac{1}{36\pi \sqrt{2}} \, \sqrt{p} \ 12^p. \label{equivalentcp} 
   \end{equation}
Note that $ \#T_{0,1}=0$.  The exact formula for $ \# \mathcal{T}_{n,p} $ in \eqref{eq:tnpexact} gives $\# \mathcal{T}_{n,p}=1$  for $n=0$ and $p=2$. This
   formula is valid provided we make  the special 
   convention that the only element of $ \mathcal{T}_{0,2}$ is a rooted planar map consisting of a single (oriented) edge between two vertices which is   viewed as a triangulation with a simple boundary of length $2$. We shall call this map the trivial triangulation.
    It will be used in the sequel to ``fill-in'' holes of size two in a triangulation with holes. We also note that there is a natural bijection between  plane triangulations (or triangulations of the sphere) having $n$ vertices and triangulations of the $1$-gon having $n-1$ inner vertices \cite[Section 1.3]{Kri07}: simply split the root edge of a triangulation of the sphere (which may be a loop!) into a $2$-gon, and add a loop inside this $2$-gon, which is declared to be the new root. Hence, in the following, we may and will view all plane triangulations as triangulations of the $1$-gon after applying the above operation, which we call the \emph{root-transformation}.

\begin{figure}[!h]
 \begin{center}
 \includegraphics[width=0.6\linewidth]{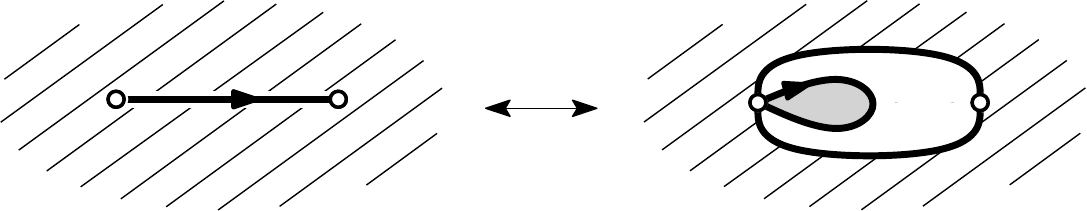}
 \caption{ \label{fig:transform-root} The root-transformation operation.}
 \end{center}
 \end{figure}

   The exponent $5/2$ appearing in \eqref{eq:tnpexact} is typical in the enumeration of planar maps, and yields that
   $$Z(p) \coloneqq \sum_{n=0}^\infty \Big(\frac{1}{12 \sqrt{3}}\Big)^n\, \#  \mathcal{T}_{n,p}<\infty, \qquad p \geq 1.$$ 
The expression of $Z(p)$ is explicit:
  \begin{equation}
  \label{eq:defzp}
  Z(p) = \frac{6^p\,(2p-5)!! }{8 \sqrt{3} \,p!} \quad \text{if \ } p\geq 2,
  \qquad \qquad  Z(1) =  \frac{2 - \sqrt{3}}{4}. \end{equation}
  The quantity $Z(p)$ can be interpreted as the partition function of the (critical) Boltzmann probability distribution of triangulations of the $p$-gon (also called the free distribution in \cite{AS03}).  More precisely, the latter is a probability
  measure on $\bigcup_{n\geq 0} \mathcal{T}_{n,p}$ that assigns mass
  $(12 \sqrt{3})^{-n} Z({p})^{-1}$ with each triangulation of $ \mathcal{T}_{n,p}$. We also have 
  $$ \sum_{p=0}^\infty Z(p+1) \,x^p =  \frac{1}{2} + \frac{(1-12x)^{3/2}-1}{24 \sqrt{3} x}, \qquad x \in [0,1/12].$$
From \eqref{eq:defzp}  and the last display, we get the following estimates, which we state for later use: 
\begin{eqnarray} Z(p+1) \!\!&\underset{p\to \infty}{\sim}& \!\!     \mathsf{t}_{ \one} \cdot  12^p p^{-5/2}, \quad \mbox{ where }\  \mathsf{t}_{\one} = \frac{\sqrt{3}}{8 \sqrt{\pi}} ,\label{eq:asympZp}\\ 
\sum_{p=0}^\infty Z(p+1)\,12^{-p}  \!\! \!\!&=& \!\!  \!\!\frac{3-\sqrt{3}}{6} \label{eq:sumZp}\;,\\
\sum_{p=0}^\infty p\, Z(p+1)\,12^{-p} \!\!  \!\!&=& \!\!  \!\!  \frac{\sqrt{3}}{6} \label{eq:sumkZp}\;.  \end{eqnarray}

We use the notation $ \tone$ following \cite[Section~6]{CLGpeeling}, as it will be useful to discuss universality results. Also, since  triangulations of the $1$-gon are in bijection with plane triangulations by the root-transformation, the Boltzmann distribution on  the latter set induces a  probability measure on the space of all  triangulations of the sphere (including the trivial one). A random triangulation distributed according to this probability measure is called a Boltzmann triangulation of the sphere.  Equivalently, the law of a Boltzmann triangulation of the sphere assigns  mass  $(12\sqrt{3})^{1-n} Z(1)^{-1}$ with every triangulation of the sphere with $n$ vertices. 

\paragraph{UIPT of the $p$-gon.} For fixed $p \geq 1$, there exists an infinite random map  $T_{\infty}^{(p)}$ such that if $T_{n}^{(p)}$ is a random triangulation chosen uniformly at random in $ \mathcal{T}_{n,p}$ then the convergence
 \begin{eqnarray} \label{def:uipt} T_{n}^{(p)} \quad   \xrightarrow[n\to\infty]{(d)} \quad  T_{\infty}^{(p)}  \end{eqnarray}holds in distribution for the so-called local distance. The infinite random map  $T_{\infty}^{(p)}$ is called the Uniform Infinite Planar Triangulation (UIPT) of the $p$-gon. In the case $p=1$, by the root-transformation, $T_{\infty}^{(1)}$ can be seen as the standard UIPT of the plane (type I), as was proved by Angel \& Schramm \cite{AS03} for type II triangulations (but the techniques extend to the type I). In the case $p=1$, there is a ``bijective'' construction of the UIPT of type I \cite[Proposition~6.2]{Ste14}. See also \cite{CLGmodif} for a recent construction of the UIPT of type I via its skeleton decomposition \cite{Kri04}.
\paragraph{Rigidity.} If $ \mathbf{h}$ and $ \mathbf{h}'$ are two triangulations with holes, we say that $ \mathbf{h}$ is a sub-triangulation of $ \mathbf{h'}$, and we write $ \mathbf{h} \subset \mathbf{h}'$, if $ \mathbf{h'}$ can be obtained from $\mathbf{h}$  by gluing triangulations with holes along the boundaries of certain holes  of $\mathbf{h}$ (again, recall that gluing the trivial triangulation inside a $2$-gon amounts to just identifying the two edges of the $2$-gon). We shall equivalently use the terms of \emph{gluing} or \emph{filling-in}. {We say that a triangulation with holes $ \mathbf{h}$ is rigid if}  two different ways of filling-in $ \mathbf{h}$ always give rise to two different triangulations with boundaries (see \cite[Definition 4.7]{AS03}). In particular, if $ \mathbf{h}$ and $ \mathbf{h}'$ are two triangulations with holes  {such that $\mathbf{h}$ is rigid and $ \mathbf{h} \subset \mathbf{h}'$}, then $\mathbf{h}'$ is obtained by filling-in  in a unique way certain holes of $ \mathbf{h}$.

\paragraph {Notation.} Without further notice, we work on the canonical space $ \Omega$ of all (possibly infinite) triangulations with holes equipped with the Borel $\sigma$-field for the local topology, and the notation $ \mathbb{P}^{(p)}, \mathbb{E}^{(p)}$ (resp.~$ \mathbb{P}^{(p)}_{\infty}, \mathbb{E}^{(p)}_{\infty}$) are used for the probability and expectation on $\Omega$ relative to the law of a Boltzmann triangulation of the $p$-gon (resp.~of the UIPT of the $p$-gon). Under these measures, the variables will be denoted by $ \mathbf{t}$ or omitted; for instance, if $T^{(p)}$ is a Boltzmann triangulation of the $p$-gon, we have $ \mathbb{E}[\phi(B_{r}(T^{(p)}))] = \mathbb{E}^{(p)}[\phi(B_{r})]$ for every positive measurable function $\phi$.

\subsection{Branching peeling explorations}

\label{sec:bpp}

We now define the branching peeling exploration, which is a means to explore a triangulation with a boundary face after face. If $\mathbf{h}$ is a triangulation with holes, we denote by $ \mathcal{C}(\mathbf{h})$ the union of its cycles.  Formally, a branching peeling exploration depends on a function $ \mathcal{A}$, called the \emph{peeling algorithm}, which associates with any finite triangulation with holes $ \mathbf{h}$  an edge of $ \mathcal{C}(\mathbf{h}) \cup \{\dagger\}$, where $\dagger$ is a cemetery point which we interpret as the desire to stop the exploration. In particular, if $ \mathbf{h}$ has no holes (meaning that $\mathbf{h}$ is a triangulation of the $p$-gon), we must have $ \mathcal{A}( \mathbf{h}) = \dagger$. We say that this peeling algorithm is deterministic since no randomness is involved in the definition of $ \mathcal{A}$. 

Let $ \mathbf{t}$ be a triangulation with a boundary. Intuitively speaking, given the peeling algorithm $ \mathcal{A}$, the branching peeling process of $\mathbf{t}$ is a way to iteratively explore $\mathbf{t}$ starting from its boundary and by discovering at each step a new triangle by \emph{peeling an edge}  determined by the algorithm $ \mathcal{A}$. If $\mathbf{h} \subset \mathbf{t}$ is a triangulation with holes and $e$ is an edge belonging to a cycle  $\mathscr{C}$ of $\mathbf{h}$, the triangulation with holes  $\mathbf{h}_{e}$ obtained by peeling $e$ is defined as follows. Let $ \mathbf{t}({ \mathscr{C}})$ be the triangulation with a boundary glued inside $\mathscr{C}$ in $\mathbf{t}$.
\begin{itemize}
\item \emph{Event $ \mathsf{V}$:}  The triangulation $ \mathbf{t}({ \mathscr{C}})$ is the trivial triangulation (this may only happen  if $p=2$). Then $\mathbf{h}_{e}$ is obtained from $\mathbf{h}$ by closing this cycle (that is by gluing the trivial triangulation in $\mathscr{C}$). See the right-most part of  Fig.~\ref{fig:casespeelb} for an illustration.	
\end{itemize}
Otherwise, let $\triangle_{e}$ be the triangle adjacent to $e$ in  $ \mathbf{t}({ \mathscr{C}})$. Then, roughly speaking, $\mathbf{h}_{e}$ is obtained from $\mathbf{h}$ by ``gluing'' $\triangle_{e}$ along $e$ inside the hole delimitated by $\mathscr{C}$. Specifically, letting $p$ be the perimeter of $\mathscr{C}$, there are two possible cases:
\begin{itemize}
\item \emph{Event $ \mathsf{C}$:} The third vertex of $\triangle_{e}$ does not belong to the cycle $\mathscr{C}$. Then $ \mathbf{h}_{e}$ is defined from $\mathbf{h}$ by gluing a new triangle on $e$. See the left-most part of Fig.~\ref{fig:casespeelb} for an illustration.
\item \emph{Event $ \mathsf{G}_{k}$:} The third vertex $W_{e}$ of $\triangle_{e}$ belongs to the cycle $\mathscr{C}$; let $k \in \{0,1, \ldots,p-1\} $ be the number of edges between  $W_{e}$ and $e$ in clockwise order. Then $ \mathbf{h}_{e}$ is obtained from $ \mathbf{h}$ by gluing a new triangle on $e$ and by \emph{only} identifying  its third vertex with $W_{e}$. See  the middle part of Fig.~\ref{fig:casespeelb} for an illustration.
\end{itemize}

\begin{figure}[!h]
 \begin{center}
 \includegraphics[width=15cm]{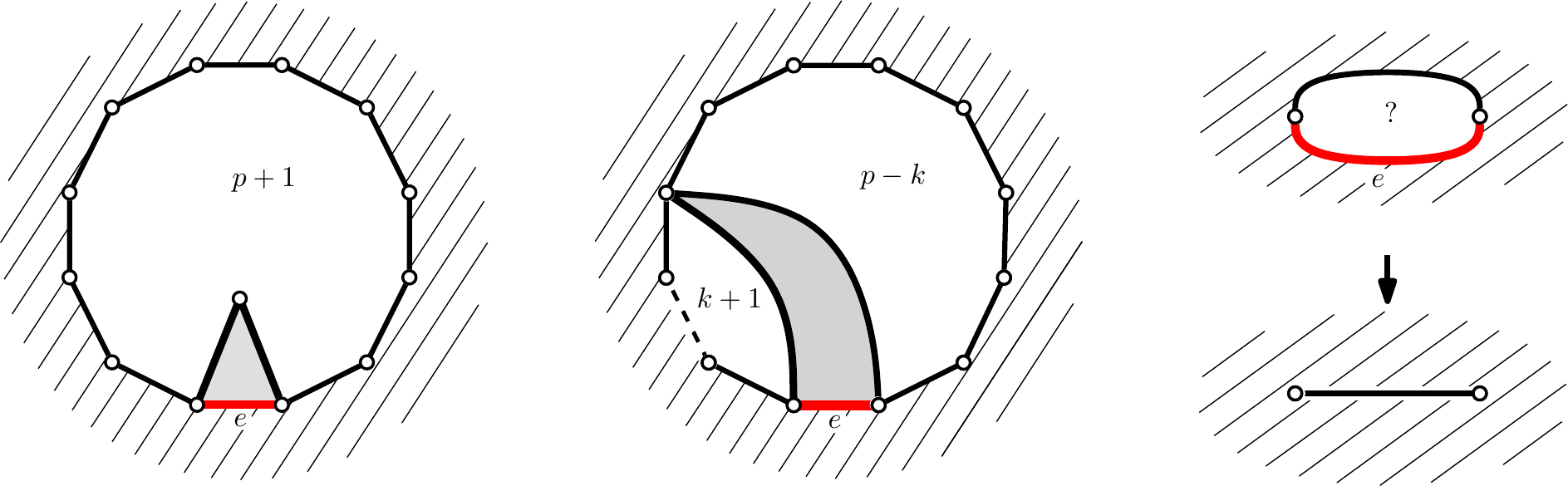}
 \caption{ \label{fig:casespeelb} From left to right: illustration of the three possible events $ \mathsf{C}$, $\mathsf{G}_{k}$ (in this case, $k+1$ and $p-k$ represent the perimeters of the two new cycles) and $\mathsf{V}$. }
 \end{center}
 \end{figure}
 
On the event $\mathsf{G}_{k}$, we insist that the two other edges of the new triangle are \emph{never} glued to an edge of $\mathscr{C}$, so that the cycle $\mathscr{C}$ of length $p$ of $ \mathbf{h}$ is \emph{always} split into two cycles of perimeter $k+1$ and $p-k$. In particular, when $k=0$ or $k=p-1$, one creates a loop, and when $k=1$ or $k=p-2$, one creates a cycle of length $2$ (which may be empty in $\mathbf{t}$), see Fig.~\ref{fig:casespeelb} for an illustration. The reader may have a look at the three peeling steps needed to peel a simple triangle in Fig.~\ref{fig:peeltrig}.
\begin{figure}[!h]
 \begin{center}
 \includegraphics[width=14cm]{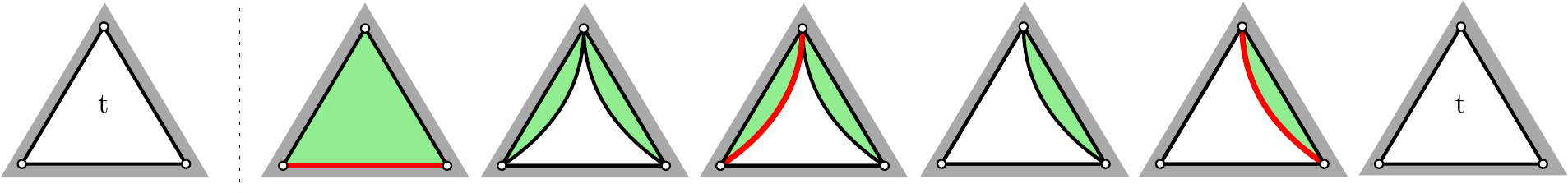}
 \caption{\label{fig:peeltrig}The three peeling steps needed to peel a simple triangle (on the left). The external face is in gray, the unexplored regions are in green and the edge to peel at each step is in red.}
 \end{center}
 \end{figure}
 
Formally, if $\mathbf{t}$ is a (finite or infinite triangulation) with a boundary, the branching peeling exploration of $ \mathbf{t}$ with algorithm $ \mathcal{A}$ is by definition the sequence of triangulations with holes
$$ \mathcal{H}_{0}(\mathbf{t}) \subset  \mathcal{H}_{1}(\mathbf{t}) \subset \cdots \subset \mathcal{H}_{n}(\mathbf{t}) \subset \cdots \subset \mathbf{t},$$ obtained as follows:
\begin{itemize}
\item The triangulation with holes $  \mathcal{H}_{0}(\mathbf{t})$ is made of the boundary of $ \mathbf{t}$, that is a simple path with one oriented edge such that the external face is on its right and that the face on its left is a hole of the same perimeter,
\item for every $ i \geq 0$, if  $  \mathcal{A}(\mathcal{H}_{i}(\mathbf{t})) \ne \dagger$, then the triangulation $ \mathcal{H}_{i+1}(\mathbf{t})$ is obtained from $ \mathcal{H}_{i}(\mathbf{t})$ by peeling the edge $ \mathcal{A}( \mathcal{H}_{i}(\mathbf{t}))$. If $  \mathcal{A}(\mathcal{H}_{i}(\mathbf{t}))= \dagger$, then $ \mathcal{H}_{i+1}(\mathbf{t}) = \mathcal{H}_{i}(\mathbf{t})$ and the exploration process stops.
\end{itemize}

We now fix $p \geq 1$ and a peeling algorithm $ \mathcal{A}$.   Let us make a couple of simple observations. First, if $ \mathbf{t}$ is finite, there exists an integer $N \geq 0$ such that $\mathcal{H}_{n}(\mathbf{t})= \mathcal{H}_{N}(\mathbf{t})$ for every $n \geq N$:  the branching peeling process of $ \mathbf{t}$ will eventually stop. This could happen when the triangulation is completely discovered, or before due to the possible value $\dagger$ given by $ \mathcal{A}$  (indeed, observe that the algorithm that stops immediately, that is $ \mathcal{A}(\cdot)=\dagger$, is a valid peeling algorithm). On the other hand, the branching peeling exploration of $ \mathbf{t}$ may continue forever if $ \mathbf{t}$ is infinite, but not necessarily always. 

 If $i \geq 0$, the triangulation with holes  $ \mathcal{H}_{i}( \mathbf{t})$ is obviously a (deterministic) function of $ \mathbf{t}$. But  note that  $(\mathcal{H}_{j}( \mathbf{t}); 0 \leq j \leq i)$
is also a (deterministic) function of $\mathcal{H}_{i}( \mathbf{t})$. {Indeed, for every $i \geq 1$,   $\mathcal{H}_{i}(\mathbf{t})$ is rigid (see e.g.~\cite[Lemma 4.8]{AS03}). As a consequence,} there is a unique way to fill-in the holes of $ \mathcal{H}_{i}(\mathbf{t})$  to obtain $ \mathbf{t}$. Finally, to simplify notation, we will often write $ \mathcal{H}_{i}$ instead of $ \mathcal{H}_{i}(\mathbf{t})$.

There are obviously many peeling algorithms one can use, but it turns out that branching peeling explorations of Boltzmann triangulations of the $p$-gon and the UIPT of the $p$-gon share several interesting properties, irrespective of the chosen peeling algorithm (as explained in the Introduction, we will later specialize in Section~\ref{sec:branchedpeelinglayers} the peeling algorithm in order to study specific metric properties of Boltzmann triangulations).  

\subsection{Peeling of Boltzmann triangulations with a boundary}
\label{sec:peel}

As before, we fix a deterministic peeling algorithm $ \mathcal{A}$. For every $n \geq 0$, we  denote by $ \mathcal{F}_{n}$ the $\sigma$-field on the set of all triangulations with holes of the $p$-gon  generated by  the mappings $\mathbf{t} \mapsto \mathcal{H}_{0}(\mathbf{t}), \mathcal{H}_{1}(\mathbf{t}) , \ldots, \mathcal{H}_{n}(\mathbf{t})$ (the dependence in $p$ is implicit).  

Recall from Section \ref{sec:enumeration} the constants $C(p)$ and $Z(p)$ for $p \geq 1$.

\begin{proposition} \label{prop:peelinggenerallaw} Fix $p \geq 1$ and $n \geq 0$. Let $ \mathbf{h}$ be a  triangulation with holes such that there exists a triangulation $\mathbf{t}$ of the $p$-gon with $ \mathcal{H}_{n}(\mathbf{t})= \mathbf{h}$. We denote by $\ell_{1}, \ell_{2}, \ldots, \ell_{k}$ the perimeters of the cycles of $ \mathbf{h}$, and by $N$ the numbers of inner vertices of $\mathbf{h}$ (not incident to the external face). Then
\begin{equation}
\label{eq:peellaw}\mathbb{P}^{(p)}( \mathcal{H}_{n} = \mathbf{h}) = \frac{(12 \sqrt{3})^{-N}}{Z(p)} \prod_{i = 1}^k Z( \ell_{i}), \qquad  \mathbb{P}^{(p)}_{\infty}( \mathcal{H}_{n}= \mathbf{h}) =  \frac{(12 \sqrt{3})^{-N}}{C(p)}  \left( \prod_{i = 1}^k Z( \ell_{i}) \right)   \left(  \sum_{j= 1}^k \frac{C(\ell_{j})}{Z(\ell_{j})} \right) .
\end{equation}
Furthermore, under $ \mathbb{P}^{(p)}$ and conditionally on $\{ \mathcal{H}_{n} = \mathbf{h}\}$, the triangulations filling-in the holes of $ \mathbf{h}$ inside $ \mathbf{t}$ are independent Boltzmann triangulations with boundaries. Also, under $ \mathbb{P}^{(p)}_{\infty}$ and conditionally on $\{ \mathcal{H}_{n} = \mathbf{h}\}$,  the triangulations filling-in the holes of $ \mathbf{h}$ inside $ \mathbf{t}$ are independent, all being Boltzmann triangulations  with boundaries, except for the $J$-th hole which is filled-in with a UIPT of the $\ell_{J}$-gon, where the index $J$ is chosen at random, independently and proportionally to $C(\ell_{\cdot})/Z(\ell_{\cdot})$.
\end{proposition}

In the previous statement, we use the conventions $\prod_{ \varnothing}=1$ and $\sum_{ \varnothing}=0$. In particular, if $ \mathbf{h}$ has no holes, then $ \mathbb{P}^{(p)}( \mathcal{H}_{n} = \mathbf{h}) =(12 \sqrt{3})^{-N}/Z(p)$ and  $ \mathbb{P}^{(p)}_{\infty}( \mathcal{H}_{n}= \mathbf{h}) = 0$.

\begin{proof} The proof is standard, see e.g.~\cite[Proposition 4.12]{AS03} and \cite[Theorem 4]{CLGpeeling}. However, since these references deal with slightly different settings, we give the proof for completeness. Since the peeling algorithm is deterministic and because any triangulation with holes {of the form $ \mathcal{H}_{i}(\mathbf{t})$}  is rigid, the event $\{ \mathcal{H}_{n} = \mathbf{h}\}$ happens if and only if $ \mathbf{t}$ is obtained from $ \mathbf{h}$ by filling-in its holes with certain triangulations $ \mathbf{t}_{1},  \mathbf{t}_{2}, \ldots , \mathbf{t}_{k}$ with boundaries of perimeters $ \ell_{1}, \ell_{2}, \ldots , \ell_{k}$ and number of inner vertices $n_{1}, n_{2}, \ldots , n_{k}$. Since the total number of inner vertices of $ \mathbf{t}$ is then $\sum_{i = 1}^{k} n_{i} + N$, we have
$$ \mathbb{P}^{(p)}( \mathcal{H}_{n} = \mathbf{h}) = \frac{1}{Z(p)} \sum_{n_{1}, \ldots , n_{k} \geq 0} \,\sum_{ \mathbf{t}_{i} \in \mathcal{T}_{n_{i}, \ell_{i}}} (12 \sqrt{3})^{- \sum_{i \geq 1} n_{i} - N}=  \frac{(12 \sqrt{3})^{-N}}{Z(p)} \prod_{i = 1}^{k} Z( \ell_{i}).$$

Now, if $T^{(p)}_{m}$ is a uniform triangulation of the $p$-gon with $m$ inner vertices ($m \geq N$) the same argument holds provided that $N + \sum_{i = 1}^{k} n_{i} =m$,  so that 
$$ \mathbb{P}( \mathcal{H}_{n}(T^{(p)}_{m})= \mathbf{h}) =  \frac{1}{\# \mathcal{T}_{m,p}}  \sum_{N+n_{1}+ \cdots  + n_{k} =m } \,  { \prod_{i=1}^{k}\# \mathcal{T}_{n_{i}, \ell_{i}}}.$$
It is an easy matter to
verify that, for any $\epsilon>0$, we can choose $K$ sufficiently large so that the asymptotic contribution of terms
corresponding to choices of $n_{1}, \ldots , n_{k}$ where $n_{i} \geq K$ for two distinct values of $i \in \{1,2, \ldots,k\} $ is bounded
above by $\epsilon$ (see \cite[Lemma 2.5]{AS03}{, \cite[Lemma 3.5]{CK15}}). Hence, using  \eqref{eq:tnpexact}{,the definition of $Z(p)$}  and the convergence \eqref{def:uipt}, we get that $\mathbb{P}^{(p)}_{\infty}( \mathcal{H}_{n} = \mathbf{h}) $ is equal to 
{ 
 \begin{eqnarray*} \lim_{m \to \infty}\mathbb{P}( \mathcal{H}_{n}(T^{(p)}_{m})= \mathbf{h}) &=&  (12 \sqrt{3})^{-N} \cdot \lim_{m \to \infty}  \frac{(12 \sqrt{3})^{m}}{\# \mathcal{T}_{m,p}}   \sum_{n_{1}+ \cdots  + n_{k} =m-N } \,   \prod_{j=1}^{k} (12 \sqrt{3})^{-n_{j}}\# \mathcal{T}_{n_{j}, \ell_{j}} \\
  &\underset{ \eqref{eq:tnpexact}}{=}& \frac{(12 \sqrt{3})^{-N}}{C(p)} \cdot \sum_{j=1}^{k} C(\ell_{j}) \prod_{i \neq j}  \left(  \sum_{n=0}^{\infty}   (12 \sqrt{3})^{-n}\# \mathcal{T}_{n, \ell_{i}} \right) \\
  &=&   \frac{(12 \sqrt{3})^{-N}}{C(p)} \cdot \sum_{j=1}^{k} C(\ell_{j}) \prod_{i \neq j} Z(\ell_{i}).
  \end{eqnarray*}
 This completes the proof.
 }
 \end{proof}

\begin{remark} The above proposition remains true when the peeling algorithm may use a source of randomness, as long as the latter is independent of the underlying random triangulation. Examples of such randomized peeling algorithms that have been used in the literature include peeling along percolation interfaces \cite{Ang03,ACpercopeel}, peeling along random walks \cite{BCsubdiffusive}, or  peeling along SLE$_{6}$ interfaces \cite{CurKPZ}. However, in this work, we focus on deterministic peelings.
\end{remark}

We will use the following extension of Proposition \ref{prop:peelinggenerallaw} at stopping times, where we keep the same notation as in the latter:

\begin{corollary}\label{cor:peelinggenerallaw} Let $\tau$ be a $( \mathcal{F}_{n})_{n \geq 0}$ stopping time and let $\mathbf{h}$ be a triangulation with holes. If $\mathbb{P}^{(p)}( \mathcal{H}_{\tau}= \mathbf{h}, \tau<\infty) >0$, then the first part of \eqref{eq:peellaw} holds when $\mathbb{P}^{(p)}( \mathcal{H}_{n} = \mathbf{h}) $ is  replaced by $\mathbb{P}^{(p)}( \mathcal{H}_{\tau} = \mathbf{h}, \tau<\infty)$. Also, if $\mathbb{P}^{(p)}_{\infty}( \mathcal{H}_{\tau}= \mathbf{h}, \tau<\infty) >0$, then the second part of \eqref{eq:peellaw} holds when $\mathbb{P}^{(p)}_{\infty}( \mathcal{H}_{n} = \mathbf{h}) $ is  replaced by $\mathbb{P}^{(p)}_{\infty}( \mathcal{H}_{\tau} = \mathbf{h}, \tau<\infty)$.
\end{corollary}

\begin{proof} 
Since the peeling algorithm is deterministic, for every $k \geq 0$, there exists a set $ \mathbb{S}_{k}$ of triangulations of the $p$-gon such that $\tau=k$ if and only if $ \mathcal{H}_{k}  \in \mathbb{S}_{k}$. In particular, there exists an integer $k$, depending only on $ \mathbf{h}$ (and the peeling algorithm $ \mathcal{A}$), such that $\{ \mathcal{H}_{\tau} = \mathbf{h}, {\tau<\infty}\} = \{ \mathcal{H}_{k} = \mathbf{h}\}$. It then suffices to apply Proposition \ref{prop:peelinggenerallaw} with this integer. 
 \end{proof}

\paragraph{One-step peeling transitions.} Proposition \ref{prop:peelinggenerallaw} entirely describes the law of a branching peeling process on random Boltzmann triangulations of the $p$-gon. However, it will be useful in the sequel to understand the one-step probability transitions during this peeling process. The proof is an easy consequence of Proposition \ref{prop:peelinggenerallaw} and is left to reader.\label{sec:onestepboltz}

If $n \geq 0$, denote by $ \mathcal{L}_{n}$ the perimeter of the cycle of  $\mathcal{H}_{n}$ to which belongs $ \mathcal{A}(\mathcal{H}_{n})$ (with the convention $ \mathcal{L}_{n}=0$ if $\mathcal{A}(\mathcal{H}_{n})= \dagger$). Then, for every $p \geq 1$ and $n \geq 0$, under $\P^{(p)}$ and conditionally on $ \mathcal{F}_{n}$ and on the event $ \{ \mathcal{A}(\mathcal{H}_{n}) \neq \dagger\}$ (which belongs to $ \mathcal{F}_{n}$), the events $ \mathsf{C}, \mathsf{G}_{k}$ (with $k \in \{0, \ldots,  \mathcal{L}_{n}-1\}$) or $ \mathsf{V}$ occur with the following probabilities:
$$  \mathbb{P}^{(p)}( \mathsf{C} \mid  \mathcal{F}_{n})   = b_{-1}^{( \mathcal{L}_{n})}, \qquad  \mathbb{P}^{(p)}( \mathsf{G}_{k} \mid \mathcal{F}_{n})= b^{( \mathcal{L}_{n})}_{k} , \quad \qquad \mathbb{P}^{(p)}( \mathsf{V} \mid \mathcal{F}_{n})= b^{(2)}_{\emptyset} \mathbbm{1}_{ \mathcal{L}_{n}=2},$$
where, for $m \geq 1$ and $0 \leq k \leq m-1$,
$$b_{-1}^{(m)}    \coloneqq     \frac{1}{12 \sqrt{3}} \frac{Z(m+1)}{Z(m)}, \qquad  b_{k}^{(m)}   \quad := \quad   \frac{Z(k+1)Z(m-k)}{Z(m)}, \qquad b_{\emptyset}^{(2)} \coloneqq  \frac{1}{Z(2)}.$$

Furthermore, as in Proposition \ref{prop:peelinggenerallaw}, conditionally on any of the above cases, the finite triangulations with boundaries that fill-in the new holes created by the peeling step are independent Boltzmann triangulations with boundaries. We will also use the limit of these transitions probabilities as $m \to \infty$:  \begin{eqnarray} \label{def:qk} q_{-1} \quad  :=  \quad  \lim_{m \to \infty} b_{-1}^{(m)} \underset{ \eqref{eq:asympZp}}{=} \frac{1}{ \sqrt{3}}, \quad \mbox{ and } \quad q_{-k}  \coloneqq  \lim_{m \to \infty} b_{k}^{(m)} \underset{ \eqref{eq:asympZp}}{=} 12^{-k} Z(k+1).  \end{eqnarray}
These quantities correspond to the one-step transition probabilities in the Uniform Infinite Half-Planar Triangulation (of type I), see \cite[Section 2.3.1]{ACpercopeel}.

A key element that we will use in the proof of Proposition \ref{prop:scalingllcycle} is that the average ``change of  boundary length'' during a peeling step in the infinite half-plane model is zero (see \cite[Remark after Proposition 3]{ACpercopeel}), that is
  \begin{equation} \label{eq:variation=0}
q_{-1}+ 2 \sum_{k=0}^\infty q_{k}=1, \qquad   q_{-1} - 2 \sum_{k=0}^\infty k q_{k} =0.
  \end{equation}
Finally, we refer to \cite[Section~3.1]{CLGpeeling} the reader interested in knowing the one-step peeling transitions inside the cycle disconnecting the external boundary from infinity in the UIPT.

\subsection{Two martingales}
\label{sec:martingales}

Here we present two useful martingales that appear in any deterministic peeling of a Boltzmann triangulation with a boundary. Roughly speaking, the first one, called the \emph{volume martingale} involves the sum of the squares {of the lengths} of the cycles and is the expected size of the full triangulation conditionally given the current stage of the peeling exploration. However, the second martingale, called the \emph{cycle martingale}, involves the sum of the cubes {of the lengths} of the cycles and has a less obvious geometric meaning. We mention that the cycle martingale has already appeared in \cite[Theorem 4]{CLGpeeling} for a specific peeling algorithm. In this work, we will use this martingale to control the $ \ell^{3}$ norm of the cycles appearing in branching peeling explorations of Boltzmann triangulations, which, in turn, will be later useful for the cutoff argument.

 In the sequel, we fix $p \geq 1$ and recall that for every $n \geq 0$, $\mathcal{F}_{n}$ is the filtration generated by $\mathcal{H}_{0}, \ldots , \mathcal{H}_{n}$ on the set of all triangulations with holes of the $p$-gon. For every $n \geq 0$, we let $\boldsymbol{\ell}({n})= (\ell_{1}(n), \ldots , \ell_{i} (n), \ldots )$ be the perimeters of the cycles of $ \mathcal{H}_{n}$ enumerated in a deterministic fashion (here and after the dependence in $p$ is implicit).  Note that $ \boldsymbol{\ell}(0)=p$.
 
\paragraph{The volume martingale.} 

Set $$g(1)=1+ \frac{2}{\sqrt{3}}, \qquad  g(p) =  \frac{1}{3}(2p-3)(2p-1), \qquad  p \geq 2.$$If $ \mathbf{h}$ is a triangulation with holes, we denote by $ | \mathbf{h}|$ the number of inner vertices of $ \mathbf{h}$ (that do not belong to the external boundary but may belong to cycles). It turns out that $g(p)$ is the expected number of internal vertices in a Boltzmann triangulation of the $p$-gon:

\begin{proposition}\label{prop:g} For every $p \geq 1$, we have
$g(p) =  \Esp{| \mathbf{t}|}$.
\end{proposition}

\begin{proof}We apply results of Krikun \cite{Kri07}. Set $W(x,y)= \sum_{p,n \geq 0}  \# \mathcal{T}_{n,p} x^{n} y^{p}$ and $W_{p}(x)= [y^{p}] W(x,y)$, so that $\Esp{| \mathbf{t}|}=x W'_{p}(x)/W_{p}(x) \big|_{x=r_{c}}$ with $r_{c}=(12 \sqrt{3})^{-1}$. In particular, $W_{p}(r_{c})=Z(p)$. Using the notation of \cite[Section~2.1]{Kri07} (Krikun uses the number of edges as size parameter; to translate his formulas we use the fact that if a triangulation of the
$p$-gon has $n$ inner vertices then by Euler's formula it has $3n + 2p - 3$ edges), we have $W(x^{3},y)=x^{3}U_{0}(x,y/x^{2})$. Then, letting $h=h(x)$ be the positive power series such that $8h^{3} x^{2}-h^{2}+x^{2}=0$, by the display between (19) and (20) in \cite{Kri07} (note the sign error in \cite{Kri07} for $W_{1}$) we have
$$W_{1}(x^{3})= \frac{1}{2}- \frac{1+2h^{3}}{2h} x, \qquad W_{p+2}(x^{3})= \frac{1}{x^{2p+1}} \cdot \frac{1}{p+1} \binom{2p}{p} \left( 1- \frac{4p+2}{p+2} h^{3} \right) h^{2p+1} \qquad  (p \geq 0).$$
Note that  $\Esp{| \mathbf{t}|}= \left(\frac{d}{dx} W_{p}(x^{3}) \right)   \big|_{x=r_{c}^{1/3}} \cdot {r_{c}^{1/3}}/(3 Z(p))$. The result then readily follows, by using the fact that $h(r_{c}^{1/3})=2^{-2/3}$ and that $h'(x)(4h(x)^{3}-1) \rightarrow  -3\sqrt{3}$ as $x \rightarrow r^{1/3}_{c}$ (this can for instance be seen by noting that by \cite[Eq.~(6)]{Kri07}, $[x^{3k+1}] h \sim (2 \pi)^{-1/2} \cdot k^{-3/2} \cdot r_{c}^{-k}$ as $k \rightarrow \infty$ and applying Tauberian theorems, which yield $(h(x)-h(r_{c}^{1/3}))/({r_{c}^{1/3}-x})^{1/2} \rightarrow -2^{-1/6}\cdot 3^{1/4}$ and $h'(x) \cdot (r_{c}^{1/3}-x)^{1/2} \rightarrow 2^{-5/6} \cdot 3^{1/4}$ as $x \rightarrow  r_{c}^{1/3}$). We leave the details to the reader.
\end{proof}

\begin{proposition} \label{prop:martingalel2} Under $ \mathbb{P}^{(p)}$, the process $(V_{n})_{n \geq 0}$ defined by $$V_{n} = | \mathcal{H}_{n}| + \sum_{i\geq 1} g\big( \ell_{i}(n)\big), \qquad  n \geq 0,$$ is a nonnegative  $( \mathcal{F}_{n})_{n \geq 0}$ uniformly integrable martingale with $V_{0}=g(p)$.
\end{proposition}

\begin{proof} 
Using Proposition \ref{prop:g}, simply observe that $V_{n}=\Esp{ | \mathbf{t}|  \mid \mathcal{F}_{n}}$ for every $n \geq 0$. Indeed, the description of the branching peeling process in Section~\ref{sec:bpp} shows that $\Esp{ | \mathbf{t}|  \mid \mathcal{F}_{n}}$ is equal to $ |\mathcal{H}_{n}|$ plus the sum of the expected values, conditionally given $\mathcal{H}_{n}$, of the number of inner vertices present in each of the holes of $ \mathcal{H}_{n}$, which is exactly $\sum_{i \geq 1} g( \ell_{i}(n))$. It follows in particular that $V_{n}$ is a uniformly integrable martingale. \end{proof}

\paragraph{The cycle martingale.}
Recall the definition of $C(\cdot)$ and $Z(\cdot)$ from Section~\ref{sec:enumeration} and set:
 \begin{eqnarray} \label{def:f} f(1):= \frac{C(1)}{Z(1)}= \frac{\sqrt{2}(2+\sqrt{3})}{3 \sqrt{\pi}} \qquad  \textrm{and} \qquad  f(p) := \frac{C(p)}{Z(p)} =  \frac{\sqrt{6}}{9\sqrt{\pi}} \cdot p(2p-1)(2p-3), \qquad  p \geq 2,  \end{eqnarray}
with $f(0)=C(0)/Z(0)=0$ by convention.

\begin{proposition} \label{prop:martingale}  Under $ \mathbb{P}^{(p)}$, the process $(M_{n})_{n \geq 0}$ defined by 
\begin{equation}
\label{eq:M}M_{n} = \sum_{i\geq 1} f\big( \ell_{i}{(n)}\big), \qquad  n \geq 0,
\end{equation} is a nonnegative $( \mathcal{F}_{n})_{n \geq 0}$ martingale called the \emph{cycle martingale} with $M_{0}=f(p)$.

\end{proposition}

\begin{proof} By Proposition \ref{prop:peelinggenerallaw}, we see that $M_{n}$ is $f(p)$ times the Radon-Nikodym derivative of the law of $ \mathcal{H}_{n}$ under $\mathbb{P}^{(p)}_{\infty}$ with respect to the law of $ \mathcal{H}_{n}$ under $\mathbb{P}^{(p)}$: 
 \begin{eqnarray} \label{eq:radonnikodym} M_{n} =  {f(p)} \cdot \frac{\mathbb{P}^{(p)}_{\infty} (  \mathcal{H}_{n} = \mathbf{h})}{ \mathbb{P}^{(p)} (  \mathcal{H}_{n} = \mathbf{h})}
 \qquad \hbox{on the event } \{ \mathcal{H}_{n} = \mathbf{h}\}. \end{eqnarray}
 
Since $ \mathcal{H}_{n}$ may be recovered in a deterministic way from $ \mathcal{H}_{n+1}$, this entails that $(M_{n})_{n \geq 0}$ is a $( \mathcal{F}_{n})_{n \geq 0}$ martingale. Indeed, fix $n \geq 0$ and observe that since the peeling algorithm $ \mathcal{A}$ is deterministic, for every triangulation with holes $\mathbf{h}_{n+1}$, there exist triangulations with holes $\mathbf{h}_{0},\mathbf{h}_{1}, \ldots,  \mathbf{h}_{n}$ such that for every triangulation with a boundary $\mathbf{t}$, $\mathcal{H}_{n+1}(\mathbf{t})=\mathbf{h}_{n+1}$ if and only if $ \mathcal{H}_{i}(\mathbf{t}) = \mathbf{h}_{i}$ for ever $0 \leq i \leq n$. In particular,  if $\phi$ is a nonnegative measurable function on the space of triangulations with holes, there exists another nonnegative measurable function $\psi$ such that $\psi( \mathcal{H}_{n+1}(\mathbf{t}))=\phi( \mathcal{H}_{n}(\mathbf{t}))$ for every triangulation $\mathbf{t}$ with a boundary. As a consequence, we have $$ \mathbb{E}^{(p)}[\phi( \mathcal{H}_{n}) \cdot M_{n+1}]  = \mathbb{E}^{(p)}[\psi( \mathcal{H}_{n+1}) \cdot M_{n+1}] \underset{\eqref{eq:radonnikodym}}{=} {f(p)} \mathbb{E}^{(p)}_{\infty}[\psi( \mathcal{H}_{n+1})] = {f(p)} \mathbb{E}^{(p)}_{\infty}[\phi( \mathcal{H}_{n})]  \underset{\eqref{eq:radonnikodym}}{=} \mathbb{E}^{(p)}[ \phi( \mathcal{H}_{n}) \cdot M_{n}],$$
so that $ \Es{ M_{n+1} | \mathcal{F}_{n}}= M_{n}$.
\end{proof} 

\noindent{\bf Remark.} The local absolute continuity \eqref{eq:radonnikodym} of the law of the UIPT with respect to Boltzmann triangulation can also be interpreted as follows.
Denote by ${\mathscr H}_{n,p}$  the space of triangulations with holes of the $p$-gon which arise after $n$ steps of peeling, and then ${\mathscr H}^*_{n,p}$
for the space of pairs $(\mathbf{h},c)$ with $\mathbf{h}\in{\mathscr H}_{n,p}$ and $c$ a cycle of $\mathbf{h}$. We think of $(\mathbf{h},c)$ as a triangulation with holes having one marked cycle.
We then set 
$${\mathbb Q}_{n,p}((\mathbf{h},c))\coloneqq \frac{f(|c|)}{f(p)}\sum_i \P^{(p)}(\mathcal{H}_{n}=\mathbf{h}, {\mathcal C}_{n,i}=c),  \qquad(\mathbf{h},c)\in{\mathscr H}^*_{n,p},$$
where  $\{{\mathcal C}_{n,i}: i=1, \ldots\}$ denotes the family of the cycles of $\mathcal{H}_{n}$. 
Because $\E^{(p)}[M_n]=f(p)$, it follows that ${\mathbb Q}_{n,p}$ defines a probability measure on ${\mathscr H}^*_{n,p}$, and then \eqref{eq:radonnikodym} enables us to identify 
${\mathbb Q}_{n,p}$ as the distribution of the triangulation with holes obtained after $n$ steps of peeling under $\P^{(p)}_{\infty}$ (that is for the UIPT), where the marked cycle is the one corresponding to the infinite end of the UIPT. This observation, together with the fact that $M_n$ is a martingale under $\mathbb{P}^{(p)}$, and the description of the filled-in holes for the UIPT  in Proposition \ref{prop:peelinggenerallaw}
are close relatives to the famous spine decomposition for branching processes; see \cite{LPP}. 

\bigskip

Let us draw a couple of important facts using this cycle martingale. First notice that this cycle martingale is \emph{not} necessarily uniformly integrable.  Indeed, consider a peeling algorithm $ \mathcal{A}$ such that $ \mathcal{A}(\mathbf{h}) \neq \dagger$ if $\mathbf{h}$ has at least one hole, so that $ \mathcal{H}_{n}(\mathbf{t})=\mathbf{t}$ for every $n$ sufficiently large if $\mathbf{t}$ is a finite triangulation of the $p$-gon. Hence, under $ \mathbb{P}^{(p)}$, we have $ \lim_{n \to \infty} M_{n} = 0$ almost surely, so that in particular  $(M_{n})_{n \geq 0}$ is not uniformly integrable.

 In the sequel, we shall need to calculate the expectation of the cycle martingale evaluated at certain (unbounded) stopping times, which will typically be of the form $\min \{k \geq 0; \mathcal{A}( \mathcal{H}_{k})=\dagger  \}$. Specifically, let $\tau$ be a $( \mathcal{F}_{n})_{n \geq 0}$ stopping time taking values in $\{0,1,2,\ldots\} \cup \{ \infty\}$ which is almost surely finite under $\mathbb{P}^{(p)}$. Using  Corollary \ref{cor:peelinggenerallaw} and its proof, write 
  \begin{eqnarray}\mathbb{E}^{(p)}[ M_{\tau}] = \sum_{k=0}^\infty \mathbb{E}^{(p)}[M_{k} \cdot \mathbbm{1}_{ \{\tau=k\} }]=   \sum_{k=0}^\infty f(p) \mathbb{E}^{(p)}_{\infty}[\mathbbm{1}_{ \{\tau=k\} }] = f(p) \mathbb{P}^{(p)}_{\infty} (\tau < \infty). \label{eq:stopping}\end{eqnarray}

Let us give two simple examples that illustrate \eqref{eq:stopping}. First, if we take $\tau$ to be a bounded stopping time, then clearly $ \mathbb{E}^{(p)}[M_{\tau}]= \mathbb{E}^{(p)}[M_{0}]=f(p)$ by the optional stopping theorem, and on the other hand $ \mathbb{P}^{(p)}_{\infty}( \tau < \infty) =1$.
Second, consider again the peeling algorithm $ \mathcal{A}$ such that $ \mathcal{A}(\mathbf{h}) \neq \dagger$ if $\mathbf{h}$ has at least one hole, and let $\tau(\mathbf{t}) = \inf \{n \geq 0; \mathcal{H}_{n}(\mathbf{t})= \mathbf{t} \}$ be the first time when $\mathbf{t}$ is completely discovered, with the convention $ \inf \emptyset = + \infty$. As above, under $ \mathbb{P}^{(p)}$, $ \tau$ is almost surely finite and $M_{\tau}=0$, so that  $\mathbb{E}^{(p)}[M_{\tau}]=0$. On the other hand $ \mathbb{P}^{(p)}_{\infty}( \tau < \infty)=0$ since a branching peeling exploration never  completely discovers the UIPT which is infinite.

\begin{corollary} \label{cor:stopping} Let $(\tau_{n})_{n \geq 0}$ be an increasing sequence of  $( \mathcal{F}_{n})_{n \geq 0}$ stopping times, which are all almost surely finite under $ \mathbb{P}^{(p)}$ as well as under  $ \mathbb{P}^{(p)}_{\infty}$. Then the process $(M_{\tau_{n}})_{n \geq 0}$ is a $ (\mathcal{F}_{\tau_{n}})_{n \geq 0}$ martingale under $ \mathbb{P}^{(p)}$. 
\end{corollary}

 \begin{proof}  Set $M^{(n)}_{k}=M_{k \wedge \tau_{n}}$ for $k \geq 0$. As $k \to \infty$, $M^{(n)}_{k} \rightarrow M_{\tau_{n}}$ almost surely.  By our assumption and  \eqref{eq:stopping}, we get that $ \Espb{M^{(n)}_{k} }=\Espb{M_{0}}= \Espb{M_{\tau_{n}}}$ for every $k \geq 0$. Therefore, since we are dealing with non-negative martingales, by Scheffé's lemma, $M^{(n)}_{k} \rightarrow M_{\tau_{n}}$ in $\mathbb{L}_{1}$, so that the martingale $(M^{(n)}_{k})_{k \geq 0}$ is uniformly integrable. As a consequence, $ \Espb{M_{\tau_{n}} | \mathcal{F}_{\tau_{m}} }= M_{\tau_{m}}$ for $m \leq n$ and the proof is complete.
\end{proof}

\subsection{Scaling limit for the locally largest cycle}
\label{sec:llc}

In a branching peeling exploration of the UIPT, one cycle naturally plays a distinguished role, namely the boundary of the unique hole containing an infinite triangulation. The peeling transitions along this distinguished cycle have been studied in great details in \cite{CLGpeeling} and different scaling limit results have been established. However, in the case of Boltzmann triangulation there is a priori no distinguished cycle to track during a branching peeling exploration. Nonetheless, we can still follow the evolution of a singled out  cycle by deciding to track at each peeling step the locally largest cycle. 

More precisely, the initial distinguished cycle $\mathscr{C}^{\ast}(0)$ is the only cycle of $ \mathcal{H}_{0}$ and $ \sigma_{0}=0$.
Then, inductively, for $i \geq 0$, if $\mathscr{C}^{\ast}(i)=\dagger$ (the cemetery point), set $\mathscr{C}^{\ast}(i+1)=\dagger$, and otherwise define $\sigma_{i+1}= \inf \{ j >\sigma_{i} ; \mathcal{A}( \mathcal{H}_{j}) \in \mathscr{C}^{\ast}(i) \}$ (with the usual convention $\inf \emptyset =\infty$). If $\sigma_{i+1}=\infty$, we define $\mathscr{C}^{\ast}(i+1)=\mathscr{C}^{\ast}(i)$. Otherwise, when peeling the edge $\mathcal{A}( \mathcal{H}_{\sigma_{i+1}})$, we define $\mathscr{C}^{\ast}(i+1)$ depending on what peeling event happens:
\begin{itemize}
\item If the event $\mathsf{V}$ occurs, we define $\mathscr{C}^{\ast}(i+1)=\dagger$,
\item If the event $\mathsf{C}$ occurs, we define  $\mathscr{C}^{\ast}(i+1)$ to be the new cycle thus created,
\item If the event $ \mathsf{G}_{k}$ occurs, one creates two new cycles when peeling the edge $\mathcal{A}( \mathcal{H}_{\sigma_{i+1}})$. We define $ \mathscr{C}^{\ast}(i+1)$ to be the cycle with largest perimeter  (if $\mathscr{C}^{\ast}(i)$ is a cycle of odd length  which is split into two cycles of equals lengths, we choose between the two in a deterministic way).
\end{itemize}

 The cycles $(\mathscr{C}^{\ast}(i))_{i \geq 0}$ are called the locally largest cycles for the algorithm $ \mathcal{A}$. Finally, we agree by convention that the perimeter of $\dagger$ is $0$. By the description in Section~\ref{sec:peel} of the one-step peeling transitions, under $ \mathbb{P}^{(p)}$, conditionally on the event $ \{ \exists i \geq 0; \mathscr{C}^{\ast}(i)=\dagger \}$ that we assume to have positive probability (which is always the case if $ \mathcal{A}(\mathbf{h}) \neq \dagger$ when $\mathbf{h}$ has at least one hole),  the law of the perimeters of $\mathscr{C}^{\ast}(0), \mathscr{C}^{\ast}(1), \ldots$ is a  Markov chain on the nonnegative integers, started at $p$, absorbed at zero and with the following probability transitions:
\begin{equation}
\label{eq:trans} \mathbf{b}(p,p-k) = 2b_{k}^{(p)} \mbox{ for } 0 \leq k < \frac{p}{2}, \quad  \mathbf{b}(2p+1,p) = b_{p}^{(2p+1)}, \quad  \mathbf{b}(p,p+1) = b_{-1}^{(p)}, \quad  \mathbf{b}(2,0) = b_{\varnothing}^{(2)},
\end{equation}
and $ \mathbf{b}(p,k)=0$ otherwise. Recall from  Section~\ref{sec:onestepboltz} the explicit expression of $b^{(p)}$. Let $({{\widetilde L}}^{(p)}({k}))_{k \geq 0}$ be a Markov chain starting from ${{\widetilde L}}^{(p)}({0})=p$ and with these probability transitions, so that we can think of ${{\widetilde L}}^{(p)}({k})$ as the length of 
$\mathscr{C}^{\ast}(k)$. The general machinery developed in \cite{BK14} enables us to identify the scaling limit of this Markov chain. In order to describe it, we first introduce some background.

Let $\nu$  be the measure on $\R$ with density
$$ \nu(\d x)= (x(1-x))^{-5/2} \mathbbm {1}_{ \{1/2 \leq x \leq 1\}} \d x$$
and let $ \Pi$ be the push-forward of $ \nu$ by the mapping $x \mapsto \ln(x)$. Note that $ \Pi$ is supported on  $ [- \ln(2),0]$ and that
$ \int  x^{2} \Pi(\d x) < \infty$. Recall from the Introduction that  $ (\xi(t))_{t \geq 0}$ is a Lévy process with characteristic  exponent $ \Phi(\lambda)= \Psi(i\lambda)$
given by the Lévy--Khinchin formula 
$$ \Phi(\lambda)=  - \frac{8}{3} i \lambda+ \int_{- \ln(2)}^{0} \left( \e^{i\lambda x}-1+i \lambda (1-\e^x)  \right)  \ \Pi(\d x), \qquad \lambda \in \R.$$
Specifically, there is the identity $ \Es {\e^{ i \lambda \xi(t)}}=\e^{t \Phi( \lambda)}$ for $ t \geq 0, \lambda \in \R$. In the literature, the Lévy--Khinchin formula is usually written with the term $1-\e^x$ replaced by $-x$, but this is essentially irrelevant  since it only changes the factor in front of $i \lambda$. We use this version to be consistent with the notation of \cite{BeGF}. Then, for $ \alpha<0$, set $$I^{(\alpha)}_{ \infty}= \int_{0}^{\infty} \e^{ -\alpha \xi(s)} \ \d s \quad \in \quad (0, \infty].$$

{Note that
\begin{equation}
\label{eq:drift}\Psi'(0)=- \frac{8}{3}+ \int_{1/2}^{1} \frac{1-x+\log(x)}{(x(1-x))^{5/2}} \mathrm{d}x =  - \frac{8}{3}+ \frac{8}{9} \left( 6 \pi-18 \right)  <0
\end{equation}
so that  $\xi$ drifts to $- \infty$. In particular,} we have $I^{(\alpha)}_{\infty}<\infty$ almost surely by \cite[Theorem 1]{BY05}. Then for every $t \geq 0$, set
$$\tau^{(\alpha)}(t)= \inf  \left\{ u \geq 0 ; \int_{0}^{u}  \e^{ - \alpha \xi(s)}\d s>t \right\} $$
with the usual convention  $ \inf \emptyset = \infty$ for $t\geq I^{(\alpha)}_{\infty}$. Finally,  by using  the Lamperti transform  \cite {Lam72} of $ \xi$, define  $ {{\widetilde X}}$ to be the self-similar process of index $-3/2$ driven by $\xi$:
\begin{equation}
\label{eq:defX}{{\widetilde X}}(t)= \exp\big({ \xi(\tau^{(-3/2)}(t))}\big) \quad \textrm { for }  \quad 0 \leq t <  I^{(-3/2)}_{ \infty}, \qquad  \qquad {{\widetilde X}}(t)= 0 \quad \textrm { for }  \quad t \geq I^{(-3/2)}_{\infty}.
\end{equation}

Note that the process  $ {{X}}$ introduced in \eqref{eq:defY} is the self-similar Markov process  driven by $ \xi$ but corresponding to the  index $\alpha=-1/2$. In this direction, we point out that in turn, the two are related by another time-change, which, for the sake of simplicity, we describe implicitly as follows:
\begin{equation}\label{eq:tildextox}
{\widetilde X}(t)=X\left(\int_0^t\frac{\d s }{{\widetilde X}(s)}\right)\,,\qquad t\geq 0.
\end{equation}

We are now in position to prove the following invariance principle for the Markov chain ${{\widetilde L}}^{(p)}$.  Denote by $\mathbb{D}(\mathbb{R}_{+},\mathbb{R})$ the space of real-valued càdlàg functions on $\R_{+}$ equipped with the $J_{1}$ Skorokhod topology, and recall the notation  $\tone$ from \eqref{eq:asympZp}. 
\begin{proposition}[Scaling limit for the {locally} largest cycle] \label{prop:scalingllcycle} The convergence 
\begin{equation}
\label{eq:scaling}\left( \frac{1}{p} {{\widetilde L}}^{(p)}({[p^{3/2}t]}); t \geq 0\right) \quad  \xrightarrow[p\to\infty]{(d)} \quad  ({{\widetilde X}}( 2 \tone t) ; t \geq 0)
\end{equation} 
holds in distribution in $\mathbb{D}(\mathbb{R}_{+},\mathbb{R})$.
\end{proposition}
\begin{remark} In addition to the convergence of the last proposition the results of \cite{BK14} show that there is also convergence of the absorption times. More precisely if $\widetilde{\sigma}^{(p)}$ is the first time when $\widetilde{L}^{(p)}$ touches $0$ then we have $p^{-3/2} \widetilde{\sigma}^{(p)} \to \widetilde{\sigma}$ in distribution, where $\widetilde{\sigma}$ is the hitting time of $0$ by $\widetilde{X}( 2 \tone \cdot)$. 
We will not use this in the sequel; however it will be argued in the proof of the forthcoming Lemma \ref{Lj1} that the convergence stated in Proposition \ref{prop:scalingllcycle} further holds in distribution in $\mathbb{D}([0,\infty],\mathbb{R})$, where $[0,\infty]$ is the compactification of $[0,\infty[= \mathbb{R}_{+}$. This fact would also immediately follow from the convergence of the absorption times.
\end{remark}

In order to apply \cite[Theorems 3 \& 4]{BK14}, we need some preparatory notation and technical lemmas. For $p \geq 1$, let  $ \Pi^{(p)}$ be the law of $ \ln({{\widetilde L}}^{(p)}({1})/p)$.
 
 \begin {lemma}\label {lem:techn} \
\begin{enumerate}
\item[(i)] If $F: \R \rightarrow \R_{+}$ is a continuous function with compact support such that $F(x) = \mathcal{O}(x^{2})$ as $ x \rightarrow 0$, then
$$p^{3/2} \cdot  \int_{\R} F(x)  \ \Pi^{(p)}(\d x)  \quad\mathop{\longrightarrow}_{p \rightarrow \infty} \quad  2 \tone \cdot \int_{\R} F(x) \ \Pi(\d x).$$
 \item[(ii)] We have
$$ p^{3/2} \cdot  \int_{-1}^{1}x \ \Pi^{(p)}(\d x)  \quad\mathop{\longrightarrow}_{p \rightarrow \infty} \quad  -\tone \frac{8 (3 \pi-7)}{9}.$$
\end{enumerate}
\end {lemma}

\begin {proof}For (i), first note that \eqref{eq:asympZp} readily entails that for fixed $x \in (1/2,1)$, if $k_{p} \sim x p$ as $p \rightarrow \infty$, then
\begin{equation}
\label{eq:equiv1}\mathbf{b}(p,k_{p})  \quad\mathop{\sim}_{p \rightarrow \infty} \quad  \frac{1}{p^{5/2}} \cdot 2 \mathsf{t}_{\one} \cdot (x(1-x))^{-5/2}.
\end{equation}
In addition, there exists a constant $C_{1}>0$ such that
\begin{equation}
\label{eq:borne1}\textrm {for every } p \geq 2, \quad \textrm {for every } \frac{p}{2} <k < p, \qquad \mathbf{b}(p,k) \leq C_{1} \cdot   \frac{p^{5/2}}{ (k(p-k))^{5/2}}.
\end{equation}
Also observe that $ \mathbf{b}(p,p)=2Z(1)$ is constant. Now write
$$ p^{3/2} \cdot \int_{\R} F(x) \Pi^{(p)}(\d x) = p^{3/2} \cdot F \left(  \ln \left( 1+ \frac{1}{p} \right)  \right)  \mathbf{b}(p,p+1)+ p^{3/2} \cdot  \sum_{k>p/2}^{p-1} F \left(  \ln \left(   \frac{k}{p} \right)  \right)  \mathbf{b}(n,k)+o(1),$$
where $o(1)$ is a quantity tending to $0$ as $p \rightarrow \infty$ capturing the term $\mathbf{b}(p,(p-1)/2)$ when $p$ is odd. The assumption on $F$ yields that {the first term of the sum in the right-hand side} tends to $0$ as $n \rightarrow \infty$. As for the second one, by a change of variables, write 
$$ p^{3/2} \cdot \sum_{k>p/2}^{p-1} F \left(  \ln \left(   \frac{k}{p} \right)  \right)  \mathbf{b}(p,k)=  p^{5/2} \cdot \int_{1/2}^{1} du \ F \left( \ln \left(    \frac{\fl {n u}}{n} \right)  \right)  \mathbf{b}(p, \fl {pu})+o(1),$$
where $o(1)$ is a quantity capturing the boundary terms.
Set $G_{p}(u)=p^{5/2} \cdot \ F ( \ln (    {\fl {n u}}/{n} )  )  \mathbf{b}(p, \fl {pu}) $ for $1/2 \leq u < 1$. The assumption on $F$ and \eqref{eq:borne1} yield the existence of a constant $C>0$ such that   
$0 \leq  G_{p}(u) \leq C (1-u)^{-1/2}$ for every $p \geq 2$ and $1/2 <u <1$. In addition, by \eqref{eq:equiv1}, for every fixed $1/2 < u <1$, $G_{p}(u) \rightarrow F(  \ln(u)) \nu(u)$ as $p \rightarrow \infty$. Assertion (i) then follows from the dominated convergence theorem.

For (ii), write 
$$ p^{3/2} \cdot  \int_{-1}^{1}x \ \Pi^{(p)}(\d x)= p^{3/2} \sum_{k>p/2}^{p+1} \left(  \ln \left(  \frac{k}{p} \right)  - \frac{k}{p} +1\right) \mathbf{b}(p,k)+ p^{3/2} \sum_{k>p/2}^{p+1} \left( \frac{k}{p}-1 \right) \mathbf{b}(p,k)+o(1).$$
The first assertion gives us that
\begin{equation}
\label{eq:cv1}p^{3/2} \sum_{k>p/2}^{p+1} \left(  \ln \left(  \frac{k}{p} \right)  - \frac{k}{p} +1\right) \mathbf{b}(p,k)  \quad\mathop{\longrightarrow}_{p \rightarrow \infty} \quad   \int_{1/2}^{1} ( \ln(u)-u+1) \nu_{\one}(u) du=  \tone \cdot \frac{16(17-6 \pi )}{9}.
\end{equation}
We next claim that
\begin{equation}
\label{eq:cv2}p^{3/2} \sum_{k>p/2}^{p+1} \left( \frac{k}{p}-1 \right) \mathbf{b}(p,k)  \quad\mathop{\longrightarrow}_{p \rightarrow \infty} \quad  - \frac{16}{3} \tone.
\end{equation}
Assertion (ii) will then readily follow by summing \eqref{eq:cv1} and \eqref{eq:cv2}. To establish \eqref{eq:cv2}, first write
$$p^{3/2} \sum_{k>p/2}^{p+1} \left( \frac{k}{p}-1 \right) \mathbf{b}(p,k)= \sqrt{p} \left( b_{-1}^{(p)}- 2  \sum_{1 \leq k < p/2} k  b_k^{(p)} \right).$$
Recall the definition of $q_{-1}$ and $q_{k}$ from \eqref{def:qk}. Using \eqref{eq:variation=0} we may write
$$b_{-1}^{(p)}- 2  \sum_{1 \leq k < p/2} k  b_k^{(p)} = (b_{-1}^{(p)}-q_{-1})  - 2  \sum_{1 \leq k < p/2} k  (b_k^{(p)}-q_k)+ 2 \sum_{k \geq p/2} k q_k.$$
We now estimate the three terms of the right-hand side of the last equality as $p \rightarrow \infty$. First, for $p \geq 1$,
$$b_{-1}^{(p)}-q_{-1}=-\frac{5}{2 \sqrt{3} (p+1)} = o\left(\frac{1}{\sqrt{p}}\right).$$
Next, since $q_k \sim \tone \cdot k^{-5/2}$ as $k \to \infty$, we have 
$$2 \sum_{k \geq p/2} k q_k \quad \mathop{\sim}_{p \rightarrow \infty } \quad \tone \cdot \frac{4 \sqrt{2}}{\sqrt{p}}.$$
Finally, it is a simple matter to check that $f : k \mapsto |b_k^{(p)}/q_k-1 | \cdot p/k=| 12^k Z(p-k)/Z(p)-1| \cdot p/k $ is increasing in $k$ on $\llbracket 1,p-1 \rrbracket$, and that $f(\fl{p/2})$ converges to a positive constant as $p \rightarrow  \infty$. It follows that there exists a constant $C>0$ such that for every $p \geq 2$ and $x \in (0,1/2]$, $|b_{\fl{px}}^{(p)}/q_{\fl{px}}-1| \leq  C x$. In addition, by \eqref{eq:asympZp}, for every fixed $x \in (0,1/2)$, $b_{\fl{px}}^{(p)}/q_{\fl{px}} \rightarrow (1-x)^{-5/2}$ as $p \rightarrow  \infty$.
Hence, writing 
$$\sqrt{p} \cdot \sum_{1 \leq k < p/2} k  (b_k^{(p)}-q_k)= \sqrt{p} \cdot\sum_{1 \leq k < p/2} k q_k  \left(\frac{b_k^{(p)}}{q_k}-1\right) =  \int_{1/p}^{1/2} \d x \ p^{3/2} \fl{px} q_{\fl{px}} \cdot \left(\frac{b_{\fl{px}}^{(p)}}{q_{\fl{px}}}-1\right),$$
the dominated convergence theorem yields that
$$\sqrt{p} \cdot \sum_{1 \leq k < p/2} k  (b_k^{(p)}-q_k)  \quad \mathop{\longrightarrow}_{p \rightarrow \infty} \quad  \tone \cdot \int_{0}^{1/2} \d x \cdot \frac{1-(1-x)^{5/2}}{x^{3/2} \cdot  (1-x)^{5/2}}= \tone \cdot \frac{8+6 \sqrt{2}}{3}.$$
Therefore, 
$$p^{3/2} \sum_{k>p/2}^{p+1} \left( \frac{k}{p}-1 \right) \mathbf{b}(p,k)  \quad\mathop{\longrightarrow}_{p \rightarrow \infty}  \tone \cdot {4 \sqrt{2}}- 2 \tone \cdot \frac{8+6 \sqrt{2}}{3}= - \frac{16}{3} \tone.$$
This establishes \eqref{eq:cv2} and completes the proof. 
\end {proof}

 We are now ready to prove Proposition~\ref{prop:scalingllcycle}.
 
\begin{proof}[{Proof of Proposition~\ref{prop:scalingllcycle}}] Let ${{X'}}$ be the self-similar Markov process with index $ \alpha=-3/2$ which is defined just like $ {{\widetilde X}} $, except that its driving Lévy process is  $ \xi'(t)=\xi(2 \tone t)$. It is a simple matter to see that ${{X'}}$ has the same distribution as $ ( {{\widetilde X}}( 2 \tone t); t \geq 0)$. It is therefore enough to show that $({{\widetilde L}}^{(p)}({[p^{3/2}t]})/p ; t \geq 0)$ converges in distribution to $ {{X'}}$. Note that the characteristic exponent $\Phi'$ of $\xi'$ is given by $\Phi'= 2 \tone \Phi$, so that
$$ \Phi'(\lambda) =  - i \tone \frac{8 (3 \pi-7)}{9}  \lambda+ \int_{- \ln(2)}^{0} \left( \e^{i\lambda x}-1+i \lambda x  \right)  \ \Pi'(\d x),$$
with $\Pi'(\d x)= 2 \tone \Pi(\d x)$ and  $\lambda \in \R$.
We now check that the assumptions  \textbf{(A1), (A2), (A3), (A4), (A5)} of \cite{BK14} hold (to keep the exposition as short as possible, we do not reproduce their statement here), and the desired result will follow \cite[Theorems 3 \& 4]{BK14}.

For \textbf{(A1)}, we need  the following vague convergence of measures on $ \overline {\R} \backslash \{ 0\}$:
$$p^{3/2} \cdot \Pi^{(p)}(\d x)  \quad\mathop{\longrightarrow}^{(v)}_{p \rightarrow \infty} \quad  \Pi'(\d x).$$
This is an immediate consequence of Lemma \ref {lem:techn}. Similarly, \textbf{(A2)} follows from Lemma \ref{lem:techn} (which shows in particular that there is no Brownian part). For \textbf{(A4)} (which implies \textbf{(A3)}), we need to check the existence of $ \beta_{0}> 3/2$ such that $ \Psi'( \beta_{0})<0$, where $ \Psi'\coloneqq 2\tone \Psi$ is the Laplace exponent of $\xi'$. 
One can for instance take $ \beta_{0}=2$.
Finally, \textbf{\textbf{(A5)}} clearly holds since $ \Pi^{(p)}$ has finite support.
\end {proof}
 
\begin{remark} There is an alternative way of establishing a less explicit version of  Proposition~\ref{prop:scalingllcycle} which circumvents the appeal to \cite{BK14} and rather uses results in \cite{CLGpeeling} for the UIPT and the relation \eqref{eq:radonnikodym} of local absolute continuity between the latter and Boltzmann triangulations. 
Specifically, consider a peeling of the UIPT of the $p$-gon of the type dealt with in \cite{CLGpeeling},
that is the exploration only concerns the unbounded region, as at each step of the peeling, the yet unexplored bounded region that may arise is filled-in. Let $ {{\widetilde L}}^{(p)}_{\infty}(n)$ denote the perimeter of the cycle resulting after $n$ steps of peeling, and observe that for the peeling algorithms considered here for a Boltzmann triangulation, the locally largest cycle is the unique cycle such that the process of its perimeter 
never drops by more than a half of its value. It then follows from \eqref{eq:radonnikodym}
that for every $n\geq 0$ and every sequence $x_0=p, x_1, \ldots, x_n$ in $\{2, 3, \ldots\}$ with $x_{i+1}\geq \frac{1}{2}x_i$ for all $i=0, \ldots, n-1$, there is the identity
$$\P\left( {{\widetilde L}}^{(p)}(0)=x_0, \ldots, {{\widetilde L}}^{(p)}(n)=x_n\right)=\frac{f(p)}{f(x_n)}
\P\left( {{\widetilde L}}_{\infty}^{(p)}(0)=x_0, \ldots, {{\widetilde L}}_{\infty}^{(p)}(n)=x_n\right).$$
Since a version of Proposition 5 of \cite{CLGpeeling} shows that 
the process $\left( {{\widetilde L}}^{(p)}_{\infty}(n): n\geq 0\right)$ has the distribution of a certain random walk 
conditioned to remain larger than $1$, the identity above determines the law of the chain ${{\widetilde L}}^{(p)}$.

On the other hand, recall also that $f(p)\sim c p^3$ and that, according to Proposition 5 of \cite{CLGpeeling}, 
there is the weak convergence
$$\left( \frac{1}{p} {{\widetilde L}}_{\infty}^{(p)}({[p^{3/2}t]}); t \geq 0\right) \quad  \xrightarrow[p\to\infty]{(d)} \quad  (S^+(t) ; t \geq 0)$$
where in the right-hand side, $(S^+(t) ; t \geq 0)$ is a spectrally negative stable L\'evy process with index $3/2$ started from $S^+(0)=1$ and conditioned to stay positive. One can then deduce from above that as $p\to\infty$, the rescaled process 
$\frac{1}{p} {{\widetilde L}}^{(p)}({[p^{3/2}\times \cdot]})$ converges in distribution to a process which
can be described as a Doob transform of a spectrally negative stable L\'evy process with index $3/2$,
killed when it becomes negative and when having a jump smaller than the negative of half of its value.
This description is however much less explicit and useful as the one obtained in Proposition \ref{prop:scalingllcycle}.

\end{remark}

\section{Branching peeling by layers}
\label{bigsec:pbl}

Recall that the {height} of a vertex $x$ in a triangulation with a boundary $\mathbf{t}$ is its distance to the boundary and that for $r \geq 0$, the ball of radius $r$ of $ \mathbf{t}$ is the map $B_{r}( \mathbf{t})$ that  consists of all the faces of $ \mathbf{t}$ which have a vertex  at height less than or equal to  $r-1$ in $ \mathbf{t}$, with the convention that $B_{0}(\mathbf{t})$ is just the boundary of $\mathbf{t}$.
In addition, in $ B_{r}( \mathbf{t})$, the edges between two vertices at distance $r$ which do not belong to a same cycle are split into two edges enclosing a $2$-gon.

Here we describe a (deterministic) branching peeling algorithm, called \emph{peeling by layers}, which will in particular allow us to discover the cycles of   $B_{r}(\mathbf{t})$. Roughly speaking, this exploration procedure ``turns'' in clockwise order around the holes of the being explored triangulation, and discovers $B_{r}( \mathbf{t})$ layer after layer (by layer, we mean all the vertices having the same height), see \eqref{eq:stoppingtimeboule} below for a precise statement.  This algorithm is an easy adaptation of the (non branching) peeling by layers of \cite[Section~4.1]{CLGpeeling}, which itself builds upon \cite{Ang03}. 

We will then extend Proposition \ref{prop:scalingllcycle} to establish the existence of the scaling limit of the locally largest cycle at given heights.

\subsection{Definition of the branching peeling by layers algorithm}
\label{sec:branchedpeelinglayers}

The branching peeling by layers  $ \mathcal{H}_{0}(\mathbf{t}) \subset \cdots \subset \mathcal{H}_{n}(\mathbf{t}) \subset \cdots \subset \mathbf{t}$ of a triangulation with a boundary $\mathbf{t}$ will be designed in such a way to satisfy the following property for every $i \geq 0$: 

\begin{quote} $(P)$: If $ \mathcal{H}_{i}(\mathbf{t}) \neq \mathbf{t}$, there exists an integer $r \geq 0$ such that all the vertices of the cycles of $ \mathcal{H}_{i}(\mathbf{t})$ are at distance either $r$ or $r+1$ from the external boundary of $ \mathcal{H}_{i}(\mathbf{t})$. In addition, the set of all the vertices with height $r$ of every cycle forms a connected interval inside this cycle.
\end{quote}

We now describe the corresponding algorithm $ \mathcal{A}$. First, if $ \mathcal{H}_{i}(\mathbf{t}) = \mathbf{t}$, set $ \mathcal{A}(\mathcal{H}_{i})=\dagger$. Otherwise, $\mathcal{H}_{i}(\mathbf{t}) \neq \mathbf{t}$ and if $ \mathcal{H}_{i}(\mathbf{t})$ satisfies $(P)$, then, for a certain $r \geq 0$, all the edges on the cycles of $ \mathcal{H}_{i}( \mathbf{t})$ are of the form $(r,r)$, $(r,r+1)$, $(r+1,r)$ or $(r+1,r+1)$ depending on the heights of their vertices read in clockwise order. The algorithm $ \mathcal{A}$ then peels a (deterministic) edge of the form $(r+1,r)$. If there is no such edge on the boundary of the cycles that means that they are all of form $(r+1,r+1)$ or $(r,r)$; in such case we peel any (deterministic) edge on the boundary of $ \mathcal{H}_{i}( \mathbf{t})$ of the form $(r,r)$.  
By induction, it is easy to check that $(P)$ holds for $ \mathcal{H}_{i}(\mathbf{t})$ for every $i \geq 0$.

Finally, for every $r \geq 0$, we introduce the stopping time $\theta_{r}$ as the first time $i \geq 0$ when all the vertices of all the cycles of $ \mathcal{H}_{i}$ have height at least $r$. We claim that 
 \begin{eqnarray} \label{eq:stoppingtimeboule} B_{r}( \mathbf{t}) = \mathcal{H}_{\theta_{r}}.  \end{eqnarray}
Indeed, for $ r \geq 0$, it is plain that all the faces in $ \mathcal{H}_{\theta_{r}}$ have a vertex a height at most $r-1$. Conversely, all the vertices in $ \mathbf{t} \backslash \mathcal{H}_{\theta_{r}}$ are at distance at least $r$ from the original boundary. Note that here it is important that in $B_{r}( \mathbf{t})$, by definition, the edges between two vertices at distance $r$ which do not belong to a same cycle are split into two edges enclosing a $2$-gon. Indeed, it may happen that a hole of perimeter $2$ in $ \mathcal{H}_{\theta_{r}}$ is later filled-in with the trivial triangulation and thus giving rise to a single edge in $ \mathbf{t}$.
\begin{proposition}\label{prop:marth} Under $ \mathbb{P}^{(p)}$, the process $(M_{\theta_{r}})_{r \geq 0} $ is  $(\mathcal{F}_{n})_{n \geq 0}$ martingale.
\end{proposition}

This is a simple consequence of Corollary \ref{cor:stopping}, since, under $ \mathbb{P}^{(p)}_{\infty}$,  $\theta_{r}<\infty$ almost surely because  $B_{r}(\mathbf{t})$ is almost surely finite.

We mention that this martingale appears in  \cite[Theorem 4]{CLGpeeling} in the case of type II triangulations (no loops) of the sphere. Note also that the scaling factor is different, as the martingale is normalized to start from $1$ in \cite{CLGpeeling}, and that in \cite{CLGpeeling} the definition of the ball of radius $r$ is slightly different (edges joining vertices of the same height  belonging to a cycle are not split into a $2$-gon) introducing a somehow different factor of the holes of perimeter $2$.

\subsection{Scaling limits for the locally largest cycle at given heights}
\label{sec:haut}

Our goal is to understand the genealogical tree structure of cycles explored during the branching peeling by layers of a large Boltzmann triangulation. To this end, we start by focusing on the evolution of a distinguished cycle, namely the (locally) largest cycle at each step. As in Section~\ref{sec:llc}, we denote by $(\mathscr{C}^{\ast}(i))_{i \geq 0}$ the sequence of locally largest cycles obtained when using the peeling by layers algorithm (and started with the initial boundary of $\mathbf{t}$),
and let $H_{n}$ the minimal height of a vertex of  $ \mathscr{C}^{\ast}(n)$.  Then, for $r \geq 0$ we denote by $\theta^\ast(r)$  the first time $k \geq 0$ when $H_{k} \geq r$, and we finally let $L(r)=|\mathscr{C}^{\ast}(\theta^\ast(r))|$ be the perimeter at height $r$  of the locally largest cycle.

Recall from the Introduction the definition of the self-similar process $ {{X}}$ and from \eqref{eq:ta} the definition of $\tone$ and $\aone$. Finally, let $(L^{(p)}(r); r \geq 0)$ be a random variable distributed as $(L(r); r \geq 0)$ under $ \mathbb{P}^{(p)}$.
  
\begin{proposition}[Scaling limit for the locally largest cycle at heights] \label{prop:hauteur} 
The convergence
$$ \left( \frac{1}{p} L^{(p)}\left([ \sqrt{p} t] \right) ; t \geq 0 \right)  \quad \mathop{\longrightarrow}^{(d)}_{p \rightarrow \infty} \quad   \left(  {{X}} \left(  \frac{2 \tone}{\aone} \cdot t  \right) ; t \geq 0 \right) $$
holds in distribution in $\mathbb{D}(\mathbb{R}_{+},\mathbb{R})$.
\end{proposition}

The proof goes along the same lines as \cite[Section~4]{CLGpeeling}: one first proves an invariance principle  for the sequence $(|\mathscr{C}^{\ast}(n)| ; n \geq 0)$ in Proposition \ref{prop:scalingllcycle}, one then establishes a scaling limit for the time-changes $( \theta^{\ast}(r); r \geq 0)$ and the conclusion follows by combining these two limit theorems. More precisely for $p\geq 1$, we consider the evolution of a locally largest cycle $ \mathcal{C}^*(n)$ and its height process $ \theta^*(r)$  under the Boltzmann measure and abusing a little notation we put under $ \mathbb{P}^{(p)}$ $$H^{(p)}(n)= H(n), \qquad \widetilde{L}^{(p)}(n) = | \mathcal{C}^*(n)|, \qquad L^{(p)}(r) = \widetilde{L}^{(p)}( \theta^*_{r})$$ so that $H^{(p)},\widetilde{L}^{(p)}$ and $ L^{(p)}$ are now living on the same probability space. We already have a scaling limit for $\widetilde{L}$ given in Proposition \ref{prop:scalingllcycle} where we recall that $ {{\widetilde X}}$ is the process defined by \eqref{eq:defX}.  In particular, setting $\widetilde{\sigma}^{(p)}_{ \varepsilon} = \inf \{ i \geq 0 : {{\widetilde L}}^{(p)}_{i} \leq \varepsilon p\}$ and  $ \widetilde{\sigma}_{\varepsilon}=\inf \{ t \geq 0 : {{\widetilde X}}(2 \tone t) \leq \varepsilon\}$, the convergence
\begin{equation}
\label{eq:limtau} \frac{1}{p^{3/2}} \cdot \widetilde{\sigma}^{(p)}_{ \varepsilon}  \quad \mathop{\longrightarrow}^{(d)}_{p \rightarrow \infty} \quad  \widetilde{\sigma}_{ \varepsilon}
\end{equation} holds in distribution, jointly with \eqref{eq:scaling}. We put similarly  $\sigma^{(p)}_{ \varepsilon} = \inf\{ r \geq 0: L^{(p)}(r) \leq \varepsilon p\}$ and $ \sigma_{\varepsilon}= \inf \{t \geq 0 : {{X}}(t) \leq \varepsilon \}$.

   The main ingredient to establish Proposition \ref{prop:hauteur} is the following result:
     \begin{lemma}  \label{LEM:HAUTEPS} For every $\varepsilon>0$, jointly with \eqref{eq:scaling} and \eqref{eq:limtau}, the convergence  
$$ \left( \frac{1}{ \sqrt{p}}H_{[p^{3/2}t] \wedge \widetilde{\sigma}^{(p)}_{ \varepsilon}}^{(p)} ; t \geq 0\right)    \quad \mathop{\longrightarrow}^{(d)}_{p \rightarrow \infty} \quad  \aone \cdot \left(\int_{0}^{t \wedge \widetilde{\sigma}_{ \varepsilon}} \frac{ \mathrm{d}s}{ {{\widetilde X}}(2 \tone s)} ; t \geq 0\right)$$
holds in distribution in $\mathbb{D}(\mathbb{R}_{+},\mathbb{R})$.
  \end{lemma}
  
\begin{remark} Contrary to Proposition \ref{prop:scalingllcycle}, our proof of Lemma \ref{LEM:HAUTEPS} does not imply convergence of the rescaled absorption time of $L^{(p)}(\cdot)$ at $0$ towards that of $X(2 \tone /\aone\cdot)$ because of the cutoff. The statement of the last lemma is however still true for $\varepsilon=0$ and this can be proved  using Theorem \ref{thm:no tentacles}. Since we do not need this fact, we do not enter details.
\end{remark}
  
We start by explaining how Proposition \ref{prop:hauteur} simply follows from Lemma \ref{LEM:HAUTEPS} by the Lamperti transformation.
  
  \begin{proof}[Proof of Proposition \ref{prop:hauteur}]For every $t \geq 0$ we set
$$\rho(t)= \inf  \left\{ u \geq 0 ; \int_{0}^{u} \frac{1}{ {{\widetilde X}}(2 \tone s)}  \d s =t \right\}. $$
By Lemma \ref{LEM:HAUTEPS} and the a.s.~strict monotonicity of $ t \in [0, \widetilde{\sigma}_{ \varepsilon}]\mapsto  \int_{0}^t  \mathrm{d}s/\widetilde X(2 \tone s)$ we have
\begin{equation}
\label{eq:theta} \left(  \frac{1}{p ^{3/2}} \theta^{\ast}( \fl{\sqrt{p} t} \wedge \sigma_{\varepsilon}^{(p)}) ; \ t \geq 0 \right)  \quad \mathop{\longrightarrow}^{(d)}_{p \rightarrow \infty} \quad \left( \rho(t/\aone) \wedge \widetilde{\sigma}_{\varepsilon} ; \ t \geq 0 \right).
\end{equation}
Recalling that $L^{(p)}(r)={{\widetilde L}}^{(p)}_{\theta^{*}(r)}$, by combining the last convergence with \eqref{eq:scaling}, we get that 
\begin{equation}
\label{eq:cvi} \left( \frac{1}{p} L^{(p)}\left( [ \sqrt{p} t] \wedge \sigma^{(p)}_{ \varepsilon}\right) ; t \geq 0 \right)  \quad \mathop{\longrightarrow}^{(d)}_{p \rightarrow \infty} \quad  \left( {{\widetilde X}} \left(  2 \tone( {\rho(t/\aone)} \wedge \widetilde{\sigma}_{\varepsilon} \right)  ; t \geq 0 \right);
\end{equation}
see Section~6.1 of Ethier and Kurtz \cite{EK86}. 
It is a simple matter to deduce from \eqref{eq:tildextox} that
$$ \left( {{\widetilde X}} \left(  2\tone (\rho(t)\wedge \widetilde{\sigma}_{\varepsilon} )\right)   ; t \geq 0 \right)   \quad \mathop{=}^{(d)} \quad  \left( {{X}} \left(  (2 \tone t) \wedge \sigma_{\varepsilon} \right)   ; t \geq 0 \right),$$
so we get from \eqref{eq:cvi} that the weak convergence stated in Proposition \ref{prop:hauteur} holds provided that on both sides, we stop the processes at the first instant when they become smaller than $\varepsilon$. 

To complete the proof, it suffices to observe that the probability that these processes exceed $c\varepsilon$ after that time can be made as small as we wish (uniformly in $p$), 
by choosing $c$ sufficiently large. Indeed, for the process $ \frac{1}{p} L^{(p)}\left([ \sqrt{p} t]\right) $, this follows from the fact that $f(L^{(p)}(n))$ is a super-martingale (thanks to Corollary \ref{cor:stopping}) and the optional sampling theorem. The argument for the self-similar Markov process $X$ is similar. Specifically, note first that $\Psi(3)\leq 0$ (indeed, in the notation of the forthcoming Section \ref{sec:cell-systems}, we have $\Psi\leq \kappa$ and $\kappa(3)=0$)
and hence the process $\exp(3\xi(t))$ is a super-martingale. We deduce from Lamperti's transformation  that $X^3(t)$ is also a super-martingale, and same conclusion follows. This completes the proof of our statement. 
\end{proof}

\begin{proof}[Proof of Lemma \ref{LEM:HAUTEPS}] The result will follow from the work \cite{CLGpeeling} and absolute continuity relations between peeling explorations in the UIPT and in Boltzmann triangulations. More precisely, we introduce a modified peeling process that only peels along the locally largest cycle: Denote by $ \overline{\mathcal{H}}_{0} \subset \cdots \overline{\mathcal{H}}_{n} \subset \cdots \subset \mathbf{t}$  a branching peeling exploration of $\mathbf{t}$ obtained by using the peeling by layers algorithm $ \mathcal{A}$, but with the following modification: $ \overline{\mathcal{H}}_{0}$ is still the boundary of the external face of $\mathbf{t}$, but, for every $i \geq 0$,  if  $\mathcal{A}(\overline{\mathcal{H}}_{i}) \ne \dagger$,  $  \overline{\mathcal{H}}_{i+1}$ is defined to be the triangulation with holes obtained from $ \overline{\mathcal{H}}_{i}$ by peeling the edge $ \mathcal{A}(  \overline{\mathcal{H}}_{i})$, and when a peeling event of type $\mathsf{G}_{k}$ occurs, by also \emph{by filling-in the hole adjacent to the cycle of smallest perimeter among the two newly created cycles}. Note that for every $i \geq 0$, $ \overline{\mathcal{H}}_{i}$ has at most one cycle, which is precisely $ \mathscr{C}^{\ast}(i)$. 

When we apply this algorithm to the triangulation $T^{(p)}$ we recover the above processes $ \widetilde{L}^{(p)}$ and $ H^{(p)}$ respectively as the perimeter of the single hole of  $ \overline{\mathcal{H}}_{n}$ and the minimal height of a vertex on it. We will show that for every $ \varepsilon>0, t_{0}>0$ and $\delta>0$ we have 
  \begin{eqnarray} \label{eq:Hintegral}   \mathbb{P}\left( \sup_{0 \leq t \leq t_{0}}\frac{1}{ \sqrt{p}} \left|  H^{(p)}_{[p^{3/2}t] \wedge \widetilde{\sigma}^{(p)}_{ \varepsilon}} - \aone \int_{0}^{t \wedge ( p^{-3/2} \widetilde{\sigma}_{ \varepsilon}^{(p)})} \frac{ \mathrm{d}s}{ p^{-1} \cdot \widetilde{L}^{(p)}_{[p^{3/2} s]}}\right|  \geq \delta\right) \xrightarrow[p\to\infty]{} 0.  \end{eqnarray}
Indeed, the statement of the lemma then easily follows by combining \eqref{eq:Hintegral} with  \eqref{eq:scaling}. 
To prove \eqref{eq:Hintegral}, note that by Corollary \ref{cor:peelinggenerallaw}, for every positive measurable function $F$ supported by sequences of finite triangulations having only one hole, we have
  \begin{eqnarray} \label{eq:backtouipt} \mathbb{E}^{(p)} \left[ F\big( (\overline{\mathcal{H}}_{k})_{ 0 \leq k \leq \widetilde{\sigma}_{  \varepsilon}^{(p)}}\big) \frac{ f( \widetilde{L}^{(p)}_{\widetilde{\sigma}_{  \varepsilon}^{(p)}})}{ f(p)} \right] = \mathbb{E}^{(p)}_{\infty} \left[ F\big( (\overline{\mathcal{H}}_{k})_{ 0 \leq k \leq \widetilde{\sigma}_{  \varepsilon}^{(p)}}\big) \mathbbm{1}_{ \widetilde{\sigma}_{  \varepsilon}^{(p)} <\infty} \right],  \end{eqnarray}  
  where we recall that $f$ was introduced in \eqref{def:f}. Notice that for the random variable appearing under the expectation in the right-hand side to be non-equal to $0$,  at each time $k \leq  \widetilde{\sigma}_{  \varepsilon}^{(p)}$ the boundary of $ \overline{\mathcal{H}}_{k}$ must be both the locally largest cycle and the cycle separating from infinity in the UIPT.
  In particular we have 
  $$ \mathbb{E}^{(p)}_{{\infty}} \left[ F\big( (\overline{\mathcal{H}}_{k})_{ 0 \leq k \leq \widetilde{\sigma}_{  \varepsilon}^{(p)}}\big) \mathbf{1}_{\widetilde{\sigma}_{  \varepsilon}^{(p)} <\infty} \right] \leq \mathbb{E}^{(p)}_{{\infty}} \left[ F\big( (\overline{\mathcal{H}}'_{k})_{ 0 \leq k \leq \widetilde{\varsigma}_{  \varepsilon}^{(p)}}\big)\mathbbm{1}_{\widetilde{\varsigma}_{  \varepsilon}^{(p)} <\infty}\right],$$
  where $ (\overline{\mathcal{H}}'_{n})_{n \geq 0}$ is the sequence of triangulations with a single hole obtained by peeling the UIPT of the $p$-gon with the peeling by layers algorithm along the cycle separating from infinity and filling-in the finite holes created during the process, and where  $\widetilde{\varsigma}_{ \varepsilon}^{(p)}$ is the first time when the perimeter drops below $ \varepsilon p$ during such an exploration.
  The process $\overline{\mathcal{H}}'$ is precisely the one studied in details in \cite[Section 4]{CLGpeeling} and we deduce from it that the analog of \eqref{eq:Hintegral} for the height process and perimeter process of $\overline{ \mathcal{H}}'$ holds. To finish the proof we use this fact together with \eqref{eq:backtouipt} and the fact that $f( \widetilde{L}^{(p)}_{\widetilde{\sigma}^{(p)}_{ \varepsilon}})/f(p)$ is bounded from below by a positive constant  depending only on $ \varepsilon$ since $f(p) \sim c p^3$ and  $\widetilde{L}^{(p)}_{\widetilde{\sigma}^{(p)}_{ \varepsilon}}> \varepsilon p /2$. \end{proof}

 \bigskip

In the rest of this section, unless explicitly mentioned, we work with the peeling by layers algorithm, which we denote by $ \mathcal{A}$. We will now show, roughly speaking, that for every $ \varepsilon>0$, with high probability as $p \rightarrow  \infty$, the structure of the cycles of a Boltzmann triangulation of the $p$-gon is well approximated (in various senses which will be made precise below) by the genealogical tree structure of cycles cut above all cycles that have perimeter less than $\varepsilon p$.

\subsection{Definition of the exploration with cutoff}

\label{sec:cutoff}

For every $c>0$, we consider the peeling by layers algorithm $ \mathcal{A}^{<c}$, defined exactly as $ \mathcal{A}$, but with the additional constraint that $ \mathcal{A}^{<c}$  may never select an edge that belongs to a cycle of length less than $c$. We denote by
$$ \mathcal{H}^{<c}_{0}(\mathbf{t}) \subset \cdots \subset \mathcal{H}^{<c}_{n}(\mathbf{t}) \subset \cdots \subset \mathbf{t}$$
the corresponding peeling process, which we call the \emph{branching peeling by layers frozen below level $c$}.  Intuitively speaking,  $ \mathcal{A}^{<c}$ yields the branching peeling by layers exploration, except that each time a new cycle of perimeter strictly less than $c$ is created, it is instantly frozen and is not explored in the sequel.
If $\mathbf{t}$ is a finite triangulation with a boundary, the branching peeling process associated with  $ \mathcal{A}^{<c}$ does not necessarily entirely explore  $\mathbf{t}$: we let $\tau_{c}=\min \{k \geq 0; \mathcal{A}^{<c}(\mathcal{H}^{<c}_k)=\dagger \}$ be the first time when branching peeling frozen below level $c$ stops, and let $\mathsf{Cut}( \mathbf{t}, c)= \mathcal{H}^{<c}_{\tau_{c}}( \mathbf{t})$ be the largest triangulation with holes obtained in this branching peeling process. To simplify notation, we denote by $\mathsf{C}_{1}^{ <c},\mathsf{C}_{2}^{ <c},\ldots$ the (possibly empty) collection of  cycles of $\mathsf{Cut}( \mathbf{t}, c)$. We emphasize that a given  cycle $\mathsf{C}_{i}^{ <c}$ is not necessarily a cycle of $ B_{r} ( \mathbf{t})$ for a certain $r \geq 0$, since it can be is ``in-between'' two successive layers.

\begin{figure}[!h]
 \begin{center}
  \includegraphics[height=6cm]{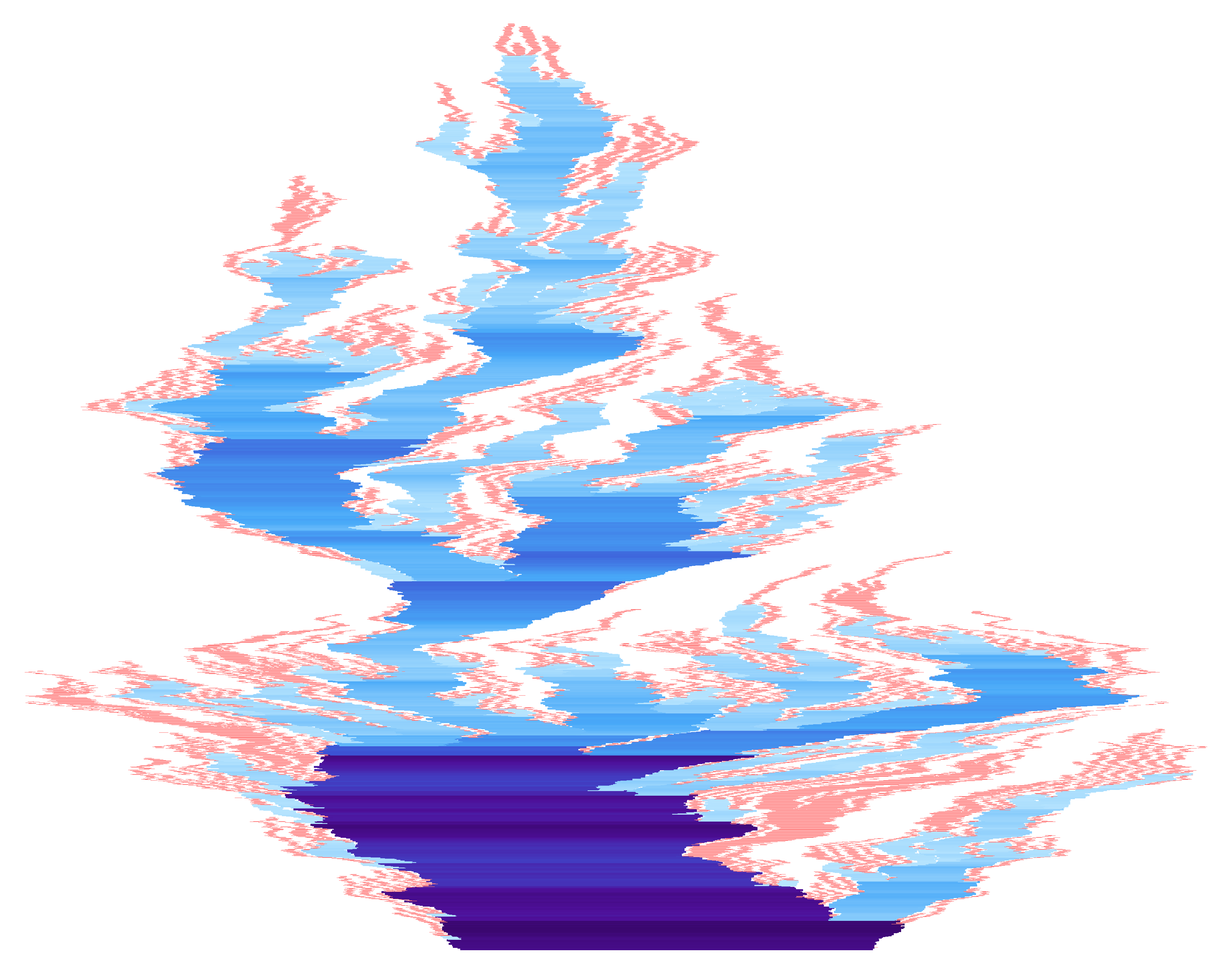}
 \caption{Illustration of the branching peeling process frozen below level $ \varepsilon p$ on an actual simulation (as in  Figure~\ref{fig:sim1}), the pink cycles are those that are not explored.} 
 \end{center}
 \end{figure}
 
 When $ \mathbf{t} = T^{(p)}$ is a Boltzmann triangulation of the $p$-gon we will take $ c = \varepsilon p$ with $ \varepsilon>0$ fixed but small. Recall the definition of the  function $f$ from \eqref{def:f}. The following lemma will play a crucial role in the estimation of various errors made by this cutoff:

\begin{lemma} \label{lem:epsexplo} We have
$$ \sup_{p \geq 1} p^{-3} \Esp{\sum_{i \geq 1} f( | \mathsf{C}_{i}^{ < \varepsilon p}|)}   \quad \mathop{\longrightarrow}_{\varepsilon \rightarrow 0} \quad  0.$$
\end{lemma}

\begin{proof}  It $\mathbf{t}$ is a finite triangulation with a boundary and $0 \leq \varepsilon p<1$, first note that  $\mathsf{Cut}( \mathbf{t}, \varepsilon p)=\mathbf{t}$ has no cycles, so that $\sum_{i \geq 1} f( | \mathsf{C}_{i}^{ < \varepsilon p}|)=0$. Without loss of generality, we may therefore assume that $\varepsilon>0$, $p \geq 1/\varepsilon$. Let  $(M_{n})_{n \geq 1}$ be the cycle martingale associated with the peeling algorithm $ \mathcal{A}^{< \varepsilon p}$ defined by \eqref{eq:M}. Recall that $\tau_{\varepsilon p}$ is the time when the branching peeling frozen below level $\varepsilon p$ stops. Then, by definition of the cycle martingale, $ \Esp{\sum_{i \geq 1} f( | \mathsf{C}_{i}^{ < \varepsilon p}|)}= \Esp{M_{\tau_{\varepsilon p}}}$.
Hence, by \eqref{eq:stopping},
$$ p^{-3} \Esp{\sum_{i \geq 1} f( | \mathsf{C}_{i}^{ < \varepsilon p}|)} = \frac{f(p)}{p^{3}} \cdot  \mathbb{P}^{(p)}_{\infty} (\tau_{\varepsilon p} < \infty).$$
Since $f(p) \sim c p^{3}$ as $p \rightarrow  \infty$ for a certain constant $c>0$, it is enough to show that $\mathbb{P}^{(p)}_{\infty} (\tau_{\varepsilon p} < \infty)$ goes to $0$ as $ \varepsilon \to 0$, uniformly in $p \geq 1$.  To see this we consider the branching peeling by layers frozen below level $ \varepsilon p$ on the UIPT of type I, and denote by $(Z_{k}^{(p)})_{k \geq 0}$ the Markov chain that evolves like the perimeter of the cycle disconnecting the boundary of the external face of the UIPT of type I from infinity in this peeling (see \cite[Section~3.1]{CLGpeeling} for the associated one-step peeling transitions). Then clearly $\mathbb{P}^{(p)}_{\infty} (\tau_{\varepsilon p} < \infty) \leq \mathbb{P}( \exists k \geq 1 ; Z_{k}^{(p)} \leq \varepsilon p)$.  On the other hand, by \cite[Section~3.2 and 3.3]{CLGpeeling}, the chain $ Z^{(p)}$ evolves as a certain random walk started from $p$ and conditioned to stay positive forever and an adaptation of \cite{CLGpeeling} to the type I setting (see in particular the second display in Section~3.3 in \cite{CLGpeeling} and \cite[Section 6]{CLGpeeling} for this adaptation) yields the existence of a constant $c>0$ such that $$ \forall \varepsilon>0, \quad  \forall p \geq 1/ \varepsilon, \qquad  \mathbb{P}^{(p)}_{\infty}( \exists k \geq 1 : Z_{k}^{(p)} \leq \varepsilon p )  \leq \frac{h( \varepsilon p )}{h(p)} \leq c \sqrt{\varepsilon}$$
with $h(p)=12^{-p}C(p)$, where we recall that $C(p)$ is defined in \eqref{equivalentcp}.  {Alternatively, as suggested by a referee, this can also be deduced from directly from the lack of cut-points in the Brownian plane and convergence of type I triangulations to the Brownian plane \cite{Bud16}.} The conclusion readily follows.
\end{proof}

\subsection{Mass of the lost cycles}
\label{sec:L3}

We introduce a natural genealogical order on cycles of a triangulation: If $ \mathbf{t}$ is a triangulation of the $p$-gon and if $ \mathscr{C}$ and $ \mathscr{C}'$ are two simple loops of $ \mathbf{t}$, we say that $ \mathscr{C}$ is an ancestor of $ \mathscr{C}'$, and write $ \mathscr{C} \preceq \mathscr{C}'$, if { $ \mathscr{C}' \subset   \mathscr{C}  \cup \mathbf{t}^{\mathscr{C}}$, where $\mathbf{t}^\mathscr{C}$ is  the component of $ \t \backslash \mathscr{C}$ which does not contain the external face of $ \mathbf{t}$}. Clearly, this partial order restricted to the cycles at heights of $ \mathbf{t}$ gives rise to a tree structure (see Section \ref{bigsec:proofthm1} for details concerning this genealogical structure). 
 For $r \geq 0$, we say that a simple path of $\mathbf{t}$ is a \emph{cycle at height $r$} if it is one of the cycles of $B_{r}(\mathbf{t})$. Fix $c>0$ and imagine a branching peeling exploration of $\mathbf{t}$ frozen below level $c$. Recall that $\mathsf{C}_{1}^{ <c},\mathsf{C}_{2}^{ <c},\ldots$ are the  cycles of $\mathsf{Cut}( \mathbf{t}, c)$. We denote by $(\ell^{ <c}_{i}(r))_{i \geq 1}$ the (possibly empty) sequence of perimeters of the cycles at height $r$ of $\mathbf{t}$ that are a descendant of (or that are equal to) one of the cycles  $\mathsf{C}_{1}^{ <c},\mathsf{C}_{2}^{ <c},\ldots$. Such cycles are called the \emph{lost cycles} at cutoff level $c$ and height $r$. 

In the case of Boltzmann triangulations of the $p$-gon, we  show that the mass (in the $\ell^{3}$ sense) of the lost cycles at cutoff level $\varepsilon p$ is negligible as $ \varepsilon \to 0$, uniformly in $p$:

\begin{proposition}  \label{prop:l3lostcycles}For every $ \delta>0$, we have
$$ \sup_{p \geq 1}\Prp{\sup_{r \geq 0} p^{-3} \sum_{i \geq 1}\left( \ell^{< \varepsilon p}_{i}(r)\right)^3>\delta}  \quad \mathop{\longrightarrow}_{\varepsilon \rightarrow 0} \quad 0.$$
\end{proposition}

\begin{proof} 
 For $i \geq 1$, let $$ \mathcal{D}_{i}^{<\varepsilon p}(r)= \{ \mathcal{C} :   \mathcal{C}  \textrm{ is a cycle at height $r$ of } \mathbf{t} \textrm{ and is a descendant of  (or is equal to) } \mathsf{C}_{i}^{ < \varepsilon p} \}$$  
be the (possibly empty) collection of descendants of $\mathsf{C}_{i}^{ < \varepsilon p}$ which are cycles at height $r$. In order to explore the lost cycles, we introduce the following process $( \mathscr{M}^{< \varepsilon p}(r) )_{r \geq 0}$:
$$ \mathscr{M}^{< \varepsilon p}(r) = \sum_{i \geq 1} \left\{ \begin{array}{cl} f(| \mathsf{C}_{i}^{ < \varepsilon p}|) & \mbox{ if }  \mathcal{D}_{i}^{< \varepsilon p}(r) = \varnothing,\\
 \displaystyle \sum_{ \mathcal{C} \in  \mathcal{D}_{i}^{<\varepsilon p}(r) }  f(\left|  \mathcal{C} \right|)& \mbox{ otherwise}. \end{array}\right.$$

Under $\mathbb{P}^{(p)}$ and conditionally given $  \mathsf{Cut}(\mathbf{t}, \varepsilon p)$, by Proposition \ref{prop:martingale} and Corollary \ref{cor:stopping}, the process $ \mathscr{M}^{< \varepsilon p}$ is a nonnegative martingale starting from 
 $$ \mathscr{M}^{ < \varepsilon p} (0) = \sum_{i \geq 1} f ( | \mathsf{C}_{i}^{ < \varepsilon p}|).$$
Since  $\sum_{i \geq 1}f\left( \ell^{< \varepsilon p}_{i}(r)\right) \leq \mathscr{M}^{<  \varepsilon p}(r) $ for every $r \geq 0$ and since $f(p) \sim c p^{3}$ as $p \rightarrow  \infty$ for a certain constant $c>0$, it is enough to show that 
$$ \sup_{p \geq 1}\Prp{\sup_{r \geq 0} p^{-3} \mathscr{M}^{<  \varepsilon p}(r)>\delta}  \quad \mathop{\longrightarrow}_{\varepsilon \rightarrow 0} \quad 0.$$
But by Doob's maximal inequality,  we have 
$$ \mathbb{P}^{(p)}\big( \sup_{r \geq 0} \mathscr{M}^{<  \varepsilon p}(r) \geq \delta p^{3} \mid  \mathsf{Cut}(\mathbf{t}, \varepsilon p) \big)  \leq \frac{1}{ \delta p^3} \cdot \sum_{i \geq 1} f( |\mathsf{C}_{i}^{ < \varepsilon p}|).$$
Hence, by taking the expectation under $\mathbb{E}^{(p)}$, we get that
$$\Prp{\sup_{r \geq 0} p^{-3}\mathscr{M}^{<  \varepsilon p}(r)>\delta} \leq  \frac{1}{ \delta p^3} \cdot \Esp{\sum_{i \geq 1} f( |\mathsf{C}_{i}^{ < \varepsilon p}|)}.$$ 
The desired result then follows from Lemma \ref{lem:epsexplo}.
\end{proof}

\subsection{Volumes estimates}
We now provide an estimate under $\mathbb{P}^{(p)}$ on the volumes of the triangulations that fill-in the holes of  $\mathsf{Cut}(  \mathbf{t}, \varepsilon p)$.  We mention that this estimate is not required for the proof of Theorem \ref{thm:main} but will be used in the proof of Theorem \ref{thm:no tentacles}. Its proof is similar to that of Proposition \ref{prop:l3lostcycles} and also relies on Lemma \ref{LEM:HAUTEPS}. 

We denote by $  \mathrm{T}_{1}^{< \varepsilon p}, \mathrm{T}_{2}^{ < \varepsilon p}, \ldots$ the triangulations with boundaries that fill-in the holes of $ \mathsf{Cut}( \mathbf{t}, \varepsilon p)$, and, as before, we let $ \mathsf{C}_{1}^{< \varepsilon p}, \mathsf{C}_{2}^{< \varepsilon p}, \ldots$ be their boundaries.  Recall from Section~\ref{sec:martingales} that the volume of a Boltzmann triangulation of the $p$-gon is of order $p^2$ (see \cite[Proposition 8 and Section~6]{CLGpeeling} for a more precise result and convergence in distribution of $|\mathbf{t}|/p^2$ under $\mathbb{P}^{(p)}$ as $p \rightarrow \infty$). We  show that the maximal volume of a triangulation with a boundary that fills-in a hole of  of  $\mathsf{Cut}(  \mathbf{t}, \varepsilon p)$ under $\mathbb{P}^{(p)}$ is small compared to $p^2$:

\begin{proposition} \label{prop:volumesmall}For every $\delta>0$, $$  \sup_{p \geq 1} \Prp{p^{-2}  \sup_{i \geq 1} | \mathrm{T}_{i}^{< \varepsilon p}| > \delta}  \quad \mathop{\longrightarrow}_{\varepsilon \rightarrow 0} \quad 0.$$
\end{proposition}

\begin{proof} By Proposition \ref{prop:peelinggenerallaw}, under $\mathbb{P}^{(p)}$ and conditionally given  $ \mathsf{Cut}(\mathbf{t}, \varepsilon p)$, the triangulations with boundaries $ \mathrm{T}^{< \varepsilon p}_{1},  \mathrm{T}^{< \varepsilon p}_{2}, \ldots$ are independent, and the law of $\mathrm{T}^{< \varepsilon p}_{i}$ is $\mathbb{P}^{(|\mathsf{C}_{i}^{< \varepsilon p}|)}$ for every $i$. A union bound therefore yields
$$ \mathbb{P}^{(p)}( \exists i \geq 1;  | \mathrm{T}_{i}^{< \varepsilon p}| \geq \delta p^2) \leq \mathbb{E}^{(p)}\left[  \sum_{i \geq 0} \mathbb{P}^{(| \mathsf{C}^{ < \varepsilon p}_{i}|)}( | \mathbf{t}| \geq \delta p^2 )\right]. $$
Using the explicit formulas for $ \# \mathcal{T}_{n,p}$ it is a simple matter to see that $\# \mathcal{T}_{n,p} \leq c \cdot C(p) n^{-5/2} (27/2)^n$ for some constant $c >0$ independent of $p$ and $n$ (see e.g.~\cite{CLGmodif} for similar estimates) and where we recall that $C(p)$ is given by \eqref{equivalentcp}. By definition of the Boltzmann distribution, by using \eqref{equivalentcp} and \eqref{eq:asympZp}, there exists a constant $c >0$, independent of $p$ and $n$, such that $ \mathbb{P}^{(p)} (  | \mathbf{t}| \geq x) \leq c f(p) x^{-3/2}$
for every $p \geq 1$ and $x \geq 1$. Therefore
$$ \mathbb{P}^{(p)}( \exists i \geq 1 ; | \mathrm{T}_{i}^{< \varepsilon p}| \geq \delta p^2) \leq   c\delta^{-3/2} \cdot \frac{1}{ p^3}\mathbb{E}^{(p)}\left[  \sum_{i \geq 0} f(| \mathsf{C}_{i}^{< \varepsilon p}|)\right]. $$
An appeal to Lemma \ref{lem:epsexplo} then completes the proof.
\end{proof}

\section{Proof of Theorem \ref{thm:main}}

\label{bigsec:proofthm1}

Throughout this section, we implicitly work under $\P^{(p)}$ for some fixed $p\geq 1$ and we explore the triangulation of the $p$-gon using the branching peeling by layers algorithm which has been described in Section~\ref{sec:branchedpeelinglayers}. We consider the family of cycles which appear in this peeling exploration, together with the boundary of the $p$-gon, and we recall
that this family is endowed with the natural (partial) order $ \preceq$ induced by their genealogy. More precisely, this yields 
a rooted  tree structure ${\mathbb C}$ which is  binary incomplete, in the sense that each vertex of ${\mathbb C}$ (i.e. each cycle) has out-degree $0, 1$ or $2$, and is planar, so that when a vertex has two children, the largest child is placed upper-left and the smaller upper-right. It will be convenient to agree that when a vertex has a single child, this child is also placed upper-left, so that all edges of ${\mathbb C}$ are either (upper-) right or left edges. The boundary of the $p$-gon is viewed as the root of ${\mathbb C}$, and when the peeling algorithm explores a new triangle with base lying in  some cycle ${\mathscr C}$, the outcome is either ($ \mathsf{V}$) or ($ \mathsf{G}$) or ($ \mathsf{C}$), where
\begin{enumerate}
\item[($ \mathsf{V}$):]
this triangle is degenerate with two vertices and a single (oriented) edge and then ${\mathscr C}$ is  a leaf of ${\mathbb C}$, 
 \item[($ \mathsf{G}$):]
this exploration splits ${\mathscr C}$ into two new cycles ${\mathscr C}_1$ and ${\mathscr C}_2$ with $|{\mathscr C}_1|\geq |{\mathscr C}_2|$ and $|{\mathscr C}_1|+|{\mathscr C}_2|=|{\mathscr C}|+1$, and then we view 
${\mathscr C}_1$ as the left child of ${\mathscr C}$ and ${\mathscr C}_2$ as the right child, 
\item[($ \mathsf{C}$):]
the third extremity of this triangle does not lie on ${\mathscr C}$ and the exploration thus produces a larger cycle ${\mathscr C}_1$ with $|{\mathscr C}_1|=|{\mathscr C}|+1$ which is then connected to ${\mathscr C}$ in ${\mathbb C}$  by an upper-left edge. 
\end{enumerate}

The purpose of this section is to prove Theorem \ref{thm:main}, we shall proceed as follows. We shall first provide some background on discrete and continuous cell-systems, a notion that was briefly alluded to in the Introduction and plays a key role in our approach. We shall then focus on maximal segments in ${\mathbb C}$ formed by vertices connected only by upper-left edges, which we call left-twigs. Roughly speaking, we view left-twigs as cells, that grow, divide and finally die out as time passes, forming a discrete cell system. We shall obtain a first limit theorem in distribution for a rescaled version of this cell-system, and then derive a second one after a time-substitution similar to \eqref{eq:stoppingtimeboule} and Proposition \ref{prop:hauteur} in the branching peeling by layers algorithm. Finally, we shall show how Theorem \ref{thm:main} follows from the preceding results and the bounds for the mass of the lost cycles in the cutoff procedure (cf. Section~\ref{sec:L3}).

\subsection{Cell systems and a self-similar growth-fragmentation process}
\label{sec:cell-systems}

We start by adapting the definition of a cell-system from \cite{BeGF} to the integer-valued case, tailored for the purpose of this work. First, we call {cell chain} a Markov chain in continuous time  ${\rm C}=({\rm C}(t), t\geq 0)$ taking values in $\Z_+=\{0,1, \ldots\}$, which is right-continuous in the sense that $\Delta{\rm C}(t):={\rm C}(t)-{\rm C}(t-)\leq 1$ for all $t\geq 0$, and is eventually absorbed at $0$, i.e.~if $\zeta=\inf\{t\geq 0: {\rm C}(t)=0\}$, then $\zeta<\infty$ a.s. and ${\rm C}(t)=0$ for all $t>\zeta$. We should think of ${\rm C}$ as the process of the size of a typical cell.

 We next associate to a cell chain a {\it discrete cell system}  whose dynamics can be described as follows.
  We start at time $t=0$ from a single cell, whose size varies as time passes according to ${\rm C}$. We  interpret each negative jump of ${\rm C}$ occurring before absorption as a splitting  event, in the sense that whenever $t<\zeta$ and 
$\Delta {\rm C}(t)\coloneqq {\rm C}(t)-{\rm C}(t-)=-y<0$, the cell divides at time $t$ into the mother cell and its daughter. After the splitting event, the mother cell has size ${\rm C}(t)$ and the daughter cell has size $y+1$ (so the sum of the sizes of the mother and the daughter after the division event equals the size of the mother before the birth plus $1$). 
Assume that the evolution of the daughter cell is governed by the law of the same Markov chain (starting of course from $y+1$), and is independent of the processes of all the other daughter particles. And so on for the granddaughters, then great-granddaughters ... We stress that the final jump of a cell at the time when it gets absorbed at $0$ is never viewed as a splitting event. 

In order to encode mathematically the cell system,
 it is convenient to label cells by the nodes of the Ulam tree $\U=\bigcup_{n=0}^{\infty} \N^n$,
with the usual convention that $\N^0=\{\varnothing\}$. So  ${\mathcal C}_{\varnothing}$ is the process of the size of the ancestor cell, which is born at time $b_{\varnothing}=0$ and evolves according to the dynamics of the Markov chain ${\rm C}$. For every $u\in\U$ and $j\in\N$, the cell labelled by $uj$ is born at time $b_{uj}\coloneqq b_u+\beta_{uj}$, where $\beta_{uj}$ denotes  the instant of  $j$-th largest jump of the process $-{\mathcal C}_u$, and  for every $s\geq 0$, ${\mathcal C}_{uj}(s)$ represents the size of the cell $uj$ at age $s$, that is at  time $b_{uj}+s$. We implicitly agree that 
$b_{uj}=\infty$ and ${\mathcal C}_{uj}(s)\equiv 0$ when ${\mathcal C}_u$ has less than $j$ jumps.
We can then represent discrete  cell systems as a collection of processes indexed by the Ulam tree
$$ (({\mathcal C}_u, b_u), u\in\U).$$
We stress that this description is a bit redundant as the birth times $b_{u}$ for $u\in\U$ can be recovered from the processes ${\mathcal C}_v$, with $v\prec u$ denoting a generic (strict) ancestor of $u$. So by a slight abuse of terminology, we shall also call $({\mathcal C}_u, u\in\U)$ a cell system based on the cell chain ${\rm C}$.

The definition of a {\it continuous cell-system} $(({\mathcal C}_u, b_u), u\in\U)$ is essentially similar.  
The building block is a so-called cell process, that is now a Feller process ${\rm C}=({\rm C}(t), t\geq 0)$ with values in $[0,\infty)$, which is assumed to have only negative jumps and to be absorbed continuously at $0$. When the process of the size of a cell has a negative jump, say $\Delta{\rm C}(t):={\rm C}(t)-{\rm C}(t-)=-y<0$, the size of the daughter cell which is born at time $t$ is $y$ (whereas it was $y+1$ for discrete cell processes).  We refer to \cite{BeGF} for details.

We now turn our attention to a specific cell system which has a central role in this work. 
Recall that $X$ denotes the self-similar Markov process defined by \eqref{eq:defY}; we see $X$ as a  cell process and write ${\mathcal X}=({\mathcal X}_u: u\in\U)$
for the (continuous) cell system which stems from $X$. Recall further that the Laplace exponent $\Psi$ of $X$
is given by  \eqref{eq:defpsi}, and consider the function
\begin{eqnarray}
\kappa(q) &\coloneqq& \Psi(q) + \int_{1/2}^1 (1-x)^q(x(1-x))^{-5/2}\d x\nonumber \\
&=&-\frac{8}{3}q + \int_{1/2}^1 (x^q-1+q(1-x)+(1-x)^q)(x(1-x))^{-5/2}\d x\nonumber \\
& =&  \quad  \frac{4 \sqrt{\pi}}{3} \frac{\Gamma(q- \frac{3}{2})}{\Gamma(q-3)},  \label{eq:mireille}
\end{eqnarray}
where $\Gamma$ is the gamma function. The last integration has been performed (formally) with a computer algebra software. In particular we see that $\kappa$ is a convex function with values in $(-\infty, \infty]$, with $\kappa(2)=\kappa(3)=0$. So 
$\{q>0: \kappa(q)\leq 0\} =[2,3]$ and the conditions of Theorem 2 of \cite{BeGF} are fulfilled.

For every $u\in\U$ and $j\in\N$, recall that  $\beta_{uj}$ denotes the instant of the $j$-th largest jump of $-X_u$, and define
$$b_u\coloneqq \sum_{v \preceq u}\beta_v,$$ 
where the notation $v \preceq u$ is meant for $v$ ancestor of $u$ (possibly $v=u$) in $\U$, and by convention $\beta_{\varnothing}=0$. So $b_u$ is the birth time of the cell labelled by $u$. 
Then for every $t\geq 0$, the family 
$${\bf X}(t)\coloneqq \{ {\mathcal X}_u(t- b_u):  b_u\leq t, u\in\U \}$$
of the sizes of the cells which are alive at time $t$ is $q$-summable for every $q\in[2,3]$, and in particular, ranking the elements of this set in the decreasing order, we can -- and henceforth will -- view ${\bf X}$ as a random process with values in $\ell^{\downarrow}_3$.

\subsection{Scaling limit for cycle lengths in a peeling exploration}
\label{sec:cyclen}

Next, call {\it left-twig} in ${\mathbb C}$ 
a maximal sequence of vertices connected only by upper-left edges. In other words, a left-twig is a segment $[{\mathscr C}, {\mathscr C}']$ in ${\mathbb C}$ 
 where the right extremity ${\mathscr C}'$ is a leaf of ${\mathbb C}$, the left extremity ${\mathscr C}$ is either  the root or is connected to its parent by a upper-right edge, and all edges between adjacent cycles  in the segment $[{\mathscr C}, {\mathscr C}']$ are left-edges.
 The left-twig starting from the root thus corresponds to the chain of the locally largest cycle (see Section~\ref{sec:llc}), the other left-twigs start from some cycle ${\mathscr C}_2$ that results from case ($ \mathsf{G}$) above  (i.e.~${\mathscr C}_2$  is the second = right child of its parent), and then follows at each step the locally largest cycle in the descent of ${\mathscr C}_2$. 
 
We further attach to each cycle ${\mathscr C}$ in ${\mathbb C}$  an independent exponential variable $e_{{\mathscr C}}$ which we can think of as the time needed for the exploration of its distinguished triangle. 
We view each left-twig as an individual $u$ endowed with  some life career. Specifically, if $({\mathscr C}_1, \ldots, {\mathscr C}_j)$ is a left-twig labelled by $u$, 
then the  lifetime $\zeta_u$ of that individual  is given $\zeta_u=\sum_{i=1}^j e_{{\mathscr C}_i}$ and the size of that individual at age $t$  by
$${{\widetilde L}}_u(t) = |{\mathscr C}_k|\qquad \hbox{whenever }\ \sum_{i=1}^{k-1} e_{{\mathscr C}_i}\leq t < \sum_{i=1}^{k}e_{{\mathscr C}_i} \leq \zeta$$
and ${{\widetilde L}}_u(t)=0$ for $t\geq \zeta$.  

In turn, the genealogical tree ${\mathbb C}$ of the cycles  induces a tree structure on the family of left-twigs. Specifically, we use Ulam's notation to label the left-twigs, also called individuals in the sequel, as follows. 
First, the individual corresponding to the left-twig which has the root of ${\mathbb C}$ as left-extremity, is viewed as the ancestor and hence labelled by $\varnothing$. The process of its size as time passes is denoted by ${{\widetilde L}}_{\varnothing}=({{\widetilde L}}_{\varnothing}(t), t\geq0)$. 

The children of $\varnothing$ form individuals at the first generation, they correspond {to} the left-twigs at distance $1$ from the left-twig $\varnothing$ in ${\mathbb C}$. More precisely, the ancestor $\varnothing$ begets children during its lifetime: each time $t>0$ at which ${{\widetilde L}}_{\varnothing}$ makes a {non positive} jump $\Delta {{\widetilde L}}_{\varnothing}(t)= {{\widetilde L}}_{\varnothing}(t)-{{\widetilde L}}_{\varnothing}(t-)\leq {0} $  corresponds to the birth of a child which has then initial size $1-\Delta {{\widetilde L}}_{\varnothing}(t)$. 
The children of $\varnothing$ are labelled $1,2, \ldots$ in the decreasing order of their sizes at birth (i.e.~the perimeter of the first cycle on that left-twig), and, say,  in increasing order of their birth time in case of ties. We agree that ${{\widetilde L}}_n\equiv 0$ when $\varnothing$ has less than $n$ children, and  iterate in an obvious way for the next generations. Finally, we obtain a labelling  of the left-twigs by $\mathbb{U}$ as well as a family of processes indexed by the Ulam tree
$$({{\widetilde L}}_u: u\in\U).$$

The next statement concerns convergence in distribution, in the sense of finite dimensional distributions,  for a sequence of  families of processes. Specifically, let $E$ be some countable set, and
consider for every $ e\in E$ a sequence of c\`adl\`ag real-valued processes $\eta^{(n)}_e=(\eta^{(n)}_e(t), t\geq 0)$, $n\in\N$.
We shall write
 $$(\eta^{(n)}_e:  e\in  E)  \quad \xrightarrow[n\to\infty]{(d)} \quad  (\eta_e:  e\in  E)$$ 
 provided that for every finite subset $F\subset  E$,
 the multivariate process $((\eta^{(n)}_e(t))_{ e\in F}: t\geq 0)$ converges in distribution in the sense of  Skorokhod towards $((\eta_e(t))_{ e\in F}: t\geq 0)$.
We shall use the notation $\xrightarrow[p\to\infty]{(d)}$ for weak convergence in the sense explained above, where the distribution on the left-hand side is implicitly considered under $\P^{(p)}$. Last, recall the definition of the self-similar Markov process $\widetilde X=(\widetilde X(t), t\geq 0)$ which has been introduced in Section~2.6. Further, if $Y=(Y(t): t\geq 0)$ is a stochastic process, we shall use the notation $Y(c\times \cdot)$ for the process rescaled in time by a factor $c$, that is $Y(c\times \cdot)=(Y(ct): t\geq 0)$.

\begin{lemma} \label{Lj1} Under $\P^{(p)}$, $({{\widetilde L}}_u: u\in\U)$ is a discrete cell-system; the  associated cell process is distributed as the continuous-time version (i.e. subordinated by an independent standard Poisson process) of the chain of the locally largest cycle ${\mathscr C}^*$.

Further, consider the self-similar cell process $(X'(t)= \widetilde X(2 \tone t ))_{ t\geq 0}$ and let 
$({\mathcal X}'_u, u\in\U)$ denote the corresponding continuous cell system.
Then there is the weak convergence of the rescaled systems
$$( p^{-1}{{\widetilde L}}_u(p^{3/2}\times \cdot): u\in\U) \quad \xrightarrow[p\to\infty]{(d)}  \quad  ({\mathcal X}'_u, u\in\U).$$
\end{lemma}

\begin{proof} We first fix $p\geq 1$ and work under $\P^{(p)}$. {It is  convenient using to consider a deterministic peeling algorithm ${\mathcal A}'$ which induces the same genealogical tree of cycles ${\mathbb C}$, and such that ${\mathcal A}'$  explores first completely the ancestral left-twig labelled by $\varnothing$, then the left-twigs of the first generation in their specified  order, and so on, generation by generation.} The exploration process of the ancestral left-twig is precisely described by the chain of the locally largest cycle discussed in Section~\ref{sec:llc}, and thus ${{\widetilde L}}_{\varnothing}=({{\widetilde L}}_{\varnothing}(t), t\geq 0)$ is the continuous time Markov chain obtained by subordinating the locally largest cycle chain $({{\widetilde L}}(k): k\geq 0)$ with an independent Poisson process with unit rate. 
If we stop the peeling algorithm ${\mathcal A}'$ once the ancestral left-twig $\varnothing$ has been completely searched, Corollary \ref{cor:peelinggenerallaw} yields that given ${{\widetilde L}}_{\varnothing}$, the processes at the first generation  ${{\widetilde L}}_1, {{\widetilde L}}_2, \ldots$ are independent, and more precisely ${{\widetilde L}}_i$ has the law of ${{\widetilde L}}_{\varnothing}$ under $\P^{(p_i)}$,
where $p_i-1\geq 1$ is the size of the $i$-th largest jump of $-{{\widetilde L}}_{\varnothing}$ whenever the latter has at least $i$ positive jumps, and $p_i=0$ otherwise. By iteration, we conclude that under  $\P^{(p)}$, $({{\widetilde L}}_u: u\in\U)$ is a discrete cell system induced by the cell chain ${{\widetilde L}}_{\varnothing}$.

It is convenient at this point to comment on a seemingly weaker notion of convergence for a sequence of a family of processes indexed by some countable set $E$.
For each $e\in E$, consider  a sequence of real-valued c\`adl\`ag processes $\eta^{(n)}_e=(\eta^{(n)}_e(t), t\geq 0)$, $n\in\N$.
We write
\begin{equation}\label{eq:quasicv}
(\eta^{(n)}_e: u\in  E) \quad  \xrightarrow[n\to\infty]{(d*)}  \quad  (\eta_e: u\in  E)
\end{equation}
 when for every finite subset $F\subset  E$ and finite time interval $[0,t]$,
 we can find for each $e\in F$ (random) strictly increasing continuous bijections $\sigma_e^{(n)}: [0,t]\to [0,t]$
 with 
 $$\lim_{n\to\infty}\sup_{0\leq s\leq t}|s-\sigma_e^{(n)}(s)| = 0 \qquad \hbox{in probability},$$ 
 such that for every family of bounded continuous functionals $(\Phi_e)_{e\in F}$ on the space of c\`adl\`ag functions on $[0,t]$ endowed with the supremum distance, 
 $$\lim_{n\to \infty} \E\left( \prod_{e\in F} \Phi_e(\eta^{(n)}_e\circ \sigma_e^{(n)})\right) =
 \E\left( \prod_{e\in F} \Phi_e(\eta_e)\right).$$
 We stress that this is {\it a priori} weaker than that of convergence in the sense of final dimensional distributions stated in Lemma \ref{Lj1}, because here we may use different time-changes $\sigma_u^{(n)}$ for different $e\in \U$, whereas we would need to use the same time change for all $u\in \U$  for the (joint) convergence in Skorokhod sense. However, we point out that if the processes $\eta_u$ have no common jump times a.s., then Proposition 2.2 on page 338 in \cite{JS03} shows that the convergence \eqref{eq:quasicv} can then be reinforced as 
$$(\eta^{(n)}_u: u\in \U) \quad  \xrightarrow[n\to\infty]{(d)} \quad  (\eta_u: u\in \U).$$

Next, recall from Proposition \ref{prop:scalingllcycle} that the rescaled chain $p^{-1}{{\widetilde L}}_{\varnothing}(p^{3/2}\times \cdot)$ converges in law  in the sense of Skorokhod for c\`adl\`ag processes, towards the self-similar process $X'$. Observe further that both $p^{-1}{{\widetilde L}}_{\varnothing}(p^{3/2}\times \cdot)$ and $X'$
attain the absorbing state $0$ at a finite time, and that  the convergence holds even when we include time infinity
(i.e. in the sense of c\`adl\`ag processes indexed by the compact time-interval $[0,\infty]$), because, just as in the proof of Proposition \ref{prop:hauteur}, for every $\varepsilon>0$, we can make the $\P^{(p)}$ probability that $p^{-1}{{\widetilde L}}_{\varnothing}(p^{3/2}\times \cdot)$ exceeds $c\varepsilon$ after entering $[0,\varepsilon]$ as small as we wish uniformly in $p$, by choosing $c>0$ sufficiently large.
Since convergence in Skorokhod sense for c\`adl\`ag processes indexed by $[0,\infty]$ implies the weak convergence (in the sense of finite dimensional distributions) of the sequence of the jump sizes ranked in the decreasing order, and since $X'$ has the Feller property (so its distribution depends continuously on its starting point), we now see that as $p\to \infty$, the law of the sequence of rescaled processes 
$(p^{-1}{{\widetilde L}}_{1}(p^{3/2}\times \cdot), p^{-1}{{\widetilde L}}_{2}(p^{3/2}\times \cdot), \ldots)$ under $\P^{(p)}$ converges weakly, in the sense $(d*)$  explained above, towards that of the first generation $({\mathcal X}'_1, {\mathcal X}'_2, \ldots)$ of a cell process induced by the self-similar cell process $X'$. More precisely, this also holds jointly with the weak convergence of $p^{-1}{{\widetilde L}}_{\varnothing}(p^{3/2}\times \cdot)$ towards ${\mathcal X}'_{\varnothing}$.
By iteration, we now see that there is the weak convergence of the rescaled systems
\begin{equation}\label{eq:preskcv}( p^{-1}{{\widetilde L}}_u(p^{3/2}\times \cdot): u\in\U) \quad  \xrightarrow[p\to\infty]{(d*)} \quad  ({\mathcal X}'_u, u\in\U).
\end{equation}
It is readily checked that the processes ${\mathcal X}'_u$ have no common jump times a.s., so, just as observed  above,  Proposition 2.2 on page 338 in \cite{JS03} shows that 
\eqref{eq:preskcv} entails our statement. 
\end{proof} 

\subsection{Scaling limit for cycle lengths in branching peeling by layers}

Recall  that the peeling algorithm we consider is of the type of
 branching peeling by layers which has been described in Section~\ref{sec:branchedpeelinglayers}. 
 If  $u \in \mathbb{U}$ labels a  left-twig $[{\mathscr C}, {\mathscr C}']$ and if $r \geq 0$,
 we set ${\mathscr C}_{u}={\mathscr C}$ and let  ${\mathcal L}_u(r)$ denote the length of the first cycle belonging to $[{\mathscr C},{\mathscr C}']$  which has all its vertices at distance at least $r $ from ${\mathscr C}_u$ (if any), where distances are measured in the triangulation ${\bf t}$. If there is no such cycle, then set ${\mathcal L}_u(r)=0$. 
  That is 
 \begin{equation}\label{eq:luc}
 {\mathcal L}_u(r) = {{\widetilde L}}_u(\theta_u(r)),
 \end{equation}
 where $\theta_u(r)$ is the first time when the cycle labelled by $u$ is at distance at least $r$ from the initial cycle ${\mathscr C}_u$ whenever there exists such a cycle, or $\theta_u(r)=\infty$ otherwise. We also set $h_{\varnothing}\equiv 0$. For every $u\in\U$ and $i\in\N$, 
we write $h_{ui}$ for the distance between ${\mathscr C}_{ui}$ and $ {\mathscr C}_{u}$ (again measured in the triangulation ${\bf t}$) whenever ${\mathscr C}_{ui}$ is non-empty, and 
$h_{ui}=\infty$ otherwise.

Next, recall from \eqref{eq:defY} the definition of the self-similar Markov process ${{X}}$, and, following Proposition \ref{prop:hauteur}, consider 
$$\bar X(t)= X  \left(  \frac{2 \tone}{\aone} \cdot t \right)  \,, \qquad t\geq 0.$$
We view $\bar X$ as a  cell process and write $\bar {\mathcal X}=(\bar {\mathcal X}_u: u\in\U)$
for the (continuous) cell system which stems from $\bar X$. We also write $\bar \beta_{uj}$ for  the instant of the $j$-th largest jump of $\bar {{\mathcal X}}_u$.
 The following claim is essentially a branching extension of Proposition \ref{prop:hauteur}.                     

\begin{corollary} \label{Cj1} There is the weak convergence
$$((p^{-1}{\mathcal L}_u(\sqrt p\times \cdot),p^{-1/2}h_u): u\in\U)
\quad  \xrightarrow[p\to\infty]{(d)} \quad 
((\bar {\mathcal X}_u,\bar \beta_u): u\in\U).$$ 
\end{corollary}

\begin{proof} Indeed, we know from Lemma \ref{Lj1} that under  $\P^{(p)}$, $({{\widetilde L}}_u: u\in\U)$ is a discrete cell system associated to the cell process ${{\widetilde L}}_{\varnothing}$. In this setting, Proposition \ref{prop:hauteur} can be rephrased as the weak convergence of the stopped cell processes
$$p^{-1}{\mathcal L}_{\varnothing}(\sqrt p\times \cdot) \quad  \xrightarrow[p\to\infty]{(d)}\quad 
\bar {\mathcal X}_{\varnothing}(\cdot).$$
It follows from Lemma \ref{Lj1}, the Markov property and the arguments used to prove Proposition  \ref{prop:hauteur} that more generally, there is the weak convergence in the sense of \eqref{eq:quasicv}
$$(p^{-1}{\mathcal L}_u(\sqrt p\times \cdot): u\in\U)  \quad \xrightarrow[p\to\infty]{(d*)} \quad 
(\bar {\mathcal X}_u: u\in\U).$$
The same argument using the absence of common jump times  as in the proof of from Lemma \ref{Lj1} enables us to replace the convergence in the sense $(d*)$ above by the stronger $(d)$. 
We can then complete the proof by considering the instant $p^{-1/2}h_{uj}$ (respectively, $\bar \beta_{uj}$) of the $j$-th largest jump of the process ${\mathcal L}_u(\sqrt p\times \cdot)$ (respectively, $\bar {\mathcal X}_u(\cdot)$). 
\end{proof} 

Next, we set for every $u\in \U$
$$H_u\coloneqq \sum_{v \preceq u}h_v,$$
where the notation $v \preceq u$ is meant for $v$ ancestor of $u$ in $\U$. Recall that the left-extremity of the left-twig labelled by $u$ is the cycle denoted by ${\mathscr C}_u$.
Then observe that  the distance measured in ${\mathbf t}$ between the $p$-gon and ${\mathscr C}_u$ (we implicitly agree that this distance is infinite whenever ${\mathscr C}_u$ is empty) can be expressed as $H_u+O(|u|)$, where the error term  $O(|u|)$
fulfills  $0\leq O(|u|)\leq |u|$. More precisely, these bounds follow from the fact our requirement on the peeling algorithm ensure that the distance (measured in ${\mathbf t}$) between any vertex of a cycle ${\mathscr C}_v$
and any vertex of the parent cycle ${\mathscr C}_{v-}$ is either $h_v$ or $h_v+1$. Hence, if we define 
$L_u(r)$ as the length of the cycle of the ball $B_r(T^{(p)})$ which is indexed by the left-twig $u$ (if any, and $L_u(r)=0$ otherwise), then there is the identity
$$L_u(r)= {\bf 1}_{\{r \leq H_u+O(|u|) \}}
{{\mathcal L}}_u(r-H_u-O(|u|)).$$

Set also 
$$b_u\coloneqq \sum_{v \preceq u}\beta'_v$$ 
for the birth-time of the cell $\bar {\mathcal{X}}_u$ and then 
$$\bar {X}_u(r)= {\bf 1}_{\{ b_u\leq  r \}}\bar{\mathcal{X}}_u(r- b_u).$$
In words, $\bar {X}_u(r)$ is the size of the cell labeled by $u$ at time $r$ (that is when its age is $r- b_u$) provided that it is already born at that time, and $0$ otherwise. We now immediately deduce from Corollary \ref{Cj1} the following.

\begin{corollary}\label{Cj2} There is the weak convergence
$$(p^{-1}{ L}_u(\sqrt p\times \cdot): u\in\U)
\quad  \xrightarrow[p\to\infty]{(d)}  \quad 
(\bar {{X}}_u(\cdot): u\in\U).$$
\end{corollary}

\subsection{Proof of Theorem 1}

We still need to introduce a few definitions and technical estimates.
For  $k\geq 1$, let $\U_k$ denote the $k$-regular tree with height $k$, that is  $\U_k\coloneqq \bigcup_{i=0}^k[k]^i$ with $[k]=\{1,\ldots, k\}$. 

\begin{lemma} \label{Lj5} We have
$$\lim_{k\to \infty}  \sup_{t\geq 0}\sum_{u\in\U\backslash \U_k}  \bar {\mathcal X}^3_u(t)= 0 \qquad \hbox{a.s.}$$

\end{lemma}

\begin{proof} Pick any $2<q<3$, so that  $ \kappa(q) <0$. Corollary 4 in  \cite{BeGF} shows that 
$$\E\left( \sum_{u\in\U} \sup_{t\geq 0} \bar {\mathcal X}^q_u(t)\right) <\infty,$$
and as a consequence 
$$\lim_{k\to \infty}  \sum_{u\in\U\backslash \U_k} \sup_{t\geq 0} \bar {\mathcal X}^q_u(t) =0\qquad \hbox{a.s.}$$
This readily entails our claim. 
\end{proof} 

Next, we fix $p\geq 1$ and work under $\P^{(p)}$. For every $\varepsilon>0$,  we say that an individual $u\in\U$ is $(\varepsilon,p)$-good and then write $u\in{\mathbb G}(\varepsilon,p)$ if and only if the perimeter of the initial cycle of the left-twig labelled by each of its ancestors (including $u$ itself) is greater than $\varepsilon p$, i.e. 
$$ {\mathcal L}_v(0)>\varepsilon p \qquad \hbox{for all }v\preceq u.$$

\begin{lemma} \label{Lj7} We have for every $\varepsilon >0$ that
$$\lim_{k\to \infty} \lim _{p\to \infty} \P^{(p)}\left( {\mathbb G}(\varepsilon,p) \subseteq  \U_k \right) = 1.$$

\end{lemma}

\begin{proof}  We know from 
Lemma 3 in  \cite{BeGF} that 
$$\E\left( \sum_{u\in\U}  \bar {\mathcal X}^q_u(0)\right) <\infty$$
for every $2<q<3$, and as a consequence 
$$\lim_{k\to \infty}  \sum_{u\in\U\backslash \U_k}  \bar {\mathcal X}^q_u(0) =0\qquad \hbox{a.s.}$$
It follows that for every $\varepsilon >0$, 
$$\lim_{k\to \infty} \P\left( \exists u\in \U\backslash \U_k: \bar {\mathcal X}_u(0) > \varepsilon/2\right) = 0.$$

Write $\partial \U_k$ for the set of individuals $u=(u_1, \ldots) \in\U$ with either $|u|=k+1$ and $u_i\leq k$ for all $i=1, \ldots, k+1$, or $|u|\leq k$, $ u_{|u|}=k+1$ and 
$u_i \leq  k$ for every $i <|u|$.
So $\partial \U_k$ is a finite subset of $\U\backslash \U_k$, and we deduce from above and Corollary \ref{Cj1} that \begin{equation} \label{presquebon}
\lim_{k\to \infty} \lim_{p\to \infty}\P^{(p)} \left(\exists u\in   \partial \U_k: {\mathcal L}_u(0) > \varepsilon p\right) =0. 
\end{equation}

 Now suppose $v=(v_1, \ldots)\in \U\backslash \U_k$ is $(\varepsilon,p)$-good, that is   ${\mathcal L}_w(0)\geq \varepsilon p$
for all $w\preceq v$. Consider first the case 
where $|v|\geq k+1$ and $v_i\leq k$ for every $i=1, \ldots, k+1$, then $(v_1, \ldots, v_{k+1})$ is an ancestor of $v$ which belongs to $\partial \U_k$.
Next consider the complementary case, so
$$j\coloneqq\inf\{i\geq 1: v_i\geq k+1\}\leq k+1;$$
then $w=(v_1, \ldots, v_j)$ is an ancestor of $v$ and thus ${\mathcal L}_w(0)> \varepsilon p$. Note that $w'=(v_1, \ldots, v_{j-1}, k+1)\in \partial \U_k$ and because children are listed in the decreasing order of their sizes at birth and $v_j\geq k+1$, we have also ${\mathcal L}_{w'}(0)\geq{\mathcal L}_{w}(0) >  \varepsilon p$. Summarizing, if $v\in \U\backslash \U_k$ is $(\varepsilon,p)$-good, then there exists 
$u\in\partial \U_k$ with ${\mathcal L}_u(0) > \varepsilon p$. Our claim thus follows from \eqref{presquebon}. 
\end{proof}

We can now deduce from  Proposition \ref{prop:l3lostcycles} the following limit.

\begin{lemma} \label{Lj6} We have for every $\delta >0$ that
$$\lim_{k\to \infty} \lim_{p\to \infty} \P^{(p)}\left( \sup_{t\geq 0} \sum_{u\in\U\backslash \U_k}  p^{-3} L^3_u(t )>\delta \right) = 0.$$

\end{lemma}

\begin{proof} Let $p\geq 1$ be fixed and work under $\P^{(p)}$. Note that  the family 
$\{L_u(t): u\in \U\backslash {\mathbb G}(\varepsilon,p)\}$ is contained in the family of perimeters of the lost cycles  at cutoff level $\varepsilon p$ and height $t $ introduced in Section~\ref{sec:L3}. 
 Thanks to Proposition \ref{prop:l3lostcycles}, for every $\eta >0$, we may  choose  $\varepsilon >0$ sufficiently small so that
$$\sup_{p\geq 1} \P^{(p)}\left( \sup_{t\geq 0} \sum_{u \in\U\backslash {\mathbb G}(\varepsilon,p)} p^{-3} L^3_u(t)>\delta \right) < \eta.$$
Then, thanks to Lemma \ref{Lj7}, for every $k$ sufficiently large, we may choose $p_k$ such that 
$$\sup_{p\geq p_k} \P^{(p)}\left( {\mathbb G}(\varepsilon,p) \not \subseteq  \U_k \right) <\eta, $$
and then
$$\sup_{p\geq p_k} \P^{(p)}\left( \sup_{t\geq 0} \sum_{u \in\U\backslash \U_k} p^{-3} L^3_u(t)>\delta \right) < 2\eta.$$
As $\eta$ is arbitrarily small, this proves our claim. 
\end{proof}

We are now able to establish Theorem \ref{thm:main}.
\begin{proof}  We now view the families $(L_u(t): u\in\U)$ and $(\bar X_u(t): u\in\U)$ as random variables in $\ell^3(\U)$, and thus
 $(L_u(t): u\in\U)_{t\geq 0}$ and $(\bar X_u(t): u\in\U)_{t\geq 0}$ as c\`adl\`ag processes with values in the complete metric space $\ell^3(\U)$ (see Corollary 4 in \cite{BeGF}). 
 Using Lemmas \ref{Lj5} and \ref{Lj6}, it is  now straightforward  to reinforce the weak convergence in the sense of finite-dimensional distributions stated in Corollary \ref{Cj2},
to weak convergence in the sense of Skorokhod for c\`adl\`ag processes with values in $\ell^3(\U)$. 

Comparing the definitions and notation of the preceding section and that of Theorem \ref{thm:main},  we see that
${\bf X} \left(  \frac{2\tone}{\aone} \times t \right)$ is obtained by ranking the elements of the family $\{\bar X_u(t): u\in\U\}$ in the decreasing order. 
Since this operation decreases the $\ell^3$-distance (see, e.g.~Theorem 3.5 in \cite{LL}), Theorem \ref{thm:main} thus follows from above. 
\end{proof}

\section{Metric approximation by the cut-off}
\label{bigsec:proofthm2}

If $T^{(p)}$ is a Boltzmann triangulation of the $p$-gon, recall that  $ \mathsf{Cut}(T^{(p)}, \varepsilon p)$ denotes the triangulation with holes obtained by performing on $T^{(p)}$ a branching peeling by layers exploration, frozen below level $ \varepsilon p$. In the proof of Theorem \ref{thm:main}, we have seen, roughly speaking, that  $ \mathsf{Cut}(T^{(p)}, \varepsilon p)$ is, asymptotically, a good approximation of $T^{(p)}$ in the $\ell^{3}$ sense, meaning that the sum of the cubes of the length of the cycles at heights of $ T^{(p)} \backslash  \mathsf{Cut}(T^{(p)}, \varepsilon p)$  becomes negligible. 

The goal of this section is to establish Theorem \ref{thm:no tentacles}, which tells us that $ \mathsf{Cut}(T^{(p)}, \varepsilon p)$ is, asymptotically, a good approximation of $T^{(p)}$ also in the metric sense, or, as explained in the end of the Introduction, that asymptotically there are no ``long  and thin tentacles'' in $ T^{(p)} \backslash  \mathsf{Cut}(T^{(p)}, \varepsilon p)$. To this end, if $ \mathsf{Height}( \mathbf{t})$  denotes the maximal height of a vertex of a triangulation with a boundary $ \mathbf{t}$, we will show that for any $\delta>0$ we have 
  \begin{eqnarray} \label{eq:goaltheoremtentacles} \sup_{p \geq 1}  \mathbb{P}\left(  \sup_{i \geq 1}\mathsf{Height}(  \mathrm{T}_{i}^{< \varepsilon p}) \geq \delta \sqrt{p}\right) \quad  \xrightarrow[ \varepsilon\to0] \quad  0,  \end{eqnarray}
where we recall that $  \mathrm{T}_{1}^{< \varepsilon p}, \mathrm{T}_{2}^{ < \varepsilon p}, \ldots$ are the triangulations with boundaries that fill-in the holes of $ \mathsf{Cut}( \mathbf{t}, \varepsilon p)$ in $T^{(p)}$. The last display  clearly entails Theorem \ref{thm:no tentacles}.

\subsection{A first approach}
The convergence \eqref{eq:goaltheoremtentacles} would readily follow from Lemma \ref{lem:epsexplo} if the following estimate was established:

\begin{conjecture} \label{open:hauteur} As $\lambda \rightarrow \infty$,
$$ \sup_{p \geq 1}\mathbb{P}^{(p)}(  \mathsf{Height}( \mathbf{t}) \geq \lambda \sqrt{p} ) = O( \lambda^{-6}).$$
\end{conjecture}
Indeed, assuming Conjecture \ref{open:hauteur}, we can proceed as in the proofs of Proposition \ref{prop:l3lostcycles} or \ref{prop:volumesmall} to establish \eqref{eq:goaltheoremtentacles}: recalling that $  \mathrm{T}_{1}^{< \varepsilon p}, \mathrm{T}_{2}^{ < \varepsilon p}, \ldots$ denote the triangulations with boundaries that fill-in the holes of $ \mathsf{Cut}( \mathbf{t}, \varepsilon p)$, and that $ \mathsf{C}_{1}^{< \varepsilon p}, \mathsf{C}_{2}^{< \varepsilon p}, \ldots$ are their boundaries, by the Markovian structure of the peeling algorithm, under $\mathbb{P}^{(p)}$, the components of $ \mathbf{t} \backslash \mathsf{Cut}(\mathbf{t}, \varepsilon p)$ are, conditionally given $ \mathsf{Cut}(\mathbf{t}, \varepsilon p)$, independent Boltzmann triangulations with perimeters $ | \mathsf{C}_{1}^{ < \varepsilon p}|,| \mathsf{C}_{2}^{ < \varepsilon p}| , \ldots $. As a consequence, if Conjecture \ref{open:hauteur} holds, we would have
  \begin{eqnarray*} \Prp{ \exists \ i \geq 1 : \mathsf{Height}( \mathrm{T}_{i}^{< \varepsilon p}} \geq \delta \sqrt{p}) &\leq& \Esp{\sum_{i \geq1} \mathbb{P}^{(| \mathsf{C}_{i}^{< \varepsilon p}|)}( \mathsf{Height}( \mathbf{t}) \geq \delta \sqrt{p}) }\\ &=&\mathbb{E}\left[ \sum_{i \geq1} \mathbb{P}^{(| \mathsf{C}_{i}^{< \varepsilon p}|)}( \mathsf{Height}( \mathbf{t}) \geq   \frac{\delta \sqrt{p}}{ \sqrt{| \mathsf{C}_{i}^{< \varepsilon p}|}} \cdot \sqrt{ |\mathsf{C}_{i}^{< \varepsilon p}|}) \right]  \\
 &\underset{ \mathrm{Conj.}}{\leq}&  C \delta^{-6}   \,\mathbb{E}\left[  \sum_{i \geq 1}   \frac{| \mathsf{C}_{i}^{< \varepsilon p}|^3}{p^3}\right],\end{eqnarray*}
  for some constant $C >0$.  Letting $ \varepsilon \to 0$, we would get that \eqref{eq:goaltheoremtentacles}  holds by another appeal to  Lemma \ref{lem:epsexplo}. \medskip 
  
  However, we have not been able to establish the estimate of Conjecture \ref{open:hauteur}, and were forced to take a different path which we now explain. To control the metric structure of $T^{(p)} \backslash \mathsf{Cut}(T^{(p)}, \varepsilon p)$ it is necessary to control \emph{uniformly} the geometry of $T^{(p)}$. A natural approach would be to use variants of Schaeffer's bijection, which are usually used to obtain uniform controls on the geometry of random planar maps. However, in our particular case of triangulations with simple boundaries, such bijective techniques seem to be not very well adapted, since the topological constraint imposed on the boundary is not simply expressed in terms of labeled trees that code these maps.
  
 For this reason, we proceed as follows. We first rely on the volume estimate of Proposition \ref{prop:volumesmall} which shows that the maximal volume of the triangulations $(\mathrm{T}_{i}^{<  \varepsilon p})_{i \geq 1}$ is small compared to the total volume of $T^{(p)}$ which is of order $p^2$. We then argue that the diameter of a ball of volume $o(p^2)$ inside a large triangulation cannot be  of order  $ \sqrt{p}$ since ``$ \mathrm{volume}^{1/4} = \mathrm{distance}$''. It is possible to make  this last heuristic precise for large triangulations  of the sphere (Proposition \ref{prop:volumedistancetrig}), using the aforementioned bijective techniques. In this setup we actually know much more since convergence towards the Brownian map has been established by Le Gall \cite{LG11}. In the case of large Boltzmann triangulations of the $p$-gon, a similar convergence is expected  towards ``the Boltzmann Brownian disk'' \cite{Bet11,BM15}, but has not yet appeared {(one of the difficulties being the fact that we work with simple boundaries)}. We will thus bypass this gap by establishing a coupling tailored to our case that enables  us to embed $ T^{(p)}$ in a triangulation of the sphere with volume $\geq p^2$, and which allow us to transfer known estimates for uniform triangulations of the sphere to Boltzmann triangulations with a boundary.

\subsection{Uniform volume-distance estimates on triangulations of the sphere}

Let $T_{n}$ be a uniform triangulation of the sphere with $n$ vertices. If $\mu_{n}$ is the uniform measure on the vertices $V(T_{n})$ of $T_{n}$,  Le Gall   \cite{LG11} showed that 
 \begin{eqnarray} \label{eq:GHPtrig} ( \mathsf{V}(T_{n}), n^{-1/4}\cdot \mathrm{d_{gr}}, \mu_{n}) \quad  \xrightarrow[n\to\infty]{(d)} \quad  3^{-1/4} \cdot ( \mathbf{m}_{\infty}, D, \mu),  \end{eqnarray} where $( \mathbf{m}_{\infty}, D, \mu)$ is the so-called Brownian map endowed with its natural mass measure $\mu$, and where the convergence holds in distribution in the Gromov--Hausdorff--Prokhorov sense. Actually, \cite{LG11} only states the convergence for the Gromov--Hausdorff topology but the latter easily follows from arguments already in \cite{LG11}, see also \cite{CLGmodif}. 
 We will actually only need the fact, which follows from \eqref{eq:GHPtrig}, that any subsequential limit $( \mathbf{m}, D, \mu)$ of $( \mathsf{V}(T_{n}), n^{-1/4}\cdot \mathrm{d_{gr}}, \mu_{n})$ has a mass measure $\mu$ of full topological support. This property can be for instance seen by using the construction of the Brownian map as a quotient of the Brownian Continuum Random Tree $ \mathcal{T}_{ \mathbf{e}}$ by a certain equivalence relation \cite{LG11}. In this construction, the mass measure on the Brownian map $  \mathbf{m}_{\infty}$ is the push-forward of the mass measure on the Brownian CRT. Since the projection $\pi : \mathcal{T}_{ \mathbf{e}} \to \mathbf{m}_{\infty}$ is continuous and since the mass measure on $ \mathcal{T}_{ \mathbf{e}}$ has full support we deduce that indeed $\mu$ has full support in $ \mathbf{m}_{\infty}$.
 
 \begin{proposition} \label{prop:volumedistancetrig}For every $ \varepsilon>0$, we have
 $$\lim_{ \delta \to 0} \sup_{n \geq 0}  \mathbb{P}\big( \exists x \in \mathrm{V}(T_{n}) : |B_{ \varepsilon n^{1/4}}(x)| \leq \delta n\big) =0.$$
 \end{proposition}
\begin{proof} We argue by contradiction, and assume that there exists  $c>0$ and a sequence of $n_{k}\to \infty$ and $\delta_{k} \to 0$ such that, for every $k \geq0$, with probability at least $c$, there exists $x \in V(T_{n_{k}})$ whose ball of radius $ \varepsilon n_{k}^{1/4}$ has volume less than $\delta_{k}\cdot n_{k}$. Using \eqref{eq:GHPtrig}, we get that with probability at least $c>0$, we can find a point $x$ in the Brownian map $ \mathbf{m}_{\infty}$ such that its ball of radius $ 3^{1/4}\cdot \varepsilon/2$ has zero mass for $\mu$. This is absurd since the random measure $\mu$ almost surely has full topological support inside the Brownian map.\end{proof}

Our goal is now to establish an analog of Proposition \ref{prop:volumedistancetrig} for Boltzmann triangulations with a boundary.  As was previously mentioned, the analog of \eqref{eq:GHPtrig} is not yet known in this case, so we will use a different argument that involves coupling.

\subsection{Coupling triangulations with boundary and triangulations of the sphere}

Roughly speaking, the main idea is to prove that, with positive probability, a Boltzmann triangulation of the $p$-gon can be seen as a macroscopic part of a uniform triangulation of the sphere with roughly $p^2$ vertices.  Recall that by root-transformation, a triangulation of the sphere with $n$ vertices can be seen as a triangulation of the $1$-gon with $n-1$ inner vertices. Denote by  $T_{> n}$ a Boltzmann triangulation of the $1$-gon conditioned on having at least $n$ internal vertices.
In particular, conditionally on $T_{>n}$ having exactly $m \geq n$ inner vertices, the triangulation $T_{>n}$ is distributed as a uniform triangulation of the $1$-gon with $m$ inner vertices and can thus be seen as a uniform triangulation of the sphere with $m+1$ vertices. 

 \begin{lemma} \label{lem:volumedistance>n}For every $ \varepsilon>0$, we have
 $$\lim_{ \delta \to 0} \sup_{n \geq 0}  \mathbb{P}\big( \exists x \in \mathrm{V}(T_{>n}) : |B_{ \varepsilon n^{1/4}}(x)| \leq \delta n\big) =0.$$
 \end{lemma}
 \begin{proof} For fixed $\alpha  \geq 1$, write 
 \begin{equation}
 \label{eq:vdist}\mathbb{P}\big( \exists x \in \mathrm{V}(T_{>n}) : |B_{ \varepsilon n^{1/4}}(x)| \leq \delta n\big) \leq \mathbb{P}\big(|T_{>n}| \geq \alpha n) + \sup_{n \leq k \leq \alpha n} \mathbb{P}\big( \exists x \in \mathrm{V}(T_{k}) : |B_{ \varepsilon n^{1/4}}(x)| \leq \delta n\big).
 \end{equation}Then, by \eqref{eq:tnpexact}, we have for $x \geq 1$,
$$\mathbb{P}(|T_{>n}| \geq x n) = \frac{\mathbb{P}^{(1)}( |\mathbf{t}| \geq x n )}{\mathbb{P}^{(1)}( |\mathbf{t}| \geq n )}  \quad     \mathop{\sim}_{n \rightarrow \infty}    \quad  \frac{ \frac{2}{3}C(1) (xn)^{-3/2} }{Z(1)} \frac{Z(1)}{ \frac{2}{3}C(1) n^{-3/2} } = x^{-3/2}.$$
As a consequence, the first term in the right-hand side of \eqref{eq:vdist} is asymptotically less than $\alpha^{-3/2}$ whereas the second one can be made arbitrarily small once $\alpha$ is fixed by letting $\delta \to 0$ and using Proposition \ref{prop:volumedistancetrig}. This completes the proof.
 \end{proof}
 
 \begin{lemma}  \label{lem:coupling}There exists an event $ \mathcal{E}_{p}$ with $\liminf_{p \rightarrow \infty} \Pr{ \mathcal{E}_{p} }>0$ such that, conditionally on $T_{>p^2}\in \mathcal{E}_{p}$, we can couple a Boltzmann triangulation of the $p$-gon $T^{(p)}$ with $T_{>p^2}$ so that 
 $$ T^{(p)} \subset T_{>p^2},$$ in the sense that $T^{(p)}$ is a sub-triangulation of $T_{>p^2}$.
 \end{lemma}
 
 \begin{proof} To simplify notation, we use  $\mathbb{P}_{>p^{2}}$ for the probability relative to the law of a Boltzmann triangulation of the $1$-gon conditioned on having at least $p^{2}$ internal vertices.
 Denote by $ \mathcal{H}_{0} \subset \mathcal{H}_{1} \subset \cdots \subset \mathbf{t}$ the sequence of triangulations with holes obtained by performing the branching peeling by layers exploration on $\mathbf{t}$. If $ \mathbf{h}$ is a triangulation with holes of the $1$-gon with more than $p^2$ inner vertices such that $\mathbb{P}_{>p^{2}}( \mathcal{H}_{k} = \mathbf{h})>0$, a simple adaptation of Proposition \ref{prop:peelinggenerallaw} shows that  
under $\mathbb{P}_{>p^{2}}$ and conditionally on 
$ \{ \mathcal{H}_{k} = \mathbf{h}\}$, the triangulations filling-in the holes of $ \mathbf{h}$ inside $\mathbf{t}$ are independent Boltzmann triangulations with boundaries. We can then let $ \mathcal{E}_{p}$ be the event defined by
 $$ \mathcal{E}_{p} = \left\{ \exists k \geq 1 :| \mathcal{H}_{k}| \geq p^2 \mbox{ and }  \mathcal{H}_{k} \mbox{ has a hole of perimeter exactly }p \right\}.$$

On this event $ \mathcal{E}_{p}$, we denote by $\theta_{p}$ the smallest integer $k$ such that  $| \mathcal{H}_{k}| \geq p^2$ and $\mathcal{H}_{k}$  has a hole of perimeter exactly $p$. Clearly, $\theta_{p}$ is a $ (\mathcal{F}_{n})$ stopping time,  and it follows from an adaptation of Corollary \ref{cor:peelinggenerallaw} that, conditionally on $ \theta_{p} < \infty$, the triangulation filling-in a  hole of perimeter exactly $p$ of $ \mathcal{H}_{\theta_{p}}$ inside $ \mathbf{t}$ is distributed as a Boltzmann triangulation of the $p$-gon. The lemma is thus proved provided that we check that 
$ \liminf_{p \to \infty} \mathbb{P}_{>p^{2}}( \mathcal{E}_{p}) > 0$. Now consider the simpler event
$$ \mathcal{V}_{  \sqrt{p}} = \{ |B_{ \sqrt{p}}( \mathbf{t})| \geq p^2 \mbox{ and }  B_{ \sqrt{p}}( \mathbf{t})\mbox{ has a hole of perimeter larger than }  p\}.$$
We claim that  there exists $c>0$ such that
\begin{equation}
\label{eq:claim} \mathbb{P}_{>p^{2}}( \mathcal{E}_{p}) \geq c \,\mathbb{P}_{>p^{2}}( \mathcal{V}_{ \sqrt{p}}).\end{equation}
Indeed, by Proposition \ref{prop:scalingllcycle}, there exists $c>0$ such that for every $p$ sufficiently large, for every $p/2 \leq k \leq p$, the probability under  $\mathbb{P}^{(k)}$ that the perimeter of the locally largest cycle becomes greater than or equal to $p$ before becoming equal to $0$ is at least $c$.

Also, note that if  $p/2 \leq k \leq p$ and if the perimeter of the locally largest cycle of a Boltzmann triangulation of the $k$-gon takes a value greater that $p$ at a certain time, it has to take the value $p$ at some earler time because it can only increase by $1$.  Now, under $\mathbb{P}_{>p^{2}}$ and  on the event $\mathcal{V}_{ \sqrt{p}}$, select a cycle of $B_{ \sqrt{p}}( \mathbf{t})$ of perimeter larger than $ p$ and  follow the evolution of the perimeter of its associated locally largest cycle (as in Section \ref{sec:llc}).
Since this perimeter takes a value between $p/2$ and $p$ the first time it  becomes smaller than $p$, the previous discussions imply that {$\mathbb{P}_{>p^{2}}( \mathcal{E}_{ {p}}) \geq c \mathbb{P}_{>p^{2}}( \mathcal{V}_{\sqrt{p}})$}. This establishes \eqref{eq:claim}.

 It thus remains to prove that  $\liminf_{p \to \infty} \mathbb{P}_{>p^{2}}(\mathcal{V}_{ \sqrt{p}}) > 0$.
We proceed as in the proof of Lemma \ref{lem:volumedistance>n}: for every $c_{1}>1$, we have $ \mathbb{P}_{>p^2}(\mathcal{V}_{ \sqrt{p}}) \geq \mathbb{P}_{>p^2}(|\mathbf{t}| \geq c_{1} \, p^2) \inf_{n \geq c_{1} p^2} \mathbb{P}(T_{n} \in \mathcal{V}_ {\sqrt{p}})$, where we recall that $T_{n}$ is a uniform triangulation of the sphere with $n$ vertices. Since $\mathbb{P}_{>p^2}(|\mathbf{t}| \geq c_{1} \, p^2) \sim c_{1}^{-3/2}$ as $p \to \infty$, our goal is achieved if we can find $c_{1}>1$ large enough and $c_{2}>0$ so that 
 \begin{eqnarray} \label{eq:goalfin} \forall n \geq c_{1} p^2,\quad  \mathbb{P}(T_{n} \in \mathcal{V}_{ \sqrt{p}}) > c_{2}.  \end{eqnarray}
To prove this last claim, we use absolute continuity relations between finite triangulations and the UIPT: It follows from  \cite{CLGmodif} (see \cite[Proposition 7]{CLGplane} for similar estimates in the case of quadrangulations) that for any $   \varepsilon>0$ we can find $d( \varepsilon)>0$ such that for all $r \geq 0$ and all $n \geq d( \varepsilon)\, r^4$ we have 
$$ \mathrm{d_{TV}}( B_{r}(T_{n}), B_{r}(T_{\infty}))\leq \varepsilon,$$ where $ \mathrm{d_{TV}}$ is the total variation distance. 
On the other hand,  $\liminf_{p \rightarrow \infty} \P(|B_{ \sqrt{p}}(T_{\infty})| \geq p^2)>0$  and with probability bounded away from $0$ as $p \rightarrow \infty$, the cycle of $B_{ \sqrt{p}}( T_{\infty})$ separating the root from infinity has perimeter larger than $p$ (this can be seen by using results of \cite{CLGpeeling,CLGplane}, and we leave details to the reader). Hence,  $c'_{2}=  \liminf_{p \rightarrow \infty }\mathbb{P}(T_{\infty} \in \mathcal{V}_{ \sqrt{p}}) >0$. Therefore, taking $c_{1} = d( \varepsilon)$ with  $\varepsilon < c_{2}'/2$ and $ r= \sqrt{p}$ the previous discussion entails \eqref{eq:goalfin}. This completes the proof of the lemma.\end{proof}

It is now a simple matter to combine Lemma \ref{lem:volumedistance>n} and Lemma \ref{lem:coupling} to get that 

\begin{corollary}For any $ \varepsilon>0$, we have
 $$\lim_{ \delta \to 0} \sup_{n \geq 0}  \mathbb{P}\big( \exists x \in \mathrm{V}(T^{(p)}) : |B_{ \varepsilon  \sqrt{p}}(x)| \leq \delta p^2\big) =0.$$
 \end{corollary}

By combining Theorem \ref{thm:main} with Theorem \ref{thm:no tentacles}, we can relate the scaling limit of the total height of a Boltzmann triangulation with boundary with the extinction time $\inf\{ t \geq 0 : \mathbf{X}(t) = \varnothing\}$ of the growth-fragmentation process $ \mathbf{X}$ (which is known to be almost surely finite\cite[Corollary 3]{BeGF}). At this point of the paper, we leave the details to the reader.

\begin{corollary} We have the following convergence in distribution as $p \to \infty$
$$ p^{-1/2}\cdot \mathsf{Height}( T^{(p)}) \quad  \xrightarrow[p\to\infty]{(d)} \quad  \frac{\aone}{2 \tone}  \cdot \inf\{ t \geq 0 : \mathbf{X}(t) = \varnothing\}.$$
\end{corollary}

\subsection{Universality}

Let us end the article with a word on universality. We believe that most of the results established in this work are still valid for more general classes of random planar maps. The only influence of these classes under considerations are the constants  $ \aone$ and $\tone$: in particular Theorem \ref{thm:main} can be adapted to the case of triangulations of type II or quadrangulations (see \cite[Section 6]{CLGpeeling} for the values of the corresponding  new constants). More generally, we expect our results to hold for any regular critical Boltzmann maps as recently discussed in \cite{Bud15}. 

\bibliographystyle{siam}

\end{document}